\documentclass[reqno]{amsart}
\usepackage{amssymb}

\newtheorem{theorem}{Theorem}[section]
\newtheorem{proposition}[theorem]{Proposition}
\newtheorem{lemma}[theorem]{Lemma}
\newtheorem{corollary}[theorem]{Corollary}
\newtheorem{conjecture}[theorem]{Conjecture}

\theoremstyle{remark}
\newtheorem{remark}[theorem]{Remark}

\numberwithin{equation}{section}

\newcommand{\abs}[1]{\lvert#1\rvert}

\begin{document}

\title[Discrete Harmonic Analysis on a Weyl Alcove]
{Discrete Harmonic Analysis on a Weyl Alcove}

\author{J.F.  van Diejen}

\author{E. Emsiz}

\address{
Facultad de Matem\'aticas, Pontificia Universidad Cat\'olica de Chile,
Casilla 306, Correo 22, Santiago, Chile}
\email{diejen@mat.puc.cl, eemsiz@mat.puc.cl}

\subjclass[2000]{Primary: 43A90; Secondary: 20C08, 33D52}
\keywords{spherical functions, affine Hecke algebras, quantum integrable systems, Bethe Ansatz,
orthogonal polynomials, root systems}

\thanks{Work was supported in part by the {\em Fondo Nacional de Desarrollo
Cient\'{\i}fico y Tecnol\'ogico (FONDECYT)} Grants \# 1090118 and  \# 11100315,
and by the {\em Anillo ACT56 `Reticulados y Simetr\'{\i}as'}
financed by the  {\em Comisi\'on Nacional de Investigaci\'on
Cient\'{\i}fica y Tecnol\'ogica (CONICYT)}}


\begin{abstract}
We introduce a representation of the double affine Hecke algebra at the critical level $q=1$ in terms of difference-reflection operators
and use it to construct an explicit integrable discrete Laplacian on the Weyl alcove corresponding to an element in the center.
The Laplacian in question is to be viewed as an integrable discretization of the conventional Laplace operator on Euclidian space perturbed by a delta-potential supported on the reflection hyperplanes of the affine Weyl group.
The Bethe Ansatz method is employed to show that our discrete
Laplacian and its commuting integrals are diagonalized by a finite-dimensional basis of periodic Macdonald spherical functions.
\end{abstract}

\maketitle

\section{Introduction}\label{sec1}
It is well-known that the spectral problem for the bosonic one-dimensional $n$-particle Schr\"odinger operator with pairwise delta-potential interactions (the Lieb-Liniger model) can be solved by means of the coordinate Bethe Ansatz method \cite{lie-lin:exact,mcg:study,bre-zin:probleme}. In this approach, the eigenfunctions of the model are expressed as a linear combination of plane waves with expansion coefficients that are determined from the eigenvalue equation.
In the repulsive coupling-parameter regime, the proof of the orthogonality and completeness of these Bethe eigenfunctions hinges on the Plancherel formula for the Fourier transform \cite{gau:closure,gau:orthogonality}.
In the case of periodic boundary conditions, i.e. with particles moving on a circle rather than along a straight line, the discretization of the spectrum is described through a system of algebraic equations for the spectral parameter that are commonly referred to as
the Bethe Ansatz equations. In the repulsive regime the solutions of these equations are characterized as minima of a family of strictly convex Morse functions \cite{yan-yan:thermodynamics} and the orthogonality and completeness of the associated eigenfunctions were demonstrated by means of the algebraic Bethe Ansatz formalism \cite{dor:orthogonality}.
Previously, the algebraic Bethe Ansatz formalism had already been applied successfully in the computation of the
quadratic norms of the Bethe eigenfunctions \cite{kor:calculations} therewith confirming a remarkable norm formula conjectured by Gaudin \cite{gau:fonction}. For an overview of the extensive literature regarding the Bethe Ansatz solution of the one-dimensional bosonic $n$-particle model with pairwise delta-potential interactions the reader is referred to the standard texts
\cite{gau:fonction,kor-bog-ize:quantum,mat:many-body,tak:thermodynamics,sut:beautiful}.
Among important recent developments stand out: the discovery of integrable quantum models with particles on the line
interacting pairwise through generalized point interactions involving combinations of
$\delta$ and $\delta^\prime$ potentials
\cite{alb-fei-kur:integrability,hal-lan-pau:generalized};
ab initio approaches for the construction of the time propagator for the Lieb-Liniger model (i.e.
the fundamental solution of the time-dependent Schr\"odinger equation) avoiding a priori knowledge concerning the
Bethe Ansatz eigenfunctions (and their completeness)
\cite{tra-wid:dynamics,pro-spo:propagator,bor-cor:macdonald};
novel methods providing an in-depth analysis of the orthogonality, normalization, and completeness of the Bethe wave functions (i.e. the Plancherel formula) for the Lieb-Liniger model on the line in the much harder attractive regime \cite{dot:replica,pro-spo:propagator} simplifying previous discussions in \cite{oxf:hamiltonian} and \cite{hec-opd:yang}.

A fundamental observation made by Gaudin revealed that
Lieb and Liniger's solution method via a Bethe Ansatz applies to a much wider class of
Schr\"odinger operators involving delta-potentials that are supported on the reflection hyperplanes of crystallographic Coxeter groups (i.e. the finite and affine Weyl groups) \cite{gau:boundary}. From this perspective, the original one-dimensional bosonic $n$-particle model considered by Lieb and Liniger corresponds to the special situation in which the Coxeter group amounts to the symmetric group in the non-periodic case and to the affine symmetric group in the case of periodic boundary conditions. An elegant construction in terms of integral-reflection operators due to Gutkin and
Sutherland led to an explicit intertwining operator relating these
Schr\"odinger operators with delta-potentials to the free Laplacian, therewith demonstrating the integrability of the quantum models in question \cite{gut-sut:completely,gut:integrable}.
It turns out that in the case of a finite Weyl group the integral-reflection operators of Gutkin and Sutherland give rise to a representation of a degenerate affine Hecke algebra, a viewpoint that proved to be useful for determining
the Plancherel formula for the corresponding eigenfunction transform both in the repulsive and in the attractive coupling parameter regime \cite{hec-opd:yang}. In the case of affine Weyl groups the integral-reflection operators give rise to a representation of a trigonometric degenerate double affine Hecke algebra at critical level, and
in both cases (finite and affine) the Schr\"odinger operator and its commuting quantum integrals can be interpreted as central elements of this degenerate (double) affine Hecke algebra represented alternatively in terms of Dunkl-type differential-reflection operators (that arise as images of the directional derivatives under the Gutkin-Sutherland intertwining operator) \cite{ems-opd-sto:periodic}.
Even though the completeness of the Bethe Ansatz eigenfunctions for the delta-potential models
associated with affine Weyl groups is known in the repulsive parameter regime
\cite{ems:completeness}, a complete proof establishing their orthogonality and closed norm formulae
generalizing the corresponding results of Dorlas \cite{dor:orthogonality} and Korepin \cite{kor:calculations}, respectively, in the case of the affine symmetric group is not yet available for general affine Weyl groups. In \cite{ems:completeness} a conjectural
Gaudin-type formula for the quadratic norms of the Bethe Ansatz eigenfunctions was formulated and in
\cite{bus-die-maz:norm} the correctness of this conjecture was confirmed on a case by case basis for small Weyl groups (up to rank $3$).  The idea of this case by case test of the norm conjecture is to reduce its proof to the verification of an elementary (but tedious) algebraic identity, which is then readily checked by brute force for small Weyl groups with the aid of symbolic computer algebra.

In \cite{hec-opd:yang} it was pointed out that the Bethe eigenfunctions of Gaudin's delta-potential models
with finite Weyl group symmetry are degenerations of the Macdonald spherical
functions that arise in the harmonic analysis of symmetric spaces of simple Lie groups over $p$-adic fields \cite{mac:spherical1,mac:spherical2,nel-ram:kostka}. This connection was explored further in \cite{die:plancherel}, where it was shown that the Macdonald spherical functions themselves can be interpreted as the Bethe Ansatz eigenfunctions of a discrete Laplacian associated with the finite Weyl group. The transition from the Macdonald spherical functions to the Bethe eigenfunctions of the delta-potential model corresponds in this respect to a continuum limit in which the difference step size of the discrete Laplacian tends to zero.
For the (from the point of view of mathematical physics most relevant) special case in which the Weyl group amounts to the symmetric group, the (finite-dimensional) spectral problem of the corresponding discrete Laplacian with periodic boundary conditions (thus passing from permutation-symmetry to affine permutation-symmetry) was studied in
\cite{die:diagonalization}. The periodic discrete quantum model in question constitutes an integrable discretization
of the original quantum $n$-particle delta-potential model on the circle introduced by Lieb and Liniger and can be solved in a similar way via the Bethe Ansatz method.
This model turns out to provide a concrete quantum-mechanical description of a class of integrable systems originating from (the crystal basis for) the quantum affine algebra $U_q\widehat{{\mathfrak sl}}(n)$ \cite[Sec.~7]{kor:cylindric}.
The orthogonality of the corresponding Bethe Ansatz wave functions amounts to a novel system of finite-dimensional discrete orthogonality relations for the Hall-Littlewood polynomials (as the Macdonald spherical functions are usually referred to when the Weyl group equals the symmetric group) \cite{die:finite-dimensional}.

Recently, an explicit unitary representation of the affine Hecke algebra in terms of difference-reflection operators was introduced for which the action of the center is diagonal on the basis of Macdonald spherical functions \cite{die-ems:unitary}. The
discrete Laplacian with finite Weyl-group symmetry studied in \cite{die:plancherel} corresponds in this representation to a specific central element of the affine Hecke algebra associated with a (quasi-)minuscule weight. In other words, in the representation of \cite{die-ems:unitary}
the center of the affine Hecke algebra provides a complete algebra of commuting quantum integrals for the discrete Laplacian in \cite{die:plancherel}. The affine Hecke algebra representations in \cite{die-ems:unitary} are the discrete counterparts of the degenerate affine Hecke algebra representations in \cite{hec-opd:yang,ems-opd-sto:periodic} containing Gaudin's delta-potential models with finite Weyl group symmetry.
The aim of the present paper is to lift the construction in \cite{die-ems:unitary} from finite Weyl groups to affine Weyl groups so as to enable dealing with periodic boundary conditions. In the very special case of the affine Weyl group of rank {\em one} such a construction was presented recently in \cite{die-ems:discrete}.

Specifically,
we introduce a representation of the double affine Hecke algebra at the critical level $q=1$ in terms of difference-reflection operators
and use it to construct explicit discrete Laplacians on the Weyl alcove corresponding to central elements associated with the (quasi-)minuscule weights. (This algebra is not to be confused with the trigonometric degenerate double affine Hecke algebra represented by the Dunkl operators for the
trigonometric Calogero-Sutherland system diagonalized by the Heckman-Opdam multivariate Jacobi polynomials \cite{opd:lecture,ols-per:quantum},
or with the $q\to 1$ degenerate double affine Hecke algebra represented by the
Dunkl-Cherednik type operators for the confined rational Ruijsenaars system with hyperoctahedral symmetry
diagonalized by the multivariate Wilson polynomials \cite{gro:multivariable,die:multivariable,die:properties}.)
The Bethe Ansatz method is used to show that our Laplacians are diagonalized by a finite-dimensional basis of periodic Macdonald spherical functions. When the affine Weyl group is equal to the affine symmetric group, this reproduces the discrete
Laplacians with periodic boundary conditions introduced in \cite{die:diagonalization} together with the Bethe Ansatz solution.
For arbitrary affine Weyl groups, the present construction provides
a discrete counterpart of the representations of the trigonometric degenerate double affine Hecke algebra at critical level in \cite{ems-opd-sto:periodic}
governing Gaudin's delta-potential models with periodic boundary conditions.

It is well-known that the
Macdonald spherical functions are limiting cases of the Macdonald polynomials \cite{mac:symmetric,mac:orthogonal}.
It is therefore expected that the difference-reflection representation of the affine Hecke algebra
in \cite{die-ems:unitary} interpolates between the
Dunkl-type differential-reflection representation of the
degenerate affine Hecke algebra in \cite{ems-opd-sto:periodic} and
Cherednik's basic representation of the double affine Hecke algebra containing Macdonald's $q$-difference operators diagonalized by the Macdonald polynomials \cite{che:double,mac:affine}.
Based on Ruijsenaars' results in the rank-one case \cite{rui:relativistic}, it is moreover plausible that
the model of Lieb and Liniger with periodic boundary conditions is a limiting case of the elliptic quantum Ruijsenaars-Schneider system introduced in \cite{rui:complete}. This points towards the expectation
that in the periodic situation
the difference-reflection representation of the double affine Hecke algebra at $q=1$ presented here
should interpolate between
the Dunkl-type differential-reflection representation of the
degenerate double affine Hecke algebra in \cite{ems-opd-sto:periodic} and
Cherednik's general double affine Hecke algebra representation related to Ruijsenaars' commuting difference operators with
elliptic coefficients and their Weyl-group generalizations
\cite{che:difference}.

The paper is organized as follows. Section \ref{sec2} first recalls some necessary preliminaries concerning (double) affine Hecke algebras and their Weyl groups. Next a fundamental multiplication relation in the double affine Hecke algebra at the critical level is formulated describing the multiplicative action of generators on basis elements at $q=1$. Sections \ref{sec3} and \ref{sec4} generalize, respectively, the difference-reflection representation and the integral-reflection representation of the affine Hecke algebra introduced in \cite{die-ems:unitary} to the level of the double affine Hecke algebra at $q=1$. For the construction of the double affine extension of the difference-reflection representation, the multiplication relation from Section \ref{sec2} turns out to be instrumental.
Both our representations of the double affine Hecke algebra at critical level act in
the space of complex functions over the weight lattice and are proven to be equivalent with the aid of an explicit intertwining operator in Section \ref{sec5}.
The action of central elements in the double affine Hecke algebra at $q=1$
under the difference-reflection representation gives rise to an integrable system of discrete Laplace operators described explicitly in Section \ref{sec6}.
Section \ref{sec7} employs the Bethe Ansatz method to diagonalize the commuting Laplacians at issue by means of a basis of periodic Macdonald spherical functions spanning the finite-dimensional subspace of periodic Weyl-group invariant functions over the weight lattice.
The paper ends in Section \ref{sec8} by discussing the unitarity of the
difference-reflection representation with respect to a suitable Hilbert space structure and consequent orthogonality relations for the periodic Macdonald spherical functions.
Some technicalities regarding the construction of the difference-reflection representation and the study of its unitarity are relegated to Appendices \ref{appA} and \ref{appB} at the end of the paper.

\section{The double Affine Hecke Algebra at critical level}\label{sec2}
The primary objective of this section is to set up notation and recall some basic facts concerning affine Weyl groups and their double affine Hecke algebras at the critical level $q=1$.  A more complete discussion with proofs can be found in the standard literature \cite{bou:groupes,hum:reflection,mac:affine,che:double}. Our presentation of these preliminaries (cf. Subsections \ref{sec2s1} and \ref{sec2s2}) closely follows Macdonald's discussion in \cite{mac:affine}.
In addition, we describe a (for our purposes) essential multiplication formula in the double affine Hecke algebra
at critical level between the basis elements of the affine Hecke algebra and a system of generators for the group algebra of the weight lattice (cf. Subsections \ref{sec2s3} and \ref{sec2s4}).

\subsection{Affine Weyl group}\label{sec2s1}
Let $R_0\subset V$ be an irreducible reduced {\em crystallographic root system} spanning a real finite-dimensional vector space $V$ with inner product  $\langle \cdot ,\cdot \rangle$,
and let $R_0^\vee:=\{ \alpha^\vee\mid \alpha\in R_0\}$ denote the dual root system of coroots
$\alpha^\vee := 2\alpha/\langle \alpha ,\alpha\rangle$.
Given a (fixed) positive constant\footnote{In most of the standard literature the constant $c$ is normalized to have unit value; for our purposes, however, it is more convenient to regard $c$ instead as a positive scale parameter.}  $c$, the nontwisted  {\em affine root system} associated with $R_0$ is
$R:= R_0^\vee + \mathbb{Z}c$.  Here an affine root $a=\alpha^\vee +rc\in R$ is regarded
as an affine linear function $a:V\to \mathbb{R}$ of the form
$a(x)=\langle x,\alpha^\vee\rangle+rc $ ($x\in V$, $\alpha\in R_0$, $r\in\mathbb{Z}$), which gives rise to
an affine reflection $s_a:V\to V$ across the hyperplane $V_a:=\{ x\in V\mid a(x)=0\}$ given by
\begin{equation}\label{sa}
s_a(x):=x-a(x)\alpha .
\end{equation}
The {\em affine Weyl group} $W_R$ is defined as the (infinite) Coxeter group generated by the affine reflections $s_a$, $a\in R$. The isotropy subgroup fixing the origin, which is generated by the orthogonal reflections
$s_\alpha:=s_{\alpha^\vee}$, $\alpha\in R_0$, is referred to as the {\em (finite) Weyl group} $W_0$ (associated with $R_0$).  For $y\in V$, let us denote by $t_y:V\to V$ the translation determined by the action $t_y(x):=x+y$.  Since $s_{\alpha^\vee}s_{\alpha^\vee+rc}=t_{rc\alpha}$, the elements of $W_R$ can be written as $vt_{c\lambda}$ with $v\in W_0$ and $\lambda$ in the {\em root lattice} $Q:=\text{Span}_{\mathbb{Z}}(R_0)$, i.e. $W_R= W_0 \ltimes t(cQ)$ (where $t(cQ)$ denotes the group of translations $t_{c\lambda}$, $\lambda\in Q$). By extending the lattice of translations from the root lattice $Q$ to the {\em weight lattice} $P:=\{\lambda\in V\mid \langle\lambda,\alpha^\vee\rangle\in\mathbb{Z},\,\forall\alpha\in R_0\}$, one arrives at the {\em extended affine Weyl group}
$W= W_0\ltimes t(cP)$ consisting of group elements of the form $vt_{c\lambda}$ with $v\in W_0$ and $\lambda\in P$.

A (fixed) choice of positive roots $R_0^+$, with a simple basis $\alpha_1,\ldots ,\alpha_n$, determines a set of positive affine roots
$R^+:=R_0^{\vee,+}\cup (R_0^\vee +\mathbb{N}c)$ generated by a
basis of simple affine roots  $a_0,\ldots ,a_n$ of the form
$a_0:=\alpha_0^\vee+c$ and $a_j:=\alpha_j^\vee$ for $j=1,\ldots ,n$. Here  $n$ denotes the rank  of $R_0$ ($=\text{dim}\, V$) and $\alpha_0:=-\vartheta$ with
$\vartheta\in R_0^+$ being the highest {\em short} root (so $\vartheta^\vee$ is the highest coroot of $R_0^\vee$).
The action of $w\in W$ on $V$ induces a dual action on the space $\mathcal{C}(V)$ of functions $f:V\to\mathbb{C}$ given by
\begin{equation}\label{W-action}
(wf)(x):=f(w^{-1}x)\qquad (w\in W,\, f\in \mathcal{C}(V),\, x\in V).
\end{equation}
 The affine root system $R\subset \mathcal{C}(V)$ is stable with respect to this dual action.
For $w\in W$ the length $\ell (w)$ is now defined as the cardinality $ \# R(w)$ of the (finite) set
\begin{equation}\label{Rw}
R(w):=R^+\cap w^{-1} (R^-),
\end{equation}
with $R^-:=-R^+=R\setminus R^+$. For $w\in W$ and $j\in \{ 0,\ldots ,n\}$ one has that
\begin{subequations}
\begin{align}
\ell (s_jw)=&\ell (w)+\text{sign}(w^{-1}a_j) , \label{ls1}\\
\ell (ws_j)=&\ell (w)+\text{sign}(w a_j) , \label{ls2}
\end{align}
\end{subequations}
where we have employed the short-hand  $s_j:=s_{a_j}$ (and with
the sign function on $R$ being defined as
$\text{sign}(a)=1$ for $a\in R^+$ and $\text{sign}(a)=-1$ for $a\in R^-$). It follows inductively
from Eqs. \eqref{ls1}, \eqref{ls2} that
any $w\in W$ can be decomposed (nonuniquely) as
\begin{equation}\label{red-exp}
w=u s_{j_1}\cdots s_{j_\ell},
\end{equation}
with $j_1,\ldots,j_\ell\in \{ 0,\ldots ,n\}$, $\ell=\ell (w)$, and $u\in\Omega:=\{ w\in W \mid \ell (w)=0\}$.
Such a decomposition (with $\ell=\ell(w)$) is called a {\em reduced expression} for $w$.
Given a reduced expression \eqref{red-exp}, one has that
\begin{subequations}
\begin{equation}
R(w)=\{ b_1,\ldots ,b_\ell\} \quad \text{and}\quad w=u s_{b_\ell}\cdots s_{b_1},
\end{equation}
where
\begin{equation}
 b_k:= s_{j_\ell}\cdots s_{j_{k+1}} a_{j_k}\quad  (k=1,\ldots ,\ell)
\end{equation}
\end{subequations}
(so $b_\ell=a_{j_\ell}$).

Since  $R^+$ is stable with respect to the action of $\Omega$ (as its group elements have length zero), it is clear that the $u\in\Omega$ act by permuting
the elements of the simple basis  $a_0,\ldots ,a_n$.
More specifically,
one has that  $ua_j=a_{u_j}$ (and thus $us_ju^{-1}=s_{u_j}$),  where $j\to u_j$ encodes the corresponding permutation of the indices
 $j=0,\ldots ,n$. The upshot is that $W=\Omega\ltimes W_R$ with $W_R$ being normal in $W$,
 whence $\Omega$ is a finite abelian subgroup: $\Omega\cong  W/W_R \cong P/Q$.
 The extended affine Weyl group $W$ can now be presented as the group generated by the commuting
elements from
$\Omega$ and the simple reflections $s_0,\ldots s_n$, subject to the relations
\begin{subequations}
\begin{equation}\label{W-rel1}
u s_j u^{-1}=s_{u_j}    \quad u\in \Omega ,\ j=0,\ldots ,n,
\end{equation}
and
\begin{equation}\label{W-rel2}
(s_js_k)^{m_{jk}}=1 \quad  j,k =0,\ldots ,n.
\end{equation}
\end{subequations}
Here $m_{jk}=1$ if $j=k$ and
$m_{jk}=\pi/\alpha_{jk}$ ($\in \{   2,3,4,6 \}$) with $\alpha_{jk}$ denoting the angle between $V_{a_j}$ and $V_{a_k}$ if $j\neq k$
(and the provision that for $n=1$ the order
$m_{10}=m_{01}=\infty$).

\begin{remark}
While here for technical simplicity it was assumed from the start that our underlying affine root system $R$ is nontwisted and reduced,
it is expected that---at the price of a fairly amount of additional notational overhead---in principle much of the construction below can be generalized so as to incorporate the cases of twisted and even nonreduced affine root systems.
\end{remark}

\subsection{Double affine Hecke algebra at critical level}\label{sec2s2}
A function $\tau:W\to \mathbb{C}\setminus \{ 0\}$ satisfying that $\text{(i)}$
$\tau_{w\tilde{w}}= \tau_w \tau_{\tilde{w}} $ if $\ell(w\tilde{w} )=\ell(w)+\ell(\tilde{w} ) $
and $\text{(ii)}$  $\tau_w=1$ if $ \ell(w)=0$ is called a {\em length multiplicative function}.
Such a function is completely determined by its values on the simple reflections and
consistency demands that
 $\tau_{s_j}=\tau_{s_k}$ if  $s_j$ and $s_k$ are conjugate in $W$ (where $j,k= 0,\ldots ,n$).
 Let us denote by $\tau_s$ and $\tau_l$ the values of $\tau$ on the reflections in the simple roots $a_j$ with $\alpha_j$ {\em short} and {\em long}, respectively (with the convention that all finite roots are short when $R_0$ is simply-laced).
The value of $\tau_w$ is thus completely determined by $\tau_s$ and the number of
reflections $s_{j_k}$ with $\alpha_{j_k}$ short and $\tau_l$ and the number of
reflections $s_{j_k}$ with $\alpha_{j_k}$ long appearing in a reduced expression \eqref{red-exp}.
Following customary habits, the {\em multiplicity function} will also be denoted by $\tau$.
This is a function $\tau: R\to \mathbb{C}\setminus \{ 0\}$ satisfying that $\tau_{wa}=\tau_{a}$ for all $w\in W$ and
$a\in R$, which means that its value is constant on the $W$-orbits of affine roots.
Upon compatibilizing both functions on the simple elements such that
 $\tau_j:=\tau_{a_j}=\tau_{s_j}$  for $j=0,\ldots ,n$, one can compute the
values of the length multiplicative function by means of the multiplicity function via the well-known formula
\begin{equation}\label{reconstruct}
\tau_w=\prod_{a\in R(w) } \tau_a .
\end{equation}
Throughout the paper we will always assume
that neither $\tau_s$ nor $\tau_\ell$ is a root of unity (unless explicitly stated otherwise).

The {\em double affine Hecke algebra at critical level}  $\mathbb{H}$ is now defined as the complex unital associative algebra
spanned by the basis
\begin{equation}\label{H-basis}
X^\lambda T_w, \qquad \lambda\in P,\ w\in W,
\end{equation}
with $X^\lambda=T_w=1$ if $\lambda=0$ and $w=1$,
such that the following multiplication relations are satisfied
(for all $u\in \Omega$, $j=0,\ldots ,n$, $w\in W$, and $\lambda,\mu\in P$):
\begin{subequations}
\begin{equation}\label{Tu}
T_uT_w=T_{uw},\qquad T_u X^\lambda = X^{u^\prime \lambda} T_u
\end{equation}
\begin{equation} \label{TjTw}
T_jT_w = T_{s_j  w} +\chi(w^{-1}a_j)(\tau_j-\tau_j^{-1})T_w
\end{equation}
\begin{equation}  \label{TjX}
T_jX^\lambda = X^{s_j^\prime \lambda }T_j+(\tau_j-\tau_j^{-1})\frac{X^\lambda-X^{s'_j\lambda}}{1-X^{-\alpha_j}} \end{equation}
\begin{equation}\label{X}
X^\lambda X^\mu = X^{\lambda+ \mu} .
\end{equation}
\end{subequations}
Here $\chi :R\to \{ 0,1\}$ denotes the characteristic function of $R^-$ and we have employed the short-hand
notations $\tau_j:=\tau_{s_j}$ and  $T_j:=T_{s_j}$. Furthermore,
the prime symbols refer to the derivative in the sense of calculus projecting $W$ onto $W_0$ and $R$ onto $R_0^\vee$. More specifically, for $v\in W_0$ and $\lambda\in P$ one has that
 $(vt_{c\lambda})^\prime=v$, whence for $a=\alpha^\vee +rc\in R$ this means in particular that
 $s_{a}^\prime=s_{a^\prime}$ with $a^\prime=(\alpha^\vee +rc)^\prime=\alpha^\vee $.
Alternatively, the relation in Eq. \eqref{TjTw} can be replaced by the equivalent relation (cf Eqs. \eqref{ls1}, \eqref{ls2})
\begin{equation} \label{TwTj}
T_wT_j = T_{ws_j} +\chi(w a_j)(\tau_j-\tau_j^{-1})T_w .
\end{equation}
 Throughout, fractions
of the type in Eq. \eqref{TjX}  involving (double) affine Hecke algebra elements or their images under a representation are to be interpreted in terms of their terminating geometric series:
 \begin{eqnarray}\label{term-geom}
\lefteqn{ \frac{X^\lambda-X^{s'_j\lambda}}{1-X^{-\alpha_j}}: =} && \\
&&   \begin{cases}
 X^\lambda+X^{\lambda-\alpha_j}+\cdots +X^{s_j^\prime\lambda+ \alpha_j}&\text{if}\ \langle \lambda,\alpha_j^\vee\rangle>0 , \\
0  &\text{if}\ \langle \lambda,\alpha_j^\vee\rangle =0 , \\
 -X^{\lambda+\alpha_j}-X^{\lambda+2\alpha_j}-\cdots -X^{s^\prime_j\lambda }&\text{if}\ \langle \lambda,\alpha_j^\vee\rangle<0 .
 \end{cases}\nonumber
 \end{eqnarray}

The double affine Hecke algebra $\mathbb{H}$ contains several important subalgebras.
The subalgebras $H_0\subset H_R\subseteq H\subset\mathbb{H}$ spanned by $T_w$ with
$w\in W_0$, $w\in W_R$, and $w\in W$, respectively,
are referred to as
the  {\em finite Hecke algebra}, the  {\em  affine Hecke algebra}, and the  {\em extended affine Hecke algebra} (i.e. the Hecke algebras of the finite Weyl group $W_0$, the affine Weyl group $W_R$, and the extended affine Weyl group $W$).
The defining relation in Eq. \eqref{TjTw} gives rise to the following multiplication rule for the basis elements of these subalgebras:
\begin{equation}
T_{w} T_{\tilde{w}}=T_{w\tilde{w}}\quad \text{if}\ \ell(w\tilde{w})=\ell(w)+\ell(\tilde{w})\qquad (w,\tilde{w}\in W).
\end{equation}
A second way in which an extended affine Hecke algebra appears inside $\mathbb{H}$ (explaining the name double affine Hecke algebra) is via its Bernstein-Zelevinsky presentation as the subalgebra
with basis $X^\lambda T_v$, $\lambda\in P$, $v\in W_0$.
This affine Hecke algebra contains a large commutative subalgebra
$\mathbb{C}[X]\subset \mathbb{H}$ spanned by $X^\lambda$, $\lambda\in P$ that is isomorphic to the group algebra
of the weight lattice. It is endowed with a natural $W_0$-action inherited from the action of the Weyl group on $ V$ restricted
to the stable lattice $P$:
$v(X^\lambda):=X^{v\lambda}$ ($v\in W_0$, $\lambda\in P$).
It follows from the relations in Eqs. \eqref{Tu}, \eqref{TjX}
that its $W_0$-invariant subalgebra
$\mathbb{C}[X]^{W_0}$  spanned by  the $W_0$-invariants
\begin{subequations}
\begin{equation}
m_\lambda (X):=\sum_{\mu\in W_0\lambda} X^\mu ,
\end{equation}
with $\lambda$ in the cone of dominant weights
\begin{equation}
P^+:=\{ \lambda\in P\mid \langle \lambda ,\alpha^\vee\rangle \geq 0,\, \forall \alpha\in R_0^+\} ,
\end{equation}
\end{subequations}
is contained in the center $\mathcal{Z}(\mathbb{H})$ of the double affine Hecke algebra at critical level:
\begin{equation}\label{center}
\mathbb{C}[X]^{W_0}\subset \mathcal{Z}(\mathbb{H}).
\end{equation}

It is immediate from the first relation in \eqref{Tu} and relation \eqref{TjTw} that for any  $w\in W$
the reduced expression in \eqref{red-exp} gives rise to the following decomposition
of the corresponding basis element $T_w$:
\begin{equation}\label{Tw-red-exp}
 T_w= T_u T_{j_1}\cdots T_{j_\ell} .
\end{equation}
In other words, the elements  $T_u$ ($u\in\Omega$) and $T_j$ ($j=0,\ldots ,n$)
generate the extended affine Hecke algebra $H$. To describe an appropriate system of generators for the group algebra
$\mathbb{C}[X]$, let us recall that a nonzero weight $\omega\in P$ is called {\em minuscule} if $0\leq \langle \omega,\alpha^\vee\rangle\leq 1$ for all $\alpha\in R_0^+$ and that it is called {\em quasi-minuscule} if
$0\leq \langle \omega,\alpha^\vee\rangle\leq 2$ for all $\alpha\in R_0^+$ with the upperbound $2$ being realized only {\em once}. Together with the zero weight the minuscule weights constitute a complete set of represententatives for $P/Q$, and the quasi-minuscule weight is unique and equal to the highest short root $\vartheta\in R_0$. Since the orbit of short roots $W_0\vartheta$ generates the root lattice $Q$, together with the minuscule weights they form a generating set for the weight lattice $P$.
Let
\begin{equation}\label{P-theta}
P_\vartheta:= \{ \nu\in P\mid  \  | \langle \nu,\alpha^\vee\rangle | \leq 1 ,\ \forall \alpha\in R_0\setminus \{ \nu\} \, \}
\end{equation}
(i.e., $P_\vartheta$ is the smallest saturated set containing all (quasi-)minuscule weights).
Then the commuting elements  $X^\nu$ ($\nu\in P_\vartheta$) generate  $\mathbb{C}[X]$ and
combined  with the above elements generating  $H$
this provides a (complete) system of generators for the double affine Hecke algebra
$\mathbb{H}$.

More specifically, the double affine Hecke algebra at critical level $\mathbb{H}$ can be
presented as the complex unital associative algebra generated by
$T_u$ ($u\in \Omega$), $T_j$  ($j=0,\ldots ,n$) and the commuting elements $X^\nu$ ($\nu\in P_\vartheta$)
(with $T_u=X^\nu=1$ if $u=1$ and $\nu =0$),
subject to the relations
\begin{subequations}
\begin{equation}\label{Tu-rel}
T_uT_{\tilde{u}}=T_{u\tilde{u}} \quad\text{and}\quad T_uT_j=T_{u_j}T_u  \qquad
(u,\tilde{u}\in\Omega ,\, 0\leq j\leq n),
\end{equation}
\begin{equation}\label{quadratic-rel}
(T_j-\tau_j)(T_j+\tau_j^{-1})=0\qquad (0\leq j\leq n),
\end{equation}
\begin{equation}\label{braid-rel}
\underbrace{T_jT_kT_j\cdots}_{m_{jk}\ {\rm factors}}
=\underbrace{T_kT_jT_k\cdots}_{m_{jk}\ {\rm factors}}\qquad (0\leq j\neq k\leq n),
\end{equation}
\begin{equation}\label{TuX-rel}
T_u X^\nu = X^{u^\prime\nu} T_u \qquad (u\in \Omega,\, \nu\in P_\vartheta),
\end{equation}
\begin{equation}  \label{TjX-rel}
T_jX^\nu = X^{s_j^\prime \nu }T_j+(\tau_j-\tau_j^{-1})\frac{X^\nu-X^{s'_j\nu}}{1-X^{-\alpha_j}} \qquad
(0\leq j\leq n,\,  \nu\in P_\vartheta)
\end{equation}
\begin{equation}\label{X-rel}
X^\nu X^{\tilde{\nu}} = X^{\nu+\tilde{ \nu}} \qquad ( \nu,\tilde{\nu}\in P_\vartheta\ \text{such that}\
 \nu+\tilde{\nu}\in P_\vartheta )
\end{equation}
\end{subequations}
(where it is understood that the number of factors $m_{jk}$ on both sides of the braid relation for $H$ in
\eqref{braid-rel} is the same as the order of the corresponding braid relation for $W$ in \eqref{W-rel2}).
Notice that
the second term on the RHS of Eq. \eqref{TjX-rel} amounts to  a linear expression in the generators because of Eq. \eqref{term-geom}
and the property that $P_\vartheta$ \eqref{P-theta} is saturated (so the $\alpha_j$-string connecting $\nu$ and $s^\prime_j\nu$ belongs to $P_\vartheta$).
The characterization of the group algebra $\mathbb{C}[X]$ as the
algebra generated by the commuting generators $X^\nu$, $\nu\in P_\vartheta$  \eqref{P-theta}  subject to the relations in Eq. \eqref{X-rel} hinges in turn on the fact that $P_\vartheta$ contains a linear basis for the weight lattice with respect to which $P_\vartheta$ is path-connected in the following sense:
for any $\nu\in P_\vartheta$ there exists a path $0\to\nu^{(1)}\to\cdots \to\nu^{(\ell )}=\nu$ of weights in $P_\vartheta$ such that two subsequent weights in the path differ by an element of the basis at issue. In other words, the commuting algebra generated by the $X^\nu$, $\nu\in P_\vartheta$  \eqref{P-theta}  subject to the relations in Eq. \eqref{X-rel} is freely generated by the generators corresponding to the basis elements and thus isomorphic to
$\mathbb{C}[X]$.
For classical root systems it is not difficult to single out a basis for $P$ with the above properties
by inspecting the tables in Bourbaki, whereas
for $|\Omega |=1$ (so $P_\vartheta=W_0\vartheta\cup \{ 0\}$) one may pick
any choice of simple basis for the simply-laced subsystem $W_0\vartheta$.  For the remaining exceptional types (viz. $E_6$ and $E_7$) a suitable basis is readily found with the aid of a small computer calculation.

\subsection{Multiplicative action of $\mathbb{C}[X]$ on $H$}\label{sec2s3}
The first three defining relations \eqref{Tu}--\eqref{TjX} for $\mathbb{H}$
encode the multiplicative action of (the generators of) $H$ on (the bases of)
$H$ and $\mathbb{C}[X]$; the last defining relation \eqref{X} describes
the commutative multiplication within  $\mathbb{C}[X]$. To complete the explicit description of the multiplicative
structure of $\mathbb{H}$
we provide a helpful formula describing the multiplicative action of (the generators of) $\mathbb{C}[X]$ on (the basis
of) $H$.

Let us recall for this purpose that the {\em Bruhat order} on $W_R$ is defined as the transitive closure of the
relations $w<ws_a\Leftrightarrow\ell (w)<\ell (ws_a)$ for $w\in W_R$ and $a\in R^+$ (cf. e.g. Ref. \cite[Sec. 5.9]{hum:reflection}). This partial order is extended to $W$ such that
elements belonging to different $W/W_R$ cosets are not comparable \cite[Sec. 2.3]{mac:affine}:
\begin{equation}
\tilde{u}\tilde{w}\leq uw \Longleftrightarrow \tilde{u}=u\ \text{and}\ \tilde{w}\leq w\qquad (\tilde{u},u\in\Omega,\, \tilde{w},w\in W_R).
\end{equation}
Some key properties of the Bruhat order that will be used frequently are (for $\tilde{w},w\in W$ and  $a\in R^+$):
\begin{subequations}
\begin{equation}\label{br-p1}
\tilde{w}\leq w\Longleftrightarrow \tilde{w}^{-1}\leq w^{-1} ,
\end{equation}
\begin{equation}\label{br-p2}
\tilde{w}\leq w \Longrightarrow  \tilde{w}s_a\leq ws_a \ \text{or}\   \tilde{w}s_a \leq w ,
\end{equation}
and
\begin{equation}\label{br-p3}
ws_a<w \Longleftrightarrow a\in R(w).
\end{equation}
\end{subequations}
Moreover, given a reduced decomposition \eqref{red-exp} for $w$, all elements of $W$ smaller than $w$ in the
Bruhat order are the ones
obtained by deleting simple reflections from this reduced expression.

\begin{theorem}\label{multiplication:thm}
For any $w\in W$ and $\nu\in P_\vartheta^\star:=P_\vartheta\setminus \{0\}$ \eqref{P-theta}, the product $T_wX^\nu$ expands in the basis \eqref{H-basis} as
\begin{subequations}
  \begin{equation}    \label{TwXnu}
T_w X^\nu=X^{w'\nu}T_w+\sum_{\substack{v\in W\\ v< w}} (B^\nu_{v,w}+\sum_{\eta\in W_0\nu}A^{\eta,\nu}_{v,w}X^\eta)T_v ,
  \end{equation}
with expansion coefficients $A_{v,w}^{\eta,\nu}$ and $B_{v,w}^\nu$ determined by the
following recurrence relations for $j\in \{0,\dots, n\}$ such that $s_jw<w$:
\begin{equation}  \label{recrel1}
  A^{\eta,\nu}_{v,w}=
\begin{cases}
  A^{s^\prime \eta, \nu}_{sv,sw}+q(   A^{\eta,\nu}_{v,sw} -   A^{s'\eta ,\nu}_{v,sw} )  & \text{if $sv>v$ and $\langle \eta,\alpha^\vee\rangle >0$},\\
 A^{s'\eta ,\nu}_{sv,sw} &    \text{if $sv>v$ and $\langle \eta,\alpha^\vee\rangle \le 0$},\\
  A^{s'\eta,\nu}_{sv,sw}+q  A^{\eta,\nu}_{v,sw}  & \text{if $sv<v$ and $\langle \eta,\alpha^\vee\rangle \ge 0$},\\
  A^{s'\eta,\nu}_{sv,sw}+q  A^{s'\eta,\nu}_{v,sw}  & \text{if $sv<v$ and $\langle \eta,\alpha^\vee\rangle <0$},
\end{cases}
\end{equation}
where we have employed the short-hand notation $s=s_j$, $\alpha=\alpha_j$ and $q=q_j:=\tau_j-\tau_j^{-1}$, and
\begin{equation} \label{recrel2}
B_{v,w}^\nu=q\sum_{\substack{\eta\in W_0\nu\\ \langle \eta,\alpha^\vee\rangle =2}} (A^{\eta ,\nu}_{v,sw}-A^{s'\eta,\nu}_{v,sw})+
\begin{cases}
   B_{sv,sw}^\nu & \text{if $sv>v$},\\
   B_{sv,sw}^\nu+q  B_{v,sw}^\nu  & \text{if $sv<v$},
\end{cases}
\end{equation}
subject to the boundary conditions
\begin{equation}\label{bc1}
A^{\eta,\nu}_{v,v}=
\begin{cases}
 1 & \text{if $\eta=v'\nu$},\\
  0 & \text{if $\eta\neq v'\nu$},
\end{cases}
\qquad B_{v,v}^\nu=0,
\end{equation}
and
 \begin{equation}\label{bc2}
A^{\eta,\nu}_{v,w}=B_{v,w}^\nu=0\quad\text{if}\quad v\not\leq w.
\end{equation}

Moreover, the expansion coefficients  in question
comply the following invariance with respect to the action of $\Omega$ on the indices:
\begin{equation}  \label{recrel0}
A^{u'\eta,\nu}_{uv,uw}=A^{\eta,\nu}_{v,w}, \quad B_{uv,uw}^\nu=B_{v,w}^\nu \qquad (u\in \Omega).
\end{equation}
\end{subequations}
\end{theorem}

The proof of this theorem is by induction on the length of $w$; the details
are relegated to the end of this section.

For $w\in W_0$, Theorem \ref{multiplication:thm} amounts
to a multiplication formula in the Bernstein-Zelevinsky presentation of the extended affine Hecke algebra
with basis $X^\lambda T_v$, $\lambda\in P$, $v\in W_0$.
The expansion for $T_w X^\nu$ simplifies in this situation if $\nu\in P_\vartheta^\star$ is
dominant:

\begin{corollary}\label{multiplication:cor}
For $\omega\in P$ (quasi-)minuscule and $w\in W_0$, one has that
\begin{subequations}
 \begin{equation}   \label{highercross3}
T_w X^\omega=  X^{w\omega}T_w+ \sum_{\substack{v\in W_0\\ v< w}} (B_{v,w}^\omega+ A_{v,w}^\omega X^{v\omega})T_v,
  \end{equation}
where the coefficients $A_{v,w}^\omega$ and $B_{v,w}^\omega$ are determined by the recurrence relations
(for $s=s_j$, $\alpha=\alpha_j$, $q=q_j$ with $j\in \{ 1,\ldots ,n\}$ such that $s_jw<w$):
\begin{eqnarray}
\lefteqn{  A_{v,w}^\omega=} && \\
&&
\begin{cases}
 A_{sv,s w}^\omega &    \text{if $sv>v$ and $\langle v\omega,\alpha^\vee\rangle=0$ or $sv<v$ and $\langle v\omega,\alpha^\vee\rangle<0$}\\
  A_{sv,sw}^\omega+q A_{v,sw}^\omega  & \text{if $sv>v$ and $\langle v\omega,\alpha^\vee\rangle > 0$ or $sv<v$ and $\langle v\omega,\alpha^\vee\rangle=0$}
\end{cases} ,\nonumber
\end{eqnarray}
and
\begin{equation}
B_{v,w}^\omega=
\begin{cases}
   B_{sv,sw}^\omega & \text{if $sv>v$}\\
   B_{sv,sw}^\omega+q  B_{v,sw}^\omega  & \text{if $sv<v$}
\end{cases}
\quad +\ \begin{cases}
\ \ qA_{v,sw}^\omega & \text{if $v^{-1}\alpha=\omega$} \\
-qA_{v,sw}^\omega& \text{if $v^{-1}\alpha=-\omega$}\\
\ \ \ 0 & \text{otherwise}
\end{cases},
\end{equation}
\end{subequations}
subject to the boundary conditions $A_{v,v}^\omega=1$, $B_{v,v}^\omega=0$, and $A_{v,w}^\omega=B_{v,w}^\omega=0$ if $v\not\leq w$.
\end{corollary}
Corollary \ref{multiplication:cor} is immediate from Theorem \ref{multiplication:thm} and the following lemma
upon setting $A^{\omega}_{v,w}:= A^{v\omega ,\omega}_{v,w} $.

\begin{lemma}
For $\omega\in P$ (quasi-)minuscule and $v,w\in W_0$ with $v<w$,
the expansion coefficients $A^{\eta ,\omega}_{v,w}$, $\eta\in W_0\omega$, vanish if $\eta\neq v\omega$.
\end{lemma}

\begin{proof}
 The proof is by induction on $\ell (w)$ starting from the elementary case that $\ell (w)=1$.
 For $w=s_j$ ($j=1,\ldots ,n$) it is readily deduced from Theorem \ref{multiplication:thm}
 (or alternatively from Eq. \eqref{TjX-rel} and the fact that $\omega$ is (quasi-)minuscule) that
 $$ T_jX^\omega-X^{s_j\omega} T_j =
 \begin{cases} 0&\text{if}\ s_j\omega =\omega  ,\\ q X^\omega &\text{if}\ s_j\omega \neq\omega ,
 \end{cases} $$
which confirms the stated vanishing property of $A^{\eta ,\omega}_{v,w}$ if $\ell (w)=1$.
When $\ell (w)> 1$ we pick $s=s_j$ such that  $sw<w$. For $\eta\neq v\omega$ the recurrence relations in Eq. \eqref{recrel1} combined with the induction hypothesis entail that
\begin{equation}\label{A-dom-rec}
A^{\eta ,\omega}_{v,w}=
\begin{cases}
-  q A^{s\eta,\omega}_{v,sw}   & \text{if $sv>v$ and $\langle \eta,\alpha^\vee\rangle >0$},\\
0 &    \text{if $sv>v$ and $\langle \eta,\alpha^\vee\rangle \le 0$},\\
 0 & \text{if $sv<v$ and $\langle \eta,\alpha^\vee\rangle \ge 0$},\\
 q  A^{s\eta,\omega}_{v,sw}  & \text{if $sv<v$ and $\langle \eta,\alpha^\vee\rangle <0$},
\end{cases}
\end{equation}
where $A^{s\eta,\omega}_{v,sw}=0$ if $s\eta\neq v\omega$. Since the equality $s\eta= v\omega$ implies that
 $\langle  \eta, \alpha^\vee\rangle=-\langle v^{-1}s\eta,v^{-1}\alpha\rangle=-\langle \omega, v^{-1}\alpha\rangle$, it is moreover clear from Eq. \eqref{ls1} that in this situation
$\langle  \eta, \alpha^\vee\rangle\le 0$ if $sv>v$ and $\langle  \eta, \alpha^\vee\rangle\ge 0$ if $sv<v$.
In other words, the coefficient $A^{s\eta, \omega}_{v,sw}$ in Eq. \eqref{A-dom-rec} vanishes (and thus so does $A^{\eta,\omega}_{v,w}$).
\end{proof}

\begin{remark}\label{rem2.5}
The recurrence relations in Theorem \ref{multiplication:thm} and Corollary \ref{multiplication:cor} are reminiscent of the recurrence for the Kazhdan-Lusztig $R$-polynomials in the expansion of affine Hecke algebra elements of the form $T_{w^{-1}}^{-1}$ ($w\in W_R$)
in terms of the standard basis $T_v$, $v\in W_R$ \cite[Ch. 7]{hum:reflection}.
In fact, it is manifest from our recurrence relations that the coefficients $A_{v,w}^{\eta ,\nu}$,
$B_{v,w}^{\nu}$ and $A_{v,w}^\omega$,
$B_{v,w}^\omega$ are themselves polynomials in the indeterminates $q_j=\tau_j-\tau_j^{-1}$ with integral coefficients.
The boundary conditions in Eqs. \eqref{bc1}, \eqref{bc2} imply that these expansion coefficients  are actually quite sparse. For instance, it is immediate from the recurrence that
all coefficients $B_{v,w}^{\nu}$ vanish unless $\nu\in W_0\vartheta$. Similarly, in Corollary \ref{multiplication:cor} the coefficients $B_{v,w}^\omega$ vanish if $\omega$ is minuscule.
\end{remark}

\begin{remark}
For $\ell (w)=0$ and $\ell (w)=1$ the formula in Theorem \ref{multiplication:thm} reproduces the relations
in Eqs. \eqref{TuX-rel} and \eqref{TjX-rel} (in the form of Eq.  \eqref{com-rel} below), respectively.
In principle the expansion coefficients $A^{\eta ,\nu}_{v,w}$, $B_{v,w}^\nu$ ($v<w$) can be computed inductively from the recurrence relations in explicit form for any $w\in W$, but the expressions soon tend to become quite cumbersome.
In the very special case that $R_0$ is the root system of rank {\em one} and $\nu=\pm \omega$, with $\omega$ denoting the minuscule weight, then it is not hard arrive in this manner
at the following explicit expansion formula
 \cite{die-ems:discrete}:
\begin{eqnarray*}
\lefteqn{T_w X^{\varepsilon}= X^{\varepsilon (-1)^{\ell(w)+r}}T_w} &&\\
&& +\varepsilon(\ell(ws)-\ell(w))X^{(\ell(ws)-\ell(w)) (-1)^{\ell(w)+r}}
\sum_{v\in W, v<w} A(\ell(w)-\ell(v))T_v ,
\end{eqnarray*}
where $\varepsilon\in \{ 1 ,-1\}$, $X:=X^\omega$,  $s:=s_1$,
$$
A(k):=\frac{1-\tau^2}{1+\tau^2}(\tau^{-k}+(-1)^{k+1}\tau^k) , \qquad  r:=\begin{cases}0 &\text{if}\
w\in W_R ,\\
1 &\text{if}\
w\in W\setminus W_R ,
\end{cases}
$$
and $\tau:=\tau_1$.
\end{remark}

\subsection{Proof of Theorem \ref{multiplication:thm}}\label{sec2s4}
We need to show that the product $T_wX^\nu$ expands on the basis
$X^\lambda T_v$ ($\lambda\in P, v\in W$) as
\begin{equation}\label{TwXnu-exp}
T_wX^\nu= \sum_{\substack{v\in W\\ v\le w}} (B_{v,w}^\nu + \sum_{\eta\in W_0\nu} A^{\eta ,\nu}_{v,w} X^\eta)T_v ,
\end{equation}
with expansion coefficients $A^{\eta ,\nu}_{v,w}$, $B_{v,w}^\nu$ governed by the recurrence relations and boundary conditions stated by the theorem.
For $\ell (w)=0$ the expansion in question reduces to Eq. \eqref{TuX-rel}.
For $\ell (w)> 0$, we pick $s:=s_j$ ($j=0,\ldots ,n$) such that $\tilde{w}:=sw<w$ and perform induction with respect to the length of $w$:
$$
 T_w X^\nu = T_s T_{\tilde{w}}X^\nu
=\sum_{v\le \tilde{w}} (B_{v,\tilde{w}}^{\nu}T_s T_v + \sum_{\eta\in W_0\nu} A^{\eta,\nu}_{v,\tilde{w}} T_sX^\eta T_v) .
$$
Invoking of the commutation relation (cf.  Eq. \eqref{TjX-rel})
\begin{subequations}
\begin{equation}\label{com-rel}
T_s X^\eta = X^{s'\eta} T_s+
\begin{cases}
 qX^\eta+q & \text{if $\langle \eta,\alpha^\vee\rangle=2\, (\iff \eta=\alpha)$},\\
 qX^\eta & \text{if $\langle \eta,\alpha^\vee\rangle=1$},\\
0& \text{if $\langle \eta,\alpha^\vee\rangle=0$},\\
 -qX^{s'\eta} & \text{if $\langle \eta,\alpha^\vee\rangle=-1$},\\
-qX^{s'\eta}-q & \text{if $\langle \eta,\alpha^\vee\rangle=-2\, (\iff \eta=-\alpha)$},
\end{cases}
\end{equation}
(where $\alpha=\alpha_j$) followed by the multiplication rule  (cf. Eqs. \eqref{ls1}, \eqref{TjTw})
\begin{equation}\label{mult-rule}
T_sT_v =\begin{cases} T_{sv} &\text{if}\ sv>v ,\\
T_{sv}+qT_v &\text{if}\ sv<v ,
\end{cases}
\end{equation}
\end{subequations}
entails that:
\begin{eqnarray}\label{TwXnu-exp2}
\lefteqn{T_w X^\nu=} && \\
&&  \sum_{v\le \tilde{w}} (B_{v,\tilde{w}}^\nu +\sum_{\eta\in W_0\nu} A^{\eta,\nu}_{v,\tilde{w}}X^{s'\eta} )T_{sv}
+q  \sum_{\substack{ v\le \tilde{w} \\ sv < v }} (B_{v,\tilde{w}}^\nu +\sum_{\eta\in W_0\nu} A^{\eta,\nu}_{v,\tilde{w}} X^{s'\eta} )T_v\nonumber \\
&&
+q\sum_{ v\le \tilde{w} } \Bigl(
\sum_{\substack{ \eta\in W_0\nu   \\ \langle \eta,\alpha^\vee\rangle> 0}} A^{\eta,\nu}_{v,\tilde{w}}X^{\eta}
- \sum_{\substack{ \eta\in W_0\nu  \\ \langle \eta,\alpha^\vee\rangle< 0}}
A^{\eta,\nu}_{v,\tilde{w}}X^{s'\eta} \Bigr) T_{v}
\nonumber \\
&&
+ q
\sum_{ v\le \tilde{w} }  \Bigl(
\sum_{\substack{ \eta\in W_0\nu   \\ \langle \eta,\alpha^\vee\rangle=2}} A^{\eta,\nu}_{v,\tilde{w}}-
\sum_{\substack{ \eta\in W_0\nu  \\ \langle \eta,\alpha^\vee\rangle=-2}} A^{\eta,\nu}_{v,\tilde{w}}\Bigr) T_v .\nonumber
\end{eqnarray}
Upon substituting $\tilde{w}=sw$, exploiting that $v\leq w$ if $sv\leq \tilde{w}$ or $v\leq \tilde{w}$
(in view of Eqs. \eqref{br-p1}, \eqref{br-p2}), and recalling that
$A^{\eta,\nu}_{v,\tilde{w}}=B_{v,\tilde{w}}^\nu=0$ if $v\not\leq \tilde{w}$ (by the induction hypothesis), it is not difficult to rewrite
the expansion in Eq.  \eqref{TwXnu-exp2}---by collecting the coefficients of the basis elements \eqref{H-basis}---in the form of
Eq. \eqref{TwXnu-exp}, with coeffcients  $A^{\eta ,\nu}_{v,w}$ and $B_{v,w}^\nu$
given by Eqs. \eqref{recrel1} and \eqref{recrel2}, respectively.

This proves that the product $T_w X^\nu$ expands as in Eq. \eqref{TwXnu-exp}  with
coefficients  $A^{\eta ,\nu}_{v,w}$ and $B_{v,w}^\nu$ satisfying the recurrence relations
in Eqs. \eqref{recrel1}, \eqref{recrel2} together with the boundary condition in Eq. \eqref{bc2}.
The invariance of these coefficients with respect to the action of $\Omega$ in Eq. \eqref{recrel0}
is then clear from the following elementary computation
(and the linear independence of the basis elements \eqref{H-basis}):
\begin{eqnarray*}
&&\sum_{v\le uw} (B_{v,uw}^\nu+\sum_{\eta\in W_0\nu}A^{\eta ,\nu}_{v,uw} X^{\eta})T_v
=T_{uw}X^\nu
=T_uT_wX^\nu = \\
&&\sum_{v\le w} (B_{v,w}^\nu+\sum_{\eta\in W_0\nu}A^{\eta ,\nu}_{v,w} X^{u'\eta})T_{uv}
=\sum_{v\le u w} (B_{u^{-1}v,w}^\nu+\sum_{\eta\in W_0\nu}A^{u'^{-1}\eta ,\nu}_{u^{-1}v,w} X^{\eta})T_v .
\end{eqnarray*}
Finally, it is immediate from the recurrence and Eq. \eqref{bc2}
that  $A^{\eta ,\nu}_{w,w}=A^{s^\prime \eta ,\nu}_{\tilde{w},\tilde{w}}$ and
$B_{w,w}^\nu=B_{\tilde{w},\tilde{w}}^\nu$, which together with Eq. \eqref{TuX-rel} implies Eq. \eqref{bc1}.

\section{Difference-Reflection Operators}\label{sec3}
When our positive scale parameter $c$ is integral-valued, the weight lattice $P\subset V$ is stable with respect to the action of $W$. Unless explicitly stated otherwise, from now onwards we will always consider an integral-valued scale parameter $c>1$:
\begin{equation}\label{int-cond}
c\in\mathbb{N}_{>1} .
\end{equation}
For this situation, we introduce in this section a representation of $\mathbb{H}$ given by difference-reflection operators acting in the space $\mathcal{C}(P)$ of complex functions $f:P\to \mathbb{C}$. Our starting point is an analogous difference-reflection representation of $H$ on $\mathcal{C}(P)$ taken from Ref. \cite{die-ems:unitary}.

\subsection{Difference-reflection representation $\hat{T}(H)$}
For any affine root  $a$, let us define the following
difference-reflection operator on $\mathcal C(V)$
\begin{subequations}
\begin{equation}\label{That1}
\hat T_a:= \tau_a+ \chi_a(s_a -1) \qquad (a\in R),
\end{equation}
where $\tau_a$ and $\chi_a$ act by multiplication and
\begin{equation}
\chi_a(x)=
\begin{cases}
\tau_a &  \text{if}\ a(x)> 0,\\
1  &  \text{if}\  a(x)= 0,\\
\tau_a^{-1} & \text{if}\  a(x)< 0 .
\end{cases}
\end{equation}
\end{subequations}
More explicitly, $\hat{T}_a$ acts on $f:V\to\mathbb{C}$ as
\begin{equation}
(\hat T_a f)(x)=
 \begin{cases}\label{That2}
  \tau_a f(s_a x)  & \text{if}\  a(x) > 0,\\
   \tau_a f(x)  & \text{if}\  a(x) =0,\\
  \tau_a^{-1} f(s_a x) +( \tau_a- \tau^{-1}_a)f(x) &   \text{if}\ a(x) < 0.
\end{cases}
\end{equation}
It is instructive to realize oneself that the difference-reflection operator
$\hat{T}_a$ maps continuous functions to continuous functions and that its action can be restricted to a discrete difference-reflection operator in
$\mathcal{C}(P)$ thanks to our integrality condition for $c$ (cf. Eq. \eqref{int-cond}).
The following proposition is a straightforward consequence of \cite[Thm. 3.1]{die-ems:unitary}.
 \begin{proposition}\label{dr-rep-aha:prp}
The assignment $T_j\mapsto \hat  T_j:=\hat{T}_{a_j}$ ($j=0,\dots, n$) and
$T_u\mapsto u$ ($u\in \Omega$)  extends (uniquely) to a
representation $h\mapsto \hat{T}(h)$ ($h\in H$) of the extended affine Hecke algebra on $C(V)$.
\end{proposition}

\begin{proof}
It was shown in \cite[Sec. 3]{die-ems:unitary} that for $c=1$ the assignment
$$T_j\mapsto \tau_{j}+ \chi_{a_j}^{-1}(s_{j} -1)\quad (j=0,\ldots ,n) $$
extends to a representation of the affine Hecke algebra $H_R$ on $\mathcal{C}(P)$.
Moreover, the proof of this result in {\em loc. cit.} extends verbatim from $\mathcal{C}(P)$ to
$\mathcal{C}(V)$ (cf. also \cite[Rem. 3.4]{die-ems:unitary}), first for $c=1$ and then for general (not necessarily integral) $c>0$ upon rescaling.
In other words, the assignment $T_j\mapsto \hat{T}_{-a_j}$ ($j=0,\ldots ,n$)
extends to a representation of $H_R$ on $\mathcal{C}(V)$.
By flipping the sign of the choice of positive roots it follows in turn that the
assignment
$T_0\mapsto\hat{T}_{\alpha_0^\vee -c}$, $T_j\mapsto \hat{T}_j$ ($j=1,\ldots ,n$)  extends to a representation
of $H_R$ on $\mathcal{C}(V)$. We will now modify the action associated with the affine generator $T_0$ following a translation argument borrowed from the proof of \cite[Thm. 4.11]{ems-opd-sto:trigonometric}.
Specifically, from the translation $t_{x}s_at_{-x}=s_{t_{x} a}=s_{a-a^\prime(x)}$ ($x\in cP$, $a\in R$) and the definition of the difference-reflection operators it is immediate that
$$t_{x} \hat
T_{\alpha^\vee+rc} t_{-x}=\hat  T_{\alpha^\vee+(rc-\langle x,\alpha^\vee \rangle)} \qquad
(x\in cP,\ \alpha\in R_0,\ r\in\mathbb{Z}).
$$
Hence, if one picks---for a given $j\in \{1,\ldots ,n\}$---a point $x\in cP$ belonging to the intersection of the hyperplanes $\langle x, \alpha_0^\vee\rangle+2c=0$ and $\langle x, \alpha_j^\vee\rangle=0$, then
 $t_{x} \hat  T_{\alpha_0^\vee-c}t_{-x}=\hat T_0$ and
$t_{x} \hat T_j t_{-x}=\hat T_j$.
The quadratic relation for $\hat{T}_0$ follows therefore from the quadratic relation for  $ \hat  T_{\alpha_0^\vee-c}$ and the braid relation
between $\hat{T}_0$ and $\hat{T}_j$ follows from the braid relation
between $ \hat  T_{\alpha_0^\vee-c}$ and $\hat{T}_j$. We thus conclude
that the assignment
$T_j\mapsto \hat{T}_{j}$ ($j=0,\ldots ,n$)  extends to a representation of $H_R$ on $\mathcal{C}(V)$.
The further extension of this representation from $H_R$ to $H$ as stated by the proposition follows from
the conjugation relation $u \hat T_j u^{-1} =\hat T_{u_j }$ ($u\in \Omega$, $j=0,\ldots ,n$), which is in turn immediate from
Eq. \eqref{W-rel1} (and the definition of the difference-reflection operators).
\end{proof}

\subsection{Action of $\hat{T}(H)$ on $\mathcal{C}(P)$}
A fundamental domain for the action of $W_R$ on $V$ is given by the
(closed) dominant Weyl alcove
\begin{align}\label{alcove}
A_c:=& \{ x\in V\mid a(x )\geq 0,\ \forall a\in R^+\}\\
=&\{ x\in V\mid 0\leq \langle x,\alpha^\vee\rangle\leq c,\ \forall\alpha\in R^+_0\}  . \nonumber
\end{align}
Since $R^+$ is stable with respect to the action of $\Omega$, so is $A_c$:
\begin{equation}\label{Ac-fix}
\Omega = \{w\in W\mid wA_c=A_c \} .
\end{equation}
Moreover, it is readily seen by induction with respect to the length that for any $w\in W$ and $f\in\mathcal{C}(V)$
\begin{equation}\label{Tact}
(\hat{T}_w f) (x) =\tau_w f (w^{-1}x) \qquad\text{if} \ x\in A_c,
\end{equation}
where $\hat{T}_w:=\hat{T}(T_w)$.
Indeed, for $\ell (w)=0$ Eq. \eqref{Tact} is immediate from the action of $\Omega\subset W$ on $\mathcal{C}(V)$ (cf. Eq. \eqref{W-action}), whereas for
$\ell (w)>0$ picking $j\in \{ 0,\ldots ,n\}$ such that $ws_j<w$ entails inductively that
$$(\hat{T}_wf)(x)=(\hat{T}_{ws_j} \hat{T}_{s_j} f)(x)=\tau_{ws_j}( \hat{T}_{s_j} f)(s_j w^{-1}x)=\tau_{w}f(w^{-1}x),$$ where the last equality hinges on Eq. \eqref{That2}  together with the observation that
$a_j(s_jw^{-1}x)=-(w a_j)(x)\geq 0$ (since $wa_j\in R^-$ by Eq. \eqref{ls2}).

We will now exploit the integrality of $c$ to restrict
the difference-reflection representation $\hat{T}(H)$ of Proposition \ref{dr-rep-aha:prp}  to a
 ($c$-dependent) representation of the extended affine Hecke algebra on $\mathcal{C}(P)$. This representation turns out to be faithful when the intersection of the weight lattice with the fundamental alcove
\begin{equation}\label{Pc}
P_c^+:= P\cap A_{c}
\end{equation}
contains a $W$-regular weight (i.e. a weight with a trivial stabilizer), which is the case when $c$ is larger than the Coxeter number
\begin{equation}\label{coxeter}
h(R_0):=\langle \rho ,\vartheta^\vee\rangle+1\quad \text{with}\quad  \rho:=\frac{1}{2}\sum_{\alpha\in R_0^+}\alpha
=\omega_1+\cdots +\omega_n .
\end{equation}
Here $\omega_1,\ldots ,\omega_n$ denotes the basis of fundamental weights:
\begin{equation}
\langle \omega_j,\alpha_k^\vee \rangle =\begin{cases} 1&\text{if}\ j= k, \\ 0&\text{if}\ j\neq k. \end{cases}
\end{equation}

\begin{proposition}
For $c> h (R_0)$, the action of the representation $\hat{T}(H)$  on  $\mathcal{C}(P)$ is faithful.
\end{proposition}
\begin{proof}
Clearly the (strictly dominant) weight $\rho$ \eqref{coxeter} lies in the interior of $A_c$ \eqref{alcove} provided $c> h(R_0)-1$. For $S\subset W$ finite and a linear combination of the basis elements
$h=\sum_{w\in S} a_w T_w\in H$ that lies in the kernel of $\hat{T}$, i.e.  with $\sum_{w\in S} a_w \hat{T}_w =0$,  it thus follows---upon acting on an arbitrary function $f\in\mathcal{C}(P)$ and evaluation at $\rho$ with the aid of Eq. \eqref{Tact}---that
\begin{equation*}\label{kernel}
 \sum_{w\in S} a_w \tau_w f (w^{-1}\rho )=0\qquad (\forall f\in\mathcal{C}(P)).
\end{equation*}
The weights $w^{-1}\rho$, $w\in W$ are all distinct if $\rho$ is $W$-regular. Hence,
picking $f$ equal to the indicator function of an arbitrary point $w^{-1}\rho$ ($w\in S$) reveals
that in this situation the corresponding coefficients $a_w$ must vanish, i.e. the kernel of $\hat{T}$ is
necessarily trivial in $H$.
It remains to infer that the condition $c>h(R_0)$ effectively guarantees that
the weight $\rho$ is $W$-regular. Given that $\rho$ lies
in the interior of $A_c$, we only need to check that it is $\Omega$-regular (cf. Eq. \eqref{Ac-fix}).
Moreover, since for any $u\in\Omega$ and $j\in\{0,\ldots ,n\}$
$$a_j(u^{-1} \rho)=(u a_j)(\rho)=a_{u_j}(\rho)=\begin{cases}  a_0(\rho)&\text{if}\ u_j=0 \\ 1&\text{if}\ u_j>0 \end{cases},
$$
one readily deduces that for $u\in\Omega\setminus\{ 1\}$ (whence $u_0>0$)
\begin{equation*}
u\rho = \sum_{\substack{1\leq j\leq n\\ j\neq u_0}} \omega_j+a_0(\rho)\omega_{u_0}\neq\rho ,
\end{equation*}
because $a_0(\rho)=c-\langle\rho,\vartheta^\vee\rangle=c+1-h(R_0)>1$.
\end{proof}

\subsection{The extension $\hat{T}(H)\to \hat{T}(\mathbb{H})$}
Our primary concern is to extend $\hat{T}(H)$ to a difference-reflection representation
$\hat{T}(\mathbb{H})$ of the double affine Hecke algebra at critical level on $\mathcal{C}(P)$.
To this end some further notation is needed.
For $x\in V$, let us write $w_x\in W_R$ for the (unique) shortest affine Weyl group element such that
\begin{equation}
x_+:=w_x x \in A_c
\end{equation}
and let $W_{R,x}$ denote the stabilizer subgroup
\begin{equation}\label{stabilizer}
W_{R,x}:=\{ w\in W_R\mid wx=x\} .
\end{equation}
The element $w_x$
is a minimal left-coset representative of $W_{R,x}$ and a minimal right-coset representative of $W_{R,x_+}$:
\begin{subequations}
\begin{align}
\ell (w_xv)&=\ell (w_x)+\ell (v)\quad \forall v\in W_{R,x} \label{minimal-left-coset} \\
\ell (vw_x)&=\ell (w_x)+\ell (v)\quad \forall v\in W_{R,x_+} .\label{minimal-right-coset}
\end{align}
\end{subequations}
The length of $w_x$ is given by the number of affine root hyperplanes $V_a$ separating $x$ from (the interior of) $A_c$:
\begin{equation}\label{Rwx}
R(w_x)=\{ a\in R^+\mid a(x)<0\} .
\end{equation}

To a weight $\nu\in P_\vartheta^\star$ \eqref{P-theta}, we now associate the operator
$\hat {X}^\nu:\mathcal{C}(P)\to \mathcal{C}(P)$ acting on $f:P\to\mathbb{C}$ via
\begin{subequations}
\begin{eqnarray}\label{Xnu}
\lefteqn{(\hat {X}^\nu f)(\lambda):=}&& \\
&& a_{\lambda ,\nu} f(\lambda-\nu) +b_{\lambda,\nu} (1-\tau_0^{-2}) f(\lambda) +
\tau^{-1}_{w_\lambda}\sum_{\substack{v<w_\lambda\\ \eta\in W_0\nu}} A^{\eta,\nu}_{v,w_\lambda}
    \tau^2_{w_{\lambda_+-\eta}}  (\hat  T_v f)(\lambda_+-\eta)
 \nonumber \\
&&+
\tau^{-1}_{w_\lambda}\sum_{v<w_\lambda} \tau_v \Bigl( B^\nu_{v,w_\lambda}  +(1-\tau_0^{-2}) \sum_{\eta\in W_0\nu} c_{\lambda_+,\eta}
        A^{\eta,\nu}_{v,w_\lambda} \Bigr) f(v^{-1}\lambda_+ ) \qquad (\lambda\in P),\nonumber
\end{eqnarray}
where
$A^{\eta,\nu}_{v,w}$ and $B^\nu_{v,w} $ refer to the expansion coefficients in Theorem \ref{multiplication:thm},
and the coefficients
$ a_{\lambda ,\nu}$, $b_{\lambda,\nu}$  and $c_{\lambda ,\eta}$ are given by
\begin{align}
 a_{\lambda ,\nu}&:=\tau_{w_{w_\lambda (\lambda-\nu)}}\tau_{w_{w_\lambda (\lambda-\nu)}w_\lambda}\tau_{w_\lambda}^{-1}, \label{a-coef}\\
b_{\lambda,\nu}&:=
\begin{cases}
c_{\lambda_+ ,w_\lambda^\prime\nu} &\text{if}\ (\lambda-\nu)_+\neq\lambda_+  \\
 \tau_{w_{w_\lambda (\lambda-\nu)}}^2  \chi (\nu^\vee+(1-\langle \lambda,\nu^\vee\rangle) ) & \text{if}\ (\lambda-\nu)_+ =\lambda_+
\end{cases}
 \label{b-coef}
\end{align}
and
\begin{equation}\label{c-coef}
c_{\lambda,\eta}:=\theta (\lambda -\eta) e^\vee_\tau (\eta ) (h_\tau^\vee)^{-\text{sign} (\langle\lambda ,\eta^\vee\rangle )}
\end{equation}
(with the convention that $\text{sign}(0):=0$).
Here
 $\theta:P\to \mathbb{N}\cup \{ 0\}$ denotes the function
 \begin{equation}\label{e:theta}
\theta(\lambda ):=\abs{\{b\in R(w_\lambda )\mid b(\lambda )=-2\}} ,
\end{equation}
\begin{equation}\label{hvee}
h^\vee_\tau:=\tau_0^2\, e_\tau^\vee (\vartheta )\quad\text{with}\quad e_\tau^\vee (\eta ):=\prod_{\alpha\in R^+_0} \tau_{\alpha^\vee}^{\langle \eta, \alpha^\vee \rangle } ,
\end{equation}
\end{subequations}
and---recall---$\chi:R\to \{0,1\}$ refers to the characteristic function of $R^-$ (cf. Remark \ref{root:rem} at the end of Appendix \ref{appA}). This definition of $\hat{X}^\nu$ is trivially extended to all $\nu\in P_\vartheta$ via the convention that the operator in question is equal to the identity operator in $\mathcal{C}(P)$ when $\nu=0$.

 \begin{theorem}\label{dr-rep-daha:thm}
The assignment $T_j\mapsto \hat  T_j$ ($j=0,\dots, n$),
$T_u\mapsto u$ ($u\in \Omega$), and $X^\nu\mapsto \hat{X}^\nu$ ($\nu\in P_\vartheta$) extends (uniquely) to a
representation $h\mapsto \hat{T}(h)$ ($h\in \mathbb{H}$) of the double affine Hecke algebra
at critical level
on $\mathcal{C}(P)$.
\end{theorem}
The proof  of this theorem is relegated to Section \ref{sec5} below. It consists of showing that the difference-reflection representation
$\hat{T}(\mathbb{H})$ in Theorem \ref{dr-rep-daha:thm}
arises from an equivalent representation of the double affine Hecke algebra at critical level
in terms of integral-reflection operators to be introduced in the next section.

\begin{remark}\label{dual-cox:rem}
In the equal label case, i.e. when $\tau_a=\tau$, $\forall a\in R$, one has that
\begin{subequations}
\begin{equation}
\tau_w=\tau^{\ell(w)} \quad\text{and}\quad
h_\tau^\vee=\tau^{2 h^\vee (R_0^\vee)}.
\end{equation}
Here $h^\vee (\cdot )$ refers to the dual Coxeter number:
\begin{equation}
h^\vee (R_0):=\langle \rho,\varphi^\vee\rangle +1,
\end{equation}
\end{subequations}
with $\varphi$ denoting the highest root of $R_0$
(so $h^\vee (R_0^\vee)=\langle\rho^\vee,\vartheta\rangle +1$ with $\rho^\vee:=\frac{1}{2}\sum_{\alpha\in R^+_0}\alpha^\vee$).
\end{remark}

\begin{remark}\label{Xnu-minuscule:rem}
For $\nu\in W_0\omega$ with $\omega\in P^+$ minuscule the action of $\hat {X}^\nu$  \eqref{Xnu}--\eqref{hvee} on $f\in \mathcal{C}(P)$ simplifies to
\begin{equation}\label{Xnu-min}
(\hat {X}^\nu f)(\lambda)=\tau_{w_{w_\lambda (\lambda-\nu)}}^2 f(\lambda-\nu)
+
\tau^{-1}_{w_\lambda}\sum_{\substack{v<w_\lambda\\ \eta\in W_0\nu}} A^{\eta,\nu}_{v,w_\lambda}
    \tau^2_{w_{\lambda_+-\eta}}  (\hat  T_v f)(\lambda_+-\eta) .
\end{equation}
Indeed, invoking of Lemma \ref{theta:lem} from Appendix \ref{appA} (below) reveals that for $\lambda\in P$ and $\nu\in P_\vartheta^\star$:
\begin{subequations}
\begin{align}
\label{pa}
w_{w_\lambda (\lambda-\nu)} \in W_{R,\lambda_+}\quad &\text{if}\quad (\lambda-\nu)_+\neq\lambda_+  \\
\intertext{and}
\label{pb} \nu\in W_0\vartheta \quad &\text{if}\quad  (\lambda-\nu)_+=\lambda_+\
\text{or}\ \theta (\lambda_+ -\nu)>0
\end{align}
\end{subequations}
(where we have used that $w_\lambda (\lambda-\nu)=\lambda_+- w_\lambda^\prime\nu$).
Hence, for $\nu\in W_0\omega$ with $\omega$ minuscule one has that
$(\lambda-\nu)_+\neq\lambda_+$ and $\theta (\lambda_+ -\nu)=0$ (by Eq. \eqref{pb}).
We thus conclude that in this situation
$a_{\lambda ,\nu}=\tau_{w_{w_\lambda (\lambda-\nu)}}^2 $ (because
$\ell (w_{w_{\lambda}(\lambda-\nu)}w_\lambda)=\ell (w_{w_{\lambda}(\lambda-\nu)})+\ell (w_\lambda)$ by Eq. \eqref{pa})
and, furthermore, that $b_{\lambda ,\nu}=c_{\lambda_+, \eta}=0$ for all $\eta\in W_0\nu$. The simplification of the action in Eq. \eqref{Xnu-min} is now evident upon recalling that the relevant coefficients
$B_{v,w}^\nu$ vanish in view of Remark \ref{rem2.5}.
\end{remark}

\section{Integral-Reflection Operators}\label{sec4}
In this section we formulate a representation of $\mathbb{H}$ on $\mathcal{C}(P)$
in terms of integral-reflection operators. This extends an analogous integral-reflection representation of
the affine Hecke algebra introduced in Ref. \cite{die-ems:unitary}.

\subsection{Integral-reflection representation $I(\mathbb{H})$}

For any affine root $a=\alpha^\vee+rc\in R$, we define a corresponding discrete integral-reflection operator of the form
\begin{subequations}
\begin{equation}\label{Ia:action}
I_a:=\tau_a s_a+(\tau_a-\tau_a^{-1}  )  J_a.
\end{equation}
Here $J_a:\mathcal C(P)\to \mathcal C(P)$ denotes a discrete integral operator
 that integrates the lattice function
$f(\lambda )$ over the $\alpha$-string from $\lambda$ to $s_a\lambda$,
where an endpoint in the negative half-space
\begin{equation}\label{Ha}
H_a:=\{x\in V \mid a(x)< 0\}
\end{equation}
is included
and
the endpoint(s) in the nonnegative half-space $V\setminus H_a$ are excluded:
\begin{eqnarray}\label{Ja:action}
\lefteqn{(J_af)(\lambda ) :=} && \\
&& \begin{cases}
-f(\lambda-\alpha)-
f(\lambda-2\alpha)-\cdots -f(s_a\lambda) &\text{if}\ a(\lambda) > 0 ,\\
0&\text{if}\ a(\lambda)= 0 ,\\
f(\lambda)+f(\lambda+\alpha)+\cdots +f(s_a\lambda-\alpha)
&\text{if}\ a(\lambda)< 0 .
\end{cases} \nonumber
\end{eqnarray}
\end{subequations}
As usual, the operators corresponding to the simple basis will be abbreviated by $I_j:=I_{a_j}$, $j=0,\ldots ,n$.

\begin{theorem}\label{integral-reflection:thm}
The assignment $T_j\to I_j$ ($j=0,\ldots ,n$), $T_u\mapsto u$ ($u\in \Omega$), and  $X^\lambda\to t_\lambda$ ($\lambda\in P$) extends (uniquely)
to a representation $h\to I(h)$ ($h\in \mathbb{H}$) of the double affine Hecke algebra at critical level $\mathbb H$ on
$\mathcal{C}(P)$.
\end{theorem}

The integral-reflection representation $I(\mathbb{H})$ is a double affine extension of
\cite[Prop. 4.1]{die-ems:unitary}, where is was demonstrated---by exploiting a duality relating to the standard polynomial representation of the affine Hecke algebra in terms of Demazure-Lusztig operators---that the assignment
$T_j\to I_j$ ($j=1,\ldots ,n$) and  $X^\lambda\to t_\lambda$ ($\lambda\in P$) extends
to a representation of the extended affine Hecke algebra with Bernstein-Zelevinsky basis $X^\lambda T_v$, $\lambda\in P$, $v\in W_0$. Hence, to prove Theorem \ref{integral-reflection:thm} it only remains to
infer the following relations (for $u\in\Omega$, $\lambda\in P$, $\nu\in P_\vartheta$):
\begin{subequations}
\begin{equation}\label{Irela}
ut_\lambda u^{-1}=t_{u^\prime\lambda}, \qquad u I_j u^{-1} =I_{u_j}\quad (j=0,\ldots ,n),
\end{equation}
\begin{equation}\label{Irelb}
(I_0-\tau_0)(I_0+\tau_0^{-1})=1,
\end{equation}
\begin{equation}\label{Irelc}
\underbrace{I_0I_jI_0\cdots}_{m_{0j}\ {\rm factors}}
=\underbrace{I_jI_0I_j\cdots}_{m_{0j}\ {\rm factors}}\qquad (j=1,\ldots ,n),
\end{equation}
and
\begin{equation} \label{Ireld}
I_0t_\nu = t_{s_0^\prime \nu }I_0+(\tau_0-\tau_0^{-1})\frac{t_\nu-t_{s'_0\nu}}{1-t_{-\alpha_0}}
\end{equation}
\end{subequations}
(cf. Eq. \eqref{term-geom}).

\subsection{Proof of Theorem \ref{integral-reflection:thm}}
Eq. \eqref{Irela} is an immediate consequence of the following
elementary relations upon specialization:
\begin{equation}\label{conj-rel1}
wt_xw^{-1} =t_{w^\prime x} \quad\text{and}\quad   wI_aw^{-1}=I_{wa}\quad  (w\in W,\, x\in V,\, a\in R).
\end{equation}
The first of these relations is manifest from the fact that $w$ acts as an affine linear transformation in $V$ and the second relation follows from the conjugation relations
\begin{equation}\label{conj-rel2}
ws_aw^{-1}=s_{wa} \quad\text{and}\quad  wJ_aw^{-1}=J_{wa}
\end{equation}
 (the former of which is obvious since $w\in W$ acts in fact as an affine {\em orthogonal} transformation in the Euclidean vector space $V$
and the latter conjugation is readily verified from the first and the definition of $J_a$ by acting with both sides on an arbitrary function $f:P\to\mathbb{C}$).
To verify Eqs. \eqref{Irelb}--\eqref{Ireld} it is convenient to proceed along the lines of the proof of Proposition \ref{dr-rep-aha:prp}.
For this purpose we prepare two lemmas.
The first formulates an auxiliary representation of $H_R$ following from \cite[Prop. 4.1]{die-ems:unitary} and the second
highlights the commutativity of the integral-reflection operators associated with perpendicular reflection hyperplanes.

\begin{lemma}\label{ir-rep-aha:lem}
The assignment $T_j\mapsto I_{\alpha_j^\vee}$ ($j=0, \ldots ,n$) extends (uniquely)
to a representation of the affine Hecke algebra $H_R$ on $\mathcal C(P)$.
\end{lemma}

\begin{proof}
From Prop. 4.1 of Ref. \cite{die-ems:unitary} cited above, it is immediate that the assignment
 $T_j\mapsto I_j=I_{\alpha_j^\vee}$ ($j=1, \ldots ,n$) extends (uniquely)
to an integral-reflection representation of the finite Hecke algebra $H_0$ on $\mathcal C(P)$.  To further extend this representation from $H_0$ to $H_R$ as stated by the lemma, it suffices to verify the quadratic relation \eqref{Irelb}
and the braid relations \eqref{Irelc}  with $I_0$ replaced by $I_{\alpha^\vee_0}$.
When acting with both sides of these relations on an arbitrary function $f\in\mathcal{C}(P)$, equality follows by relying once more on the integral-reflection representation of the finite Hecke algebra $H_0$.
Indeed, the quadratic relation for $I_{\alpha^\vee_0}$ follows from the quadratic relation for the integral-reflection representation of the finite Hecke algebra corresponding to the
rank-one root system with simple basis $\alpha_0$ and  the braid relation between $I_{\alpha^\vee_0}$ and
$I_{\alpha^\vee_j}$, $j\in \{ 1,\ldots ,n\}$ hinges on the
braid relation for the integral-reflection representation of the finite Hecke algebra corresponding to the rank-two root system with simple basis $\alpha_0$, $\alpha_j$. (Notice in this connection that
$\langle\alpha_0,\alpha_j^\vee\rangle\leq 0$ because $-\alpha_0=\vartheta\in P^+$, i.e.  $\alpha_0$ and $\alpha_j$ form the basis of a rank-two root system; moreover,
the orthogonal projections of $\lambda\in P$  onto the line spanned by $\alpha_0$ and onto the plane spanned by $\alpha_0$ and $\alpha_j$, respectively, belong to the weight lattices of the corresponding root systems of rank one and rank two.)
\end{proof}

\begin{lemma}\label{I-com:lem}
For $a=\alpha^\vee +rc$ and $b=\beta^\vee +lc$ in $R$ with $\langle \alpha,\beta^\vee\rangle =0$, one has that
 $$[I_a,I_b]:=I_a I_b-I_b I_a=0.$$
\end{lemma}
\begin{proof}
Since $[I_a,I_b]$ is equal to
\begin{align*}
\tau_a\tau_b & [s_a,s_b] +
\tau_a (\tau_b-\tau_b^{-1})[s_a,J_b]+
 \tau_b (\tau_a-\tau_a^{-1}) [J_a,s_b]  \\
&+ (\tau_a-\tau_a^{-1}) (\tau_b-\tau_b^{-1})   [J_a,J_b],
\end{align*}
it is sufficient to verify that
\begin{equation*}
[s_a,s_b]=0,\quad
[s_b,J_a]=0,\quad  [s_a,J_b]=0, \quad [J_a,J_b]=0.
\end{equation*}
The vanishing of the first three brackets is plain from the second relation in Eq. \eqref{conj-rel1} and the first relation in Eq. \eqref{conj-rel2}.  Moreover, acting with $J_aJ_b$ on an arbitrary function $f\in\mathcal{C}(P)$ and evaluation at $\lambda\in P$ manifestly produces that $(J_aJ_b f)(\lambda)=0$ if $a(\lambda)b(\lambda)=0$, whereas for  $a(\lambda)b(\lambda)\neq 0$ it yields a sum of the form
$$(J_aJ_b f)(\lambda)=\text{sign}(a(\lambda)b(\lambda))\sum_{\mu\in X_{ab;\lambda}} f(\mu).$$
Here $X_{ab;\lambda} \subset P$ denotes the intersection of $\lambda+Q$ and the rectangle with vertices
$\lambda$, $s_a\lambda$, $s_b\lambda$,  and $s_as_b\lambda$, leaving out the weights on the (boundary) $\alpha$-string and $\beta$-string intersecting at the vertex that belongs to the positive quadrant $\{ x\in V\mid a(x)>0,\, b(x)>0\}$
(which means, in particular, that the only vertex actually contained in
$X_{ab;\lambda}$ is the one belonging to the negative quadrant $\{ x\in V\mid a(x)<0,\, b(x)<0\}$).
Since $X_{ab;\lambda}$ is clearly symmetric in $a$ and $b$ so is the value of $(J_aJ_b f)(\lambda)$, whence
$ [J_a,J_b]=0$.
\end{proof}

After these preparations, we are now in the position to finish the proof of Eqs. \eqref{Irelb}--\eqref{Ireld}.
We will need the decomposition of $\vartheta^\vee$ in the simple basis of $R_0^\vee$:
\begin{equation}\label{highest-coroot}
 \vartheta^\vee = m_1\alpha_1^\vee +\cdots +m_n\alpha_n^\vee .
\end{equation}
Given a concrete root system $R_0$, the corresponding values of the strictly positive integers $m_1,\ldots,m_n$ can be read-off from the tables in Bourbaki \cite{bou:groupes}.

{\em  Proof of Eq. \eqref{Irelb}.}  One readily infers with the aid of  Bourbaki's tables that the GCD of
$m_1,\ldots ,m_n$ is equal to $1$. In other words, there exists a $\mu\in P$ such that $\langle \mu,\vartheta^\vee\rangle=1$. From the second conjugation relation in Eq. \eqref{conj-rel1} we now deduce that $t_{c\mu} I_{\alpha_0^\vee} t_{-c\mu}= I_{t_{c\mu}\alpha_0^\vee}=I_0$. Hence, the quadratic relation for $I_{\alpha_0^\vee}$ following from Lemma \ref{ir-rep-aha:lem} implies the quadratic relation for $I_0$.

{\em  Proof of Eq. \eqref{Irelc}.}
Let $J$  be the subset of indices $j\in\{ 1,\ldots ,n\}$ such that $\text{GCD} (m_1,\ldots,\hat{m}_j,\ldots ,m_n)=1$ (where the hat indicates that the corresponding value $m_j$ is ommitted). From the tables in Bourbaki it is seen that
$|J|=n$ for the (simply laced) root systems of type $ADE$, $|J|=n-1$ for the root systems of type $BCF$, and $|J|=0$ for the root system of type $G$.  For $j\in J$, there exists a $\mu\in P$ such that  $\langle \mu ,\vartheta^\vee\rangle =1$ and $\langle \mu,\alpha_j^\vee\rangle=0$. In this situation Eq. \eqref{conj-rel1} implies that
$t_{c\mu} I_{\alpha_0^\vee} t_{-c\mu}= I_{t_{c\mu}\alpha_0^\vee}=I_0$
and
$t_{c\mu} I_{\alpha_j^\vee} t_{-c\mu}= I_{t_{c\mu}\alpha_j^\vee}=I_{\alpha^\vee_j}=I_j$, whence the braid relation between
$I_0$ and $I_j$ is then a consequence of the braid relation between $ I_{\alpha_0^\vee} $ and  $I_{\alpha_j^\vee} $ following from Lemma \ref{ir-rep-aha:lem}. Inspection of the tables in Bourbaki reveals moreover that for $j\in \{1,\ldots ,n\}\setminus J$ one has that $\langle\alpha_0,\alpha_j^\vee\rangle=0$, except when $\alpha_j$ amounts to the short simple root of the root system of type $G$. When $\langle\alpha_0,\alpha_j^\vee\rangle=0$, the order $m_{0j}$ of $s_0s_j$ is equal to $2$. The braid relation reduces in this situation to the commutativity  $I_0I_j=I_jI_0$, which follows in turn from Lemma \ref{I-com:lem}.
When $\alpha_j$ is equal to the short simple root of the root system of type $G$, then $\tau_j=\tau_0$ and, furthermore, the roots $\alpha_0$ and $\alpha_j$ form the simple basis of a type $A$ root system of rank two centered at the intersection of the lines $a_0(x)=0$ and $a_j(x)=0$. Since the affine roots $a_0$ and $a_j$ take integral values on $P$, this lattice is contained in the weight lattice of our translated rank-two root system of type $A$. The upshot is that in this situation the braid relation between $I_0$ and $I_j$ follows from that of type $A$.

{\em  Proof of Eq. \eqref{Ireld}.}  For any $j\in \{ 1,\ldots ,n\}$,
the integral-reflection operator $I_j$ satisfies the relation
\begin{equation*}
I_jt_\lambda = t_{s_j \lambda }I_j+(\tau_j-\tau_j^{-1})\frac{t_\lambda-t_{s_j\lambda}}{1-t_{-\alpha_j}},
 \qquad \lambda\in P
\end{equation*}
(by \cite[Prop. 4.1]{die-ems:unitary}). By picking $\alpha_j$ short (so $W_0\alpha_j=W_0\alpha_0$ and $\tau_j=\tau_0$) and applying the second conjugation relation in Eq. \eqref{conj-rel1} with
an appropiate element $w\in W_0$ such that $w\alpha_j=\alpha_0$, we arrive at
Eq. \eqref{Ireld} with $I_0$ replaced by $I_{\alpha_0^\vee}$.
The translation employed in the proof of Eq. \eqref{Irelb} above now transforms the relation in question into
the one for $I_0$.


\section{Equivalence of $\hat{T}(\mathbb{H})$ and $I(\mathbb{H})$}\label{sec5}
In this section it is shown
that the difference-reflection representation $\hat{T}(\mathbb{H})$
arises from the integral-reflection representation $I(\mathbb{H})$
by means of an explicit intertwining operator.

\subsection{Intertwining operator}
Let  $\mathcal{J}:\mathcal{C}(P)\to\mathcal{C}(P)$ be the operator determined by the following action
on $f\in \mathcal \mathcal{C}(P)$:
\begin{equation}\label{int-op}
(\mathcal Jf)(\lambda):=\tau_{w_\lambda}^{-1} (I_{w_\lambda} f)(\lambda_+) \qquad (\lambda\in P) .
\end{equation}
It is clear from this definition that $\mathcal{J}$ acts trivially on lattice functions with support inside the fundamental alcove:   $ (\mathcal Jf)(\lambda)=f(\lambda)$ if $\lambda\in P_c^+$.
Moreover, in the definition of $\mathcal{J}$ one may actually replace $w_\lambda$ by any $w\in W_R$ such that $w\lambda=\lambda_+$. This hinges on the following useful invariance property for the action of $I_w$ on $f\in\mathcal{C}(P)$:
for any $w,\tilde{w}\in W_R$ and $\mu\in P_c^+$ one has that
\begin{equation}\label{stable}
\tau_w^{-1} (I_wf)(\mu)=\tau_{\tilde{w}}^{-1} (I_{\tilde{w}}f)(\mu)\quad\text{if}\quad w^{-1} \mu=\tilde{w}^{-1}\mu .
\end{equation}
Indeed, from the decomposition $w= v_\mu w_{w^{-1}\mu}$ with
$v_\mu\in W_{R,\mu}$ (so $\ell (w)=\ell (v_\mu )+\ell (w_{w^{-1}\mu })$ by Eq. \eqref{minimal-right-coset}) it is readily seen---via a reduced decomposition of $v_\mu$---that $\tau_w^{-1} (I_wf)(\mu)=(\mathcal{J}f)(w^{-1}\mu)$. Notice in this connection that $(I_jf)(\mu)=\tau_j f(\mu)$
for any simple reflection $s_j\in W_{R,\mu }$.

\begin{proposition}\label{intertwining:prp}
The operator $\mathcal{J}$  \eqref{int-op} constitutes a linear automorphism of $\mathcal{C}(P)$.
\end{proposition}
To demonstrate the bijectivity of $\mathcal{J}:C(P)\to C(P)$  it will be shown below that the operator in question is triangular with respect to a suitable partial order on $P$ that is inherited from the Bruhat order on $W_R$.
Next we will verify that the following intertwining relations associated with the generators of $\mathbb{H}$ are satisfied:
\begin{subequations}
\begin{align}
\mathcal{J} I_j &=  \hat{T}_j  \mathcal{J} \qquad (j=0,\ldots ,n), \label{ita} \\
  \mathcal{J}u  &= u \mathcal{J}  \qquad\ \  (u\in\Omega), \label{itb} \\
 \mathcal J t_\nu&= \hat{X}^\nu \mathcal J  \qquad (\nu\in P_\vartheta ).\label{itc}
\end{align}
\end{subequations}
These intertwining relations imply---in combination with the bijectivity of
$\mathcal{J}$ \eqref{int-op}---that Theorem \ref{dr-rep-daha:thm} follows as a direct consequence of
Theorem \ref{integral-reflection:thm}, i.e. the
difference-reflection representation $\hat{T}(\mathbb{H})$ arises
from the integral-reflection representation $I(\mathbb{H})$ in Theorem \ref{integral-reflection:thm} upon intertwining with
$\mathcal{J}$:
\begin{equation}\label{intertwining}
\mathcal{J} I(h)=\hat T(h)  \mathcal{J} \qquad (\forall h \in\mathbb{H}) .
\end{equation}

\subsection{Triangularity}
Our starting point for the triangularity proof for $\mathcal{J}$  \eqref{int-op} is the following (very rough) partial order on $V$ stemming from the Bruhat order:
\begin{equation}\label{br-order-orbit}
\forall x,y\in V:\quad x\leq y\quad\text{iff}\quad
\text{(i)}\ x_+=y_+\ \text{and}\
\text{(ii)}\ w_x\leq w_y.
\end{equation}
By definition, this order only compares points
belonging to the same $W_R$-orbit.
For $x,y\in V$, we denote by $[x,y]$
the interval $\{ z\in V\mid x\leq z\leq y\}$ (so $[x,y]=\emptyset$ if $x\not\leq y$).
Since a nonempty interval $[x,y]$ contains only a finite number of points, its convex hull $\text{Conv}\, [x,y]$ is a compact polytope in $V$.
We will now employ the convex polytope generated by the interval $[x_+,x]\subset W_Rx$ of all points smaller or equal to a given point $x\in V$ to refine the partial order in Eq. \eqref{br-order-orbit} (so as to permit comparing points belonging to different $W_R$-orbits).

\begin{lemma}\label{br-order-conv:lem} For any $x,y\in V$, one has that:
\begin{itemize}
\item[$i)$.] $\text{Conv}\, [y_+,y] \subseteq\text{Conv} (W_0 y) $,
\item[$ii)$.] $\text{Conv}\, [x_+,x]=\text{Conv}\, [y_+,y]  \Longrightarrow x =y$,
\item[$iii)$.] $x\in \text{Conv}\, [y_+,y] \Longrightarrow \text{Conv}\, [x_+,x] \subseteq\text{Conv}\, [y_+,y] $.
\end{itemize}
\end{lemma}

\begin{proof}
$i)$.  It is sufficient to verify that $[y_+,y]\subseteq \text{Conv}(W_0y)$. To this end we perform downward induction with respect to the partial order in Eq. \eqref{br-order-orbit} starting from the (trivial)
    maximal point $y$ ($\in [y_+,y]\cap \text{Conv}(W_0y)$).
Assuming that $x\in (y_+,y]:=[y_+,y]\setminus \{ y_+\}$ belongs to $\text{Conv}(W_0y)$,
let $s=s_a$, $a\in R^+$ denote any reflection such that $w_{x}s< w_{x}$ (so $sx\in [y_+,y]$ with $sx<x$ since $(sx)_+=x_+=y_+$ and $w_{sx}\leq w_{x} s<w_{x}\leq w_y$).
We then have that $$ \text{Conv}\, \{ x,sx \}\subseteq \text{Conv}\, \{ x,s^\prime x \}\subseteq \text{Conv}(W_0y) ,$$ where $\text{Conv}\, \{ x,sx \}$ denotes to the line segment connecting $x$ and $sx$ and---recall---$s^\prime\in W_0$ refers to the derivative of $s$. The first inclusion hinges on Eq. \eqref{sa} and the observation that
$a^\prime (x)\leq a(x)<0$ (cf. Eqs. \eqref{br-p3}, \eqref{Rwx})
and the second inclusion follows from the fact that
the convex polytope $\text{Conv}(W_0y)$ is $W_0$-invariant.
We thus conclude that the point
$sx\in [y_+,y]$ belongs to $ \text{Conv}(W_0y)$, which completes the induction step.

$ii)$. Since all points on the orbit $W_0y$ are
vertices of $\text{Conv}(W_0y)$, part $i)$ implies in particular that the point $y$
is a vertex of the $\text{Conv}\, [y_+,y]$. Morover, since all vertices of
the latter polytope are contained in the generating set
$[y_+,y]$, it follows that the point $y$ can be characterized as
the unique vertex of the convex polytope $\text{Conv}\, [y_+,y]$ that is maximal
with respect to the order in Eq. \eqref{br-order-orbit}.

$iii)$. It is sufficient to verify that
$[x_+,x]\subseteq \text{Conv}\, [y_+,y]$ when $x\in \text{Conv}\, [y_+,y]$, which will again be done by downward induction with respect to the order in Eq. \eqref{br-order-orbit} starting from the maximal point $x$.
Assuming that $\tilde{x}\in (x_+,x]$ belongs to $\text{Conv}\, [y_+,y]$,
let $s=s_a$, $a\in R^+$ denote any reflection such that $w_{\tilde{x}}s< w_{\tilde{x}}$ (so $s\tilde{x}\in [x_+,x]$ with $s\tilde{x}<\tilde{x}$). To complete the induction step it remains to show that $s\tilde{x}\in \text{Conv}\, [y_+,y]$.
To this end we note that the convex polytope cut out
by the intersection of $\text{Conv}\, [y_+,y]$ with
the nonpositive half-space $\overline{H}_a=H_a\cup V_a$ (cf. Eq. \eqref{Ha}) contains $\tilde{x}$ (because of Eqs. \eqref{br-p3}, \eqref{Rwx}) and it is finitely generated by
the points of $[y_+,y]\cap H_a$ and the vertices of the boundary facet
$\text{Conv}\, [y_+,y]\cap V_a$.
Since $s ([y_+,y]\cap H_a)\subseteq [y_+,y]$ (again by Eqs. \eqref{br-p3}, \eqref{Rwx}) and
$s(\text{Conv}\, [y_+,y]\cap V_a)=\text{Conv}\, [y_+,y]\cap V_a\subseteq\text{Conv}\, [y_+,y]$ (because the points of $V_a$ are fixed by $s$), it follows that $s\tilde{x}\in s(\text{Conv}\, [y_+,y]\cap \overline{H}_a)\subseteq \text{Conv}\, [y_+,y]$.
\end{proof}

It is immediate from part $ii)$ of Lemma \ref{br-order-conv:lem} that the inclusion relation
$$\text{Conv}\, [x_+,x] \subseteq\text{Conv}\, [y_+,y]\qquad (x,y\in V)$$
defines a partial order on $V$. This partial order refines the order in Eq. \eqref{br-order-orbit}:
$$x\leq y\Leftrightarrow x\in [y_+,y]\Leftrightarrow [x_+,x]\subseteq [y_+,y]\Rightarrow \text{Conv}\, [x_+,x]\subseteq \text{Conv}\, [y_+,y] .$$
We will now weaken the order in question to
the following partial order $\preceq$ on the weight lattice:
\begin{subequations}
\begin{equation}\label{br-order-P}
\forall \mu,\lambda\in P\quad \mu\preceq \lambda\ \text{iff}\ \text{(i)}\ \lambda-\mu\in Q\ \text{and}\
\text{(ii)}\ \text{Conv}\, [\mu_+,\mu ]\subseteq \text{Conv}\,[\lambda_+,\lambda ].
\end{equation}
It is obvious from this definition and part $iii)$ of Lemma \ref{br-order-conv:lem} that the set of weights smaller or equal to a given weight $\lambda\in P$ consists of the finite intersection
\begin{equation}
\{\mu\in P\mid \mu\preceq \lambda\}=
\text{Conv} [\lambda_+,\lambda ]\cap (\lambda+Q).
\end{equation}
\end{subequations}

\begin{proposition}\label{triangular:prp}
The operator $\mathcal{J}$ \eqref{int-op} is triangular with respect to the ordering in Eq. \eqref{br-order-P} in the sense that
for all $f\in \mathcal{C}(P)$ and $\lambda\in P$:
\begin{equation}\label{triangular}
(\mathcal{J}f)(\lambda )= \sum_{\mu\in P,\, \mu \preceq \lambda} J_{\lambda,\mu} f(\mu),
\quad\text{with}\quad J_{\lambda ,\lambda}=\tau_{w_\lambda}^{-2} ,
\end{equation}
and with the expansion coefficients $J_{\lambda ,\mu}$, $\mu\prec \lambda$ being Laurent polynomials in the
indeterminates $\tau_j$ with integral coefficients.
\end{proposition}

\begin{proof}
The proof is by induction on the length of $w_\lambda$, starting from the trivial situation that $\ell (w_\lambda )=0$ (i.e. $\lambda \in P_c^+$) in which case $(\mathcal{J}f)(\lambda) =f(\lambda)$.
For $\ell(w_\lambda)>0$, we pick $j\in\{ 0,\ldots ,n\}$ such that $w_\lambda s_j <w_\lambda$, whence $w_\lambda =w_{s_j\lambda}s_j$ with $\ell (w_\lambda)=\ell (w_{s_j\lambda})+1$ (i.e. $s_j\lambda<\lambda$).
Elementary manipulations now reveal that:
\begin{align}\label{Iw+1}
(\mathcal{J}f)(\lambda) =& \tau_{w_\lambda}^{-1}(I_{w_{\lambda}} f)(\lambda_+)=
\tau_j^{-1}\tau_{w_{s_j\lambda}}^{-1}
(I_{w_{s_j\lambda}} I_{j} f)((s_j\lambda )_+) \\
 \stackrel{\text{(i)}}{=}&
\tau_j^{-1}\sum_{ \mu\in P,\,\mu \preceq s_j\lambda} J_{s_j\lambda ,\mu} (I_{j} f) (\mu )
\stackrel{\text{(ii)}}{=}
\sum_{\mu\in P,\, \mu \preceq \lambda} J_{\lambda ,\mu} f(\mu ),
\nonumber
\end{align}
where the equality $\text{(i)}$ relies on the induction hypothesis, and the equality $\text{(ii)}$ hinges---upon recalling
that
$(I_jf)(\mu )$ involves a linear combination of function values supported on the $\alpha_j$ string from
$\mu$ to $s_j\mu$---on the observation that
the convex hull of
$ \text{Conv}\, [\lambda_+,s_j\lambda ]$ and
$s_j(\text{Conv}\, [\lambda_+,s_j\lambda ] )$ is contained in
$\text{Conv}\, [\lambda_+,\lambda]  $ (because
$
[\lambda_+,s_j\lambda]\cup s_j([\lambda_+,s_j\lambda])\subseteq [\lambda_+,\lambda]$).
It is manifest from the definition of the integral-reflection operators that the expansion coefficients $J_{\mu ,\lambda}$
are Laurent polynomials in the
indeterminates $\tau_j$ with integral coefficients. It remains to verify that the value of the diagonal coefficient $J_{\lambda ,\lambda}$ pans out as stated in Eq. \eqref{triangular}. Indeed,
since $\lambda$ is the unique maximal
vertex of the polytope $\text{Conv}\, [\lambda_+,\lambda]$
with respect to $\preceq$ (cf. part $ii)$ of Lemma \ref{br-order-conv:lem}), it follows that in Eq. \eqref{Iw+1} the
only contribution to the diagonal on the LHS of equality $(ii)$ stems from the term corresponding to $\mu=s_j\lambda$.  Moreover, one reads-off from the explicit action of $I_j$ that the coefficient of $f(\lambda )$ in
$(I_jf)(s_j\lambda)$  is given by   $\tau_j^{-1}$ if $s_j\lambda < \lambda$ (so $a_j(\lambda )<0$).
A comparison of the coefficients of $f(\lambda)$ on both sides of equality $(ii)$ thus entails that
$J_{\lambda,\lambda}=\tau_j^{-2}J_{s_j\lambda ,s_j\lambda}$, whence the stated value of $J_{\lambda,\lambda}$ follows from the induction.
\end{proof}

The triangularity in Eq. \eqref{triangular} implies that $f\in\mathcal{C}(P)$ can be uniquely solved from the linear equation $(\mathcal{J}f)(\lambda)=g(\lambda)$ for any $g\in\mathcal{C}(P)$, by performing induction in $\lambda$ with respect to the order in Eq. \eqref{br-order-P}. In other words, the intertwining operator $\mathcal{J}:\mathcal{C}(P)\to \mathcal{C}(P)$ is a bijection.

\subsection{Intertwining relations}
We will now verify the intertwining relations in Eqs. \eqref{ita}--\eqref{itc}.
The proof of the first two relations hinges on some short computations analogous to those in the proofs of \cite[\text{Lem}.~5.2]{die-ems:unitary} and \cite[\text{Prop}.~4.2]{die-ems:discrete}. They are included here merely to keep the presentation self-contained. The proof of the last relation is more intricate and some of the harder details are hidden away in Appendix \ref{appA} at the end of the paper.
In all three cases the idea of the proof is to act with the operator at the LHS of the relation on an arbitrary function
$f\in\mathcal{C}(P)$ and then
pull the action of the integral-reflection representation through the intertwining operator
so as to recover the operator at the RHS.

\subsubsection{Proof of Eq. \eqref{ita}} For any $j\in \{0,\ldots n\}$:
\begin{eqnarray*}
(\mathcal{J} I_j f)(\lambda ) &=& \tau^{-1}_{w_\lambda}  ( I_{w_\lambda} I_j f)( \lambda_+ )  \\
&\stackrel{\text{Eq.}~\eqref{TwTj}}{=}&
\tau^{-1}_{w_\lambda}
 \left( (I_{w_\lambda s_j} f)( \lambda_+ )
 +\chi (w_\lambda a_j )(\tau_j-\tau_j^{-1})
 ( I_{w_\lambda} f)( \lambda_+ )  \right)  \\
& = & \tau_j \tau^{-1}_{w_\lambda} (I_{w_\lambda} f)(\lambda_+ ) +\\
&&
\tau_j^{\text{sign} (w_\lambda a_j)}
\bigl(
\tau^{-1}_{w_\lambda s_j} (I_{w_\lambda s_j} f)( \lambda_+ ) -
\tau^{-1}_{w_\lambda} (I_{w_\lambda} f)( \lambda_+ )
\bigr) \\
&\stackrel{\text{Eq.}~\eqref{stable}}{=}& (\hat{T}_j\mathcal{J}f)(\lambda ).
\end{eqnarray*}

\subsubsection{Proof of Eq. \eqref{itb}}
For any $u\in \Omega$:
\begin{align*}
(\mathcal J u f)(\lambda)& = \tau^{-1}_{w_{\lambda}} (I_{w_{\lambda}}u f)(\lambda_+) = \tau^{-1}_{w_{\lambda}} (u^{-1}I_{w_{\lambda}}u f)(u^{-1}\lambda_+) \\
&= \tau^{-1}_{w_{u^{-1}\lambda}} (I_{w_{u^{-1}\lambda}} f)((u^{-1}\lambda)_+) =
(\mathcal  J f)(u^{-1}\lambda)= (u\mathcal J f)(\lambda)
\end{align*}
(where it was used that $u^{-1} I_{w_\lambda} u=I_{u^{-1}w_{\lambda}u}$
and
$ u^{-1} w_\lambda u=w_{u^{-1}\lambda} $, cf. Eq. \eqref{conj-rel1}).

\subsubsection{Proof of Eq. \eqref{itc}} For $\nu=0$ Eq. \eqref{itc} is trivial whereas for
any $\nu\in P_\vartheta^\star$:
\begin{align}\label{Jtf}
&(\mathcal J t_\nu f)(\lambda)= \tau^{-1}_{w_{\lambda}} (I_{w_{\lambda}}t_\nu f)(\lambda_+) \stackrel{\text{Thm}.~\ref{multiplication:thm}}{=} \\
&\tau^{-1}_{w_\lambda} (t_{w^{\prime}_{\lambda}\nu} I_{w_{\lambda}}f)(\lambda_+)
+\tau^{-1}_{w_\lambda} \sum_{\substack{v\in W_R\\ v< w_\lambda}} \Bigl( B^\nu_{v,w_\lambda} (I_vf)(\lambda_+)+ \sum_{ \eta\in W_0\nu} A^{\eta,\nu}_{v,w_\lambda} (t_\eta I_vf)(\lambda_+) \Bigr) . \nonumber
\end{align}
The first term may be rewritten as
\begin{eqnarray*}
  \tau^{-1}_{w_\lambda} (t_{w'_\lambda\nu} I_{w_\lambda}f)(\lambda_+)
  & =& \tau^{-1}_{w_\lambda} (I_{w_\lambda}f)(w_\lambda(\lambda-\nu)) \\
& \stackrel{\text{Eq.}~\eqref{intertwining-property}}{=} & a_{\lambda,\nu} (\mathcal{J}f)(\lambda-\nu) + b_{\lambda,\nu}(1-\tau_0^{-2})(\mathcal {J}f)(\lambda) ,
\end{eqnarray*}
with $a_{\lambda,\nu}$ and $b_{\lambda,\nu}$ as in  Eqs. \eqref{a-coef}--\eqref{hvee}.
Furthermore, invoking of the stability property in Eq. \eqref{stable} allows to recast
the term with coefficient $B^\nu_{v,w_\lambda} $ as
 $$(I_vf)(\lambda_+)=\tau_v (\mathcal Jf)(v^{-1}\lambda_+) .$$
Finally, for the term with coefficient $A^{\eta,\nu}_{v,w_\lambda} $ we obtain
\begin{align*}
&( t_\eta I_vf)(\lambda_+)  =(I_vf)(\lambda_+-\eta) \\
& \stackrel{\text{(i)}}{=} \tau^2_{w_{\lambda_+-\eta}} (\mathcal JI_v f)(\lambda_+-\eta)
+c_{\lambda_+ ,\eta} (1-\tau^{-2}_0) (I_vf)(\lambda_+)\\
&\stackrel{\text{(ii)}}{=}  \tau^2_{w_{\lambda_+-\eta}} (\hat  T_v \mathcal{J}f)(\lambda_+-\eta)
+\tau_v c_{\lambda_+ ,\eta} (1-\tau^{-2}_0) (\mathcal{J}f)(v^{-1}\lambda_+) ,
\end{align*}
with $c_{\lambda,\eta}$ given by Eqs. \eqref{c-coef}--\eqref{hvee}.
Here we used $\text{(i)}$ Lemma \ref{Iqm-action:lem} in the form
 $$ f(\lambda_+-\eta) =\tau^2_{w_{\lambda_+-\eta}}(\mathcal Jf)(\lambda_+-\eta)+c_{\lambda_+,\eta}(1-\tau_0^{-2})f(\lambda_+)$$
 with $f$ replaced by $I_vf$ and
 $\text{(ii)}$  Eq. \eqref{ita} in combination with Proposition \ref{dr-rep-aha:prp} and Theorem \ref{integral-reflection:thm}
 for the first term and Eq. \eqref{stable} for the second term.
Substitution in the RHS of Eq. \eqref{Jtf} now entails that
 \begin{eqnarray*}
\lefteqn{(\mathcal J t_\nu f)(\lambda)=
 a_{\lambda ,\nu} (\mathcal{J}f)(\lambda-\nu) +b_{\lambda,\nu}(1-\tau^{-2}_0)  (\mathcal{J}f)(\lambda) +} && \\
 &&
\tau^{-1}_{w_\lambda}\sum_{\substack{v<w_\lambda\\ \eta\in W_0\nu}} A^{\eta,\nu}_{v,w_\lambda}
    \tau^2_{w_{\lambda_+-\eta}}  (\hat  T_v \mathcal{J}f)(\lambda_+-\eta) +
 \nonumber \\
&&
\tau^{-1}_{w_\lambda}\sum_{v<w_\lambda} \tau_v \Bigl( B^\nu_{v,w_\lambda}  +(1-\tau_0^{-2}) \sum_{\eta\in W_0\nu}
        c_{\lambda_+ ,\eta} A^{\eta,\nu}_{v,w_\lambda} \Bigr) (\mathcal{J}f)(v^{-1}\lambda_+ ) \nonumber \\
        && =( \hat{X}^\nu \mathcal{J} f)(\lambda) .
\end{eqnarray*}

\begin{remark}
The intertwining property in Eq. \eqref{intertwining} constitutes a discrete analog of the intertwining property in \cite[\text{Thm.}~5.3]{ems-opd-sto:periodic} between the integral-reflection and the Dunkl-type differential-reflection representation of the degenerate double affine Hecke algebra at critical level.
In particular, the intertwining relation in Eq. \eqref{itc} producing the operator $\hat{X}^\nu$ as the image of the translation operator $t_\nu$ is the discrete counterpart of the intertwining relation in \cite[\text{Eq.}~(5.10)]{ems-opd-sto:periodic}, where  the directional derivative gets mapped onto the corresponding Dunkl-type differential-reflection operator. In other words, the operators $\hat{X}^\nu$, $\nu\in P_\vartheta^*$ should be viewed as the discrete Dunkl-type operators associated with the present construction.
From this perspective, our proof of Eq. \eqref{itc}
based on the multiplication formula in Theorem~\ref{multiplication:thm} corresponds to a discrete counterpart of the proof of \cite[\text{Eq.}~(5.10)]{ems-opd-sto:periodic} based on the multiplication formula \cite[\text{Eq.}~(5.11)]{ems-opd-sto:periodic}.
\end{remark}


\section{Integrable Discrete Laplacians}\label{sec6}
The images of the center under representations of (degenerate) affine Hecke algebras
provide a fruitful framework for describing quantum integrable systems, cf. e.g. Refs.
\cite{hec-opd:yang,che:double,mac:affine,die-ems:unitary} for examples of integrable systems arising in this manner.
In the present setup quantum integrable systems are obtained similarly via the image  of the
commutative subalgebra
$\mathbb{C}[X]^{W_0}\subset \mathcal{Z}(\mathbb{H})$,
in the spirit of
\cite{ems-opd-sto:periodic} where the quantum integrals for the Schr\"odinger operator with a delta potential on the affine root system
arise from the algebra of $W_0$-invariant polynomials contained in the center of
the degenerate double affine Hecke algebra at critical level.
For our integral-reflection representation, the quantum integrable system at issue consists merely of (linear combinations of) free discrete Laplace operators in $\mathcal{C}(P)$ of the form
\begin{equation}\label{free-laplacian}
m_\omega(t):=I(m_\omega(X))=\sum_{\nu\in W_0\omega } t_\nu\qquad (\omega\in P^+).
\end{equation}
The difference-reflection representation on the other hand produces a nontrivial integrable $\tau$-deformation of these Laplace operators (cf. Remark \ref{free-laplacians:rem} below). In this section we compute the simplest of these (deformed) Laplacians explicitly.

\subsection{Laplacians associated with the (quasi-)minuscule weights}
To any symmetric polynomial $p\in \mathbb{C}[X]^{W_0}$, we associate a (deformed) discrete Laplace operator $L_p:\mathcal{C}(P)\to\mathcal{C}(P)$ defined by
\begin{equation}\label{Lp}
L_p =p(\hat{X}):=\hat{T}(p(X)) \qquad (p\in \mathbb{C}[X]^{W_0}).
\end{equation}
Since the symmetric subalgebra $\mathbb{C}[X]^{W_0}\subset\mathbb{H}$ is commutative, the associated Laplacians $L_p$, $p\in \mathbb{C}[X]^{W_0}$ form a quantum integrable system:
\begin{equation}\label{integrability}
[L_p, L_{\tilde{p}}] =0\qquad (\forall p,\tilde{p}\in \mathbb{C}[X]^{W_0}).
\end{equation}
The next theorem makes the action of $L_p$ on $\mathcal{C}(P)$ explicit for $p=m_\omega$ with $\omega$ (quasi-)minuscule.

\begin{theorem}\label{action:thm}
For $\omega\in P^+$ (quasi-)minuscule, the Laplacian
$L_\omega:=L_{m_\omega}=\sum_{\nu\in W_0\omega}\hat{X}^\nu$ acts on $\mathcal{C}(P)$ via
\begin{equation}\label{Lomega}
(L_\omega f)(\lambda)=
\sum_{\nu\in W_0\omega} \Bigl( a_{\lambda ,\nu} f(\lambda -\nu)+ b_{\lambda ,\nu}(1-\tau^{-2}_0)f(\lambda) \Bigr)\qquad
(f\in \mathcal{C}(P),\lambda\in P),
\end{equation}
where $a_{\lambda ,\nu}$ and $b_{\lambda,\nu}$ are given by Eqs. \eqref{a-coef}--\eqref{hvee}.
\end{theorem}

\begin{proof}
Acting with $L_\omega\mathcal{J}$ on an arbitrary function $f\in\mathcal{C}(P)$ entails that (for all $\lambda\in P$):
\begin{align*}
  (L_\omega \mathcal{J}f)(\lambda) &\stackrel{\text{Eq.}~\eqref{intertwining}}{=} (\mathcal{J} m_\omega(t) f) (\lambda)\stackrel{\text{Eq.}~\eqref{int-op}}{=}
\tau_{w_\lambda}^{-1} (I_{w_\lambda} m_\omega(t) f)(\lambda_+)  \\
&\stackrel{\text{Eq.}~\eqref{center}}{=}\tau_{w_\lambda}^{-1} (m_\omega(t) I_{w_\lambda} f)(\lambda_+)
=\tau_{w_\lambda}^{-1} \sum_{\nu\in W_0\omega} (I_{w_\lambda} f)(w_\lambda(\lambda-\nu))\\
&\stackrel{\text{Eq.}~\eqref{intertwining-property}}{=}
\sum_{\nu\in W_0\omega} a_{\lambda,\nu}(\mathcal Jf)(\lambda-\nu)+b_{\lambda,\nu}(1-\tau^{-2}_0) (\mathcal Jf)(\lambda) ,
\end{align*}
whence the theorem follows from the bijectivity of $\mathcal{J}$.
\end{proof}

\subsection{$W_R$-invariant reduction}
The $W_R$-invariant subspace
\begin{equation}\label{WRinv1}
\mathcal{C}(P)^{W_R}:=\{ f\in\mathcal{C}(P)\mid wf=f,\ w\in W_R\}
\end{equation}
consists of the functions $f:\mathcal{C}(P)\to\mathbb{C}$ that are symmetric with respect to the action
of $W_0$ and periodic with respect to the lattice of translations $t (cQ)$.

\begin{proposition}\label{WR-reduction:prp}
The $W_R$-invariant subspace $\mathcal{C}(P)^{W_R}$ is stable with respect to the action of $L_p$, $p\in \mathbb{C}[X]^{W_0}$.
\end{proposition}
\begin{proof}
Since the relations $s_jf=f$ and $\hat{T}_jf=\tau_jf$ for $j\in \{ 0,\ldots ,n\}$ are equivalent, one may alternatively characterize the space of $W_R$-invariant functions as
\begin{equation}\label{WRinv2}
\mathcal{C}(P)^{W_R}=\{ f\in\mathcal{C}(P)\mid \hat{T}_wf=\tau_w f,\ w\in W_R\} .
\end{equation}
The proposition thus follows from the fact that the Laplacians
$L_p$, $p\in \mathbb{C}[X]^{W_0}$ commute with the operators $\hat{T}_w$, $w\in W_R$ (cf. Eq. \eqref{center}).
\end{proof}

The finite intersection $P^+_c$ \eqref{Pc} constitutes a fundamental domain for the action of $W_R$ on $P$. Hence, we may identify the space
$\mathcal{C}(P)^{W_R}$ with the finite-dimensional space $\mathcal{C}(P^+_c)$ of functions $f:P^+_c\to\mathbb{C}$.
The following theorem
describes the restriction of the action of the operator $L_\omega$ in Theorem \ref{action:thm} to
$\mathcal{C}(P^+_c)$.
The proof hinges on Macdonald's celebrated product formula for the generalized Poincar\'e series of the Coxeter group
associated with the length multiplicative function $\tau$ \cite{mac:poincare}. For the stabilizer
$W_{R,\lambda}$ of $\lambda\in P_c^+$, Macdonald's formula for the Poincar\'e polynomial in question becomes:
\begin{subequations}
\begin{align}\label{poincare-stab}
W_{R,\lambda}(\tau^2)&:= \sum_{w\in W_{R,\lambda}}\tau_w^2=
\prod_{a\in R_\lambda^+}\frac{1-\tau_a^2 \tau_s^{2\text{ht}_s(a)} \tau_l^{2\text{ht}_l(a)} }{1-\tau_s^{2\text{ht}_s(a)} \tau_l^{2\text{ht}_l(a)}} \\
&= \prod_{\substack{\alpha\in R_0^+\\ \langle \lambda,\alpha^\vee\rangle=0}}
\frac{1-\tau^2_{\alpha^\vee} e_\tau (\alpha^\vee ) }{1-e_\tau (\alpha^\vee )}
\prod_{\substack{\alpha\in R_0^+\\ \langle \lambda,\alpha^\vee\rangle=c  }}
\frac{1-\tau^2_{\alpha^\vee} h_\tau e_\tau (-\alpha^\vee )}{1-h_\tau e_\tau (-\alpha^\vee )}, \nonumber
\end{align}
where
\begin{equation}\label{htau}
h_\tau:=\tau_0^2\, e_\tau (\vartheta^\vee )\quad\text{and}\quad e_\tau (\eta ):=\prod_{\alpha\in R^+_0} \tau_{\alpha^\vee}^{\langle \eta, \alpha\rangle } \quad (\eta\in P^\vee )
\end{equation}
(with $P^\vee$ denoting the lattice of coweights).
Here
we used the notation $R_\lambda^+:=\{ a\in R^+\mid a(\lambda)=0\} $ and the heights $\text{ht}_s(a)$ and
$\text{ht}_l(a)$ of  $a=k_0a_0+\cdots +k_na_n\in R^+$ are defined as
\begin{equation}
\text{ht}_s(a):=\sum_{\substack{0\leq j\leq n\\ \alpha_j\ \text{short}}} k_j \quad \text{and}\quad \text{ht}_l(a):=\sum_{\substack{1\leq j\leq n\\ \alpha_j\ \text{long}}} k_j
\end{equation}
(with---recall---the convention that all finite roots are {\em short} if $R_0$ is simply-laced).
\end{subequations}

\begin{theorem}
\label{action-symmetric:thm}
For $\omega\in P^+$ (quasi-)minuscule,
the restriction of the action of  $L_\omega$ to $\mathcal C(P_c^+)\simeq \mathcal C(P)^{W_R}$ is given by
\begin{subequations}
\begin{equation}\label{Lomega-sym}
(L_\omega f)(\lambda)=
U_{\lambda,\omega}(\tau^2) f (\lambda)+ \sum_{\substack{\nu\in W_0\omega\\ \lambda-\nu\in P_c^+  }} V_{\lambda,\nu}(\tau^2) f (\lambda-\nu)
\end{equation}
$(f\in \mathcal{C}(P_c^+),  \lambda\in P_c^+)$, with
\begin{align}\label{Vlambdanu}
V_{\lambda,\nu}(\tau^2)&:=
\prod_{a\in R_\lambda^+\setminus R^+_{\lambda-\nu}}\frac{1-\tau_a^2 \tau_s^{2\text{ht}_s(a)} \tau_l^{2\text{ht}_l(a)} }{1-\tau_s^{2\text{ht}_s(a)} \tau_l^{2\text{ht}_l(a)}} \\
&=\prod_{\substack{\alpha\in R_0^+\\ \langle \lambda,\alpha^\vee\rangle=0 \\ \langle \nu,\alpha^\vee\rangle<0 }}
\frac{1-\tau^2_{\alpha^\vee} e_\tau (\alpha^\vee ) }{1-e_\tau (\alpha^\vee )}
\prod_{\substack{\alpha\in R_0^+\\ \langle \lambda,\alpha^\vee\rangle=c \\ \langle \nu,\alpha^\vee\rangle>0 }}
\frac{1-\tau^2_{\alpha^\vee} h_\tau e_\tau (-\alpha^\vee )}{1-h_\tau e_\tau (-\alpha^\vee )} \nonumber
\end{align}
and
\begin{eqnarray}\label{Ulambdaomega}
U_{\lambda ,\omega}(\tau^2):=
\sum_{\substack{\nu\in W_0\omega \\ (\lambda-\nu)_+=\lambda   }} \tau_{w_{\lambda-\nu}}^2+
(1-\tau_0^{-2})
\sum_{\substack{\nu\in W_0\omega \\ w_{\lambda-\nu} \lambda=\lambda   }}c_{\lambda ,\nu} ,
\end{eqnarray}
where the coefficients $c_{\lambda ,\nu}$ are of the form in Eqs. \eqref{c-coef}--\eqref{hvee}.
\end{subequations}
\end{theorem}

\begin{proof}
It is immediate from Theorem \ref{action:thm} that the action of $L_\omega$ reduces to an
action on
$\mathcal{C}(P^+_c)\cong\mathcal{C}(P)^{W_R}$ of the form in Eq. \eqref{Lomega-sym} with
\begin{align*}\nonumber
V_{\lambda ,\nu}(\tau^2)&=\sum_{\substack{\eta\in W_0\omega\\ (\lambda -\eta)_+=\lambda-\nu}}
a_{\lambda ,\eta} \stackrel{\text{(i)}}{=}
\sum_{\substack{\eta\in W_0\omega\\ (\lambda -\eta)_+=\lambda-\nu}} \tau_{w_{\lambda-\eta}}^2
\stackrel{\text{(ii)}}{=}
\sum_{\mu\in W_{R,\lambda} (\lambda -\nu)} \tau_{w_\mu}^2\\
&=
W_{R,\lambda}(\tau^2)/(W_{R,\lambda}\cap W_{R,\lambda-\nu})(\tau^2) \label{Valt}
\end{align*}
and
\begin{eqnarray*}\nonumber
U_{\lambda ,\omega}(\tau^2)&=&
\sum_{\substack{\nu\in W_0\omega \\ (\lambda-\nu)_+=\lambda   }} a_{\lambda ,\nu}+
(1-\tau_0^{-2})
\sum_{\nu\in W_0\omega } b_{\lambda,\nu}\\
&\stackrel{\text{(i),(ii),(iii)}}{=}&
\sum_{\substack{\nu\in W_0\omega \\ (\lambda-\nu)_+=\lambda   }} \tau_{w_{\lambda-\nu}}^2+
(1-\tau_0^{-2})
\sum_{\substack{\nu\in W_0\omega \\ w_{\lambda-\nu} \lambda=\lambda   }} c_{\lambda ,\nu} . \label{Ualt}
\end{eqnarray*}
Here we used that for any $\lambda\in P^+_c$ and $\nu\in W_0\omega$:
 $\text{(i)}$ $a_{\lambda,\nu}=\tau_{w_{\lambda-\nu}}^2$, $\text{(ii)}$
$(\lambda-\nu)_+\neq \lambda$ iff $w_{\lambda-\nu}\in W_{R,\lambda}$ (by part $i)$ of Lemma \ref{theta:lem}),
and $\text{(iii)}$ $\nu^\vee+(1-\langle\lambda,\nu^\vee\rangle)\in R^+$ if $(\lambda-\nu)_+= \lambda$ (by part $ii)$ of Lemma \ref{theta:lem}), whence $b_{\lambda ,\nu}=0$ in this situation.
Finally, by plugging in Macdonald's product formula for the Poincar\'e polynomial in Eq. \eqref{poincare-stab}
the coefficient $V_{\lambda ,\nu}(\tau^2)$ is rewritten in the form given by Eq. \eqref{Vlambdanu}.
\end{proof}

\begin{remark} For $\omega$ minuscule, the actions of $L_\omega$ in Theorems \ref{action:thm} and \ref{action-symmetric:thm} reduce to (cf. Remark \ref{Xnu-minuscule:rem})
\begin{subequations}
\begin{equation}\label{Lminuscule}
(L_\omega f)(\lambda)=
\sum_{\nu\in W_0\omega}  \tau^2_{w_{w_{\lambda}(\lambda-\nu)}}  f(\lambda -\nu)\qquad
(f\in \mathcal{C}(P),\lambda\in P)
\end{equation}
and
\begin{equation}\label{Lminuscule-red}
(L_\omega f)(\lambda)=
\sum_{\substack{\nu\in W_0\omega\\ \lambda-\nu\in P_c^+  }} V_{\lambda,\nu}(\tau^2) f (\lambda-\nu)
\qquad
(f\in \mathcal{C}(P^+_c),\lambda\in P^+_c) ,
\end{equation}
\end{subequations}
respectively.
For the root systems of type $A$, i.e. with the (affine) Weyl group being equal to the (affine) permutation group,
the operator $L_\omega$ \eqref{Lminuscule-red} was previously found in \cite[Prop. 2.2]{die:diagonalization}.
From this perspective, our present construction provides a Hecke-algebraic framework for the discrete integrable Laplacians in Ref. \cite{die:diagonalization}  permitting their generalization from the affine permutation group to an arbitrary affine Weyl group.
\end{remark}

\begin{remark}\label{free-laplacians:rem}
For $\tau\to 1$, the intertwining operator $\mathcal{J}$ \eqref{int-op} becomes trivial:
\begin{equation}
(\mathcal{J}f)(\lambda)\stackrel{\tau\to 1}{=} f(\lambda )\qquad (f\in \mathcal{C}(P),\lambda\in P).
\end{equation}
The action of $L_\omega$ on $\mathcal{C}(P)$ reduces in this situation
therefore to that of $m_\omega(t)$ \eqref{free-laplacian}, viz.
\begin{equation}
(L_\omega f)(\lambda)\stackrel{\tau\to 1}{=} \sum_{\nu\in W_0 \omega} f(\lambda -\nu).
\end{equation}
This amounts to the action of a conventional Laplacian on $\mathcal{C}(P)$ (shifted by an additive constant such that the diagonal term vanishes).
\end{remark}

\section{Diagonalization: Periodic Macdonald Spherical Functions}\label{sec7}
In this section we diagonalize the commuting Laplacians $L_p$, $p\in \mathbb{C}[X]^{W_0}$
in $\mathcal{C}(P^+_c)$ with the aid of a basis of periodic Macdonald spherical functions. In the remaining of the paper, we will restrict attention to the following (repulsive) parameter regime
\begin{equation}\label{tau-rep}
-1<\tau^2 <1.
\end{equation}
(Actually this parameter restriction is relevant only from Subsection \ref{sec7.3} onwards.)

\subsection{Affine Macdonald spherical functions}
For a spectral parameter $\xi\in V$, the {\em affine Macdonald spherical function} $\Phi_\xi\in\mathcal{C}(P)$ is defined as
\begin{subequations}
\begin{equation}\label{msf}
\Phi_\xi := \mathcal{J}\phi_\xi\quad\text{with}\quad \phi_\xi := I(\mathbf{1}_0)e^{i\xi} ,
\end{equation}
where $e^{i\xi}\in \mathcal{C}(P)$ denotes the plane wave function $e^{i\xi}(\lambda ):=e^{i\langle\lambda ,\xi\rangle}$ ($\lambda\in P$) and
\begin{equation}\label{idem}
\mathbf{1}_0:= \sum_{v\in W_0} \tau_v T_v .
\end{equation}
\end{subequations}
For our purposes it is sufficient to restrict attention to affine Macdonald spherical functions
corresponding to a {\em regular} spectral parameter $\xi$ taken from
\begin{equation}\label{Vreg}
V_{\text{reg}}:=\{ \xi \in V\mid \langle \xi ,\alpha \rangle \not\in 2\pi\mathbb{Z},\, \forall \alpha\in R_0^+\} .
\end{equation}
A celebrated formula for the affine Hecke algebra element $\mathbf{1}_0X^\lambda\mathbf{1}_0$,
$\lambda\in P$ originating from the work of Macdonald
(see e.g. Refs. \cite[Thm. 1]{mac:spherical1}, \cite[(4.1.2)]{mac:spherical2} and Refs. \cite[Thm. 2.9(a)]{nel-ram:kostka}, \cite[Thm. 6.9]{par:buildings}), implies that
the function $\phi_\xi$, $\xi\in V_{\text{reg}}$ decomposes as the following linear combination of plane waves
\cite[Prop. 5.9]{die-ems:unitary}
\begin{subequations}
\begin{equation}\label{cf-decomposition}
  \phi_\xi=\sum_{v\in W_0} C(v\xi)e^{iv\xi } \qquad (\xi\in V_{\text{reg}}),
  \end{equation}
where
\begin{equation}\label{cfun}
 C(\xi):= \prod_{\alpha\in R^+_0} \frac{1-\tau_{\alpha^\vee}^2 e^{-i\langle \xi ,\alpha\rangle }}{1-e^{-i\langle \xi ,\alpha\rangle}} .
\end{equation}
\end{subequations}

The plane waves $e^{i\xi}$, $\xi\in V/2\pi Q^\vee$ are the characters of the weight lattice $P$. It is therefore clear that the plane waves $e^{i\xi_1},\ldots ,e^{i\xi_m}$ are linearly independent in $\mathcal{C}(P)$
iff they
belong to {\em distinct} wave vectors
$\xi_1,\ldots ,\xi_m$ on the compact torus $V/2\pi Q^\vee$.
In particular, for $\xi\in V_{\text{reg}}$ the plane waves $e^{iv\xi}$, $v\in W_0$ are linearly independent, because the stabilizer of $\xi\in V_{\text{reg}}$ inside the
affine Weyl group $W_{\hat{R}}=W_0\ltimes
  t(2\pi Q^\vee)$ of $\hat{R}:=R_0+2\pi\mathbb{Z}$
is trivial.
As a consequence, one concludes from the explicit plane-wave expansion in Eqs. \eqref{cf-decomposition}, \eqref{cfun} that for a regular spectral parameter the affine Macdonald spherical functions
$\Phi_{\xi_1},\ldots ,\Phi_{\xi_m}$ are linearly independent in $\mathcal{C}(P)$
iff the corresponding spectral values
$\xi_1,\ldots ,\xi_m$ belong to {\em distinct}
$W_{\hat{R}}$ orbits of $V_{\text{reg}}$.
In other words, a complete domain of regular spectral values parametrizing the affine Macdonald spherical functions in Eqs. \eqref{cf-decomposition}, \eqref{cfun}
is given by the fundamental domain $2\pi A^\vee$ of $V_{\text{reg}}$ with respect to the action of $W_{\hat{R}}$, where $A^\vee$ refers to the open alcove
 \begin{equation}\label{Avee}
A^\vee:=\{ \xi\in V \mid 0<\langle \xi ,\alpha\rangle< 1,\,\forall\alpha\in R_0^+\} .
 \end{equation}

\subsection{Periodicity}
Since $T_j\mathbf{1}_0=\tau_j\mathbf{1}_0$ for $j=1,\ldots ,n$, it is clear that
the affine Macdonald spherical function is $W_0$-invariant:
\begin{equation}\label{W0-invariance}
\hat{T}_j\Phi_\xi=\hat{T}_j\mathcal{J}\phi_\xi=
\mathcal{J}I_j\phi_\xi=\mathcal{J}\tau_j\phi_\xi=\tau_j\Phi_\xi ,
\end{equation}
so $s_j\Phi_\xi=\Phi_\xi$ for $j=1,\ldots ,n$.
The following proposition provides additional constraints on spectral parameter $\xi$ guaranteeing the affine Macdonald spherical function to be periodic with respect to translations in $t(cQ)$.

\begin{proposition}\label{BAE:prp}
For $\xi\in V_{\text{reg}}$ \eqref{Vreg},
the affine Macdonald spherical function $\Phi_\xi$ \eqref{msf}, \eqref{idem} lies in the
$W_R$-invariant subspace $\mathcal{C}(P)^{W_R}$ \eqref{WRinv1} provided the spectral parameter satisfies
the following algebraic system of equations of Bethe type
\begin{equation}\label{BAE}
 e^{ic\langle \xi,\nu\rangle}
=\prod_{\alpha\in R^+_0}  \Bigl( \frac{1-\tau_{\alpha^\vee}^2e^{i\langle \xi,\alpha \rangle }}{\tau_{\alpha^\vee}^2-e^{i\langle \xi,\alpha \rangle } }
\Bigr)^{\langle \nu, \alpha^\vee\rangle },\qquad \forall \nu\in Q .
\end{equation}
\end{proposition}
\begin{proof}
One deduces from Eqs. \eqref{WRinv2} and \eqref{W0-invariance} that the affine Macdonald spherical function is $W_R$-invariant provided
$\hat{T}_0\Phi_\xi=\tau_0\Phi_\xi$, or equivalently: $I_0\phi_\xi =\tau_0\phi_\xi$.
For $\xi\in V_{\text{reg}}$ \eqref{Vreg}, the plane wave decomposition in Eq. \eqref{cf-decomposition}
combined with the explicit action of $\tau_0 I_0$ on $e^{i\xi}$ (cf. Eqs. \eqref{Ia:action}--\eqref{Ja:action}):
\begin{equation*}
\tau_0 I_0 e^{i\xi}=\frac{\tau_0^2-1}{  1-e^{-i\langle \xi, \vartheta\rangle}    }e^{i\xi}
+\frac{1-\tau_0^2e^{-i\langle \xi, \vartheta\rangle}}{1-e^{-i\langle \xi, \vartheta\rangle}} e^{ic\langle \xi, \vartheta\rangle} e^{is_{\vartheta}\xi} ,
\end{equation*}
produces on the one hand that
\begin{align*}
\tau_0 I_0 \phi_\xi =
&\sum_{v\in W_0} \frac{\tau_0^2-1}{1-e^{-i\langle v\xi,\vartheta\rangle}} C(v\xi)e^{iv\xi}  \\
  & + \sum_{v\in W_0} \frac{1-\tau_0^2 e^{i\langle v\xi,\vartheta\rangle} }{1-e^{i\langle v\xi,\vartheta\rangle}} C(s_\vartheta v\xi) e^{ -ic\langle  v\xi,\vartheta\rangle}  e^{iv\xi} \quad (\xi\in V_{\text{reg}}).
\end{align*}
When comparing this expression with the corresponding plane wave decomposition of $\tau_0^2\phi_\xi$,
it readily follows---upon exploiting the linear independence
of the plane waves $e^{iv\xi}$, $v\in W_0$ when $\xi\in V_{\text{reg}}$---that the affine Macdonald spherical function is $W_R$-invariant for $\xi\in V_{\text{reg}}$ provided
\begin{equation*}
e^{i c\langle v\xi,\vartheta\rangle} =
\frac{C(s_{\vartheta}v\xi)}{C(v\xi)}
\frac{1-\tau_0^2e^{i\langle v\xi,\vartheta \rangle }}{\tau_0^2-e^{i\langle v\xi,\vartheta \rangle } },
\qquad \forall v\in W_0.
\end{equation*}
By substituting the product expansion for $C (\cdot )$ over $R_0^+$ (cf. \eqref{cfun}) and canceling the
common factors in the numerator and denominator:
\begin{align*}
\frac{C(s_\vartheta\xi)}{C(\xi)}&=
\prod_{\alpha\in R(s_\vartheta)}\frac{1-\tau_{\alpha^\vee}^2e^{i\langle \xi,\alpha \rangle }}{\tau_{\alpha^\vee}^2-e^{i\langle \xi,\alpha \rangle } } =
\prod_{\substack{\alpha\in R^+_0\\ \langle  \vartheta,\alpha^\vee\rangle>0  }}
\frac{1-\tau_{\alpha^\vee}^2e^{i\langle \xi,\alpha \rangle }}{\tau_{\alpha^\vee}^2-e^{i\langle \xi,\alpha \rangle } } \\
&=\frac{1-\tau_0^2e^{i\langle \xi,\vartheta \rangle }}{\tau_0^2-e^{i\langle \xi,\vartheta \rangle } }
\prod_{\alpha\in R^+_0\setminus\{ \vartheta\}}
\Bigl(\frac{1-\tau_{\alpha^\vee}^2e^{i\langle \xi,\alpha \rangle }}{\tau_{\alpha^\vee}^2-e^{i\langle \xi,\alpha \rangle } }
\Bigr)^{\langle \vartheta ,\alpha^\vee\rangle}
\end{align*}
(where it was used that $0\leq \langle \vartheta,\alpha^\vee\rangle\leq 1$ for all $\alpha\in R_0^+\setminus \{\vartheta\}$),
the relation for the spectral parameter at issue is rewritten in the form of the algebraic system given by Eq. \eqref{BAE} (first for $\nu\in W_0\vartheta$ and then for $\nu\in Q$ upon recalling that the short roots  generate $Q$ over $\mathbb{Z}$).
\end{proof}

The following explicit formula for the periodic Macdonald spherical function is immediate from the plane waves decomposition in Eqs. \eqref{cf-decomposition}, \eqref{cfun}, the $W_R$-invariance
in Proposition \ref{BAE:prp}, and the trivial action of the intertwining operator $\mathcal{J}$ \eqref{int-op} on functions supported in the fundamental domain $P_c^+$.
\begin{corollary}\label{msf-decomposition:cor}
For $\xi\in V_{\text{reg}}$ \eqref{Vreg} subject to the periodicity conditions in Proposition \ref{BAE:prp}, the affine Macdonald spherical function is given explicitly by
\begin{equation}\label{msf-decomposition}
  \Phi_\xi (\lambda)=\sum_{v\in W_0} C(v\xi)e^{i\langle v\xi,\lambda_+\rangle} \qquad (\lambda\in P),
 \end{equation}
 with $C(\xi )$ taken from Eq. \eqref{cfun}.
\end{corollary}

The algebraic system in Eq. \eqref{BAE} inherits from $\Phi_\xi$ the invariance with respect to the action of the affine Weyl group $W_{\hat{R}}$ on the spectral parameter $\xi$.
It constitutes a generalization of the Bethe Ansatz equations
in \cite[Prop. 3.2]{die:diagonalization} to the case of arbitrary Weyl groups in the spirit of
\cite[Thm. 2.6]{ems-opd-sto:periodic}.
A standard technique due to C.N. Yang and C.P. Yang \cite{yan-yan:thermodynamics} allows one to characterize
the solutions of such Bethe Ansatz equations in terms of the global minima of a family of strictly convex Morse functions
\cite{mat:many-body,gau:fonction,kor-bog-ize:quantum,dor:orthogonality}.
We will now detail the relevant solutions for the algebraic system of Bethe type equations in Eq. \eqref{BAE} by adapting Yang and Yang's method to the present context sticking closely to the approach in
\cite[Sec. 4]{die:diagonalization} and \cite[Secs. 2,\ 9-11]{ems-opd-sto:periodic}.
To avoid linear dependencies it suffices to consider only solutions
belonging to the fundamental alcove $2\pi A^\vee$ \eqref{Avee} (cf. Remark \ref{WR-symmetry} below).

\subsection{Solution of the Bethe type equations}\label{sec7.3}
For any $\mu\in P^\vee$, let $\mathcal{V}_\mu:V\to \mathbb{R}$ be a smooth auxiliary function of the form
\begin{subequations}
\begin{equation}\label{mf1}
\mathcal{V}_\mu(\xi)=\frac{c}{2}\langle \xi,\xi \rangle-2\pi \langle \rho^\vee+\mu,\xi\rangle +\sum_{\alpha\in R^+_0}\frac{2}{\langle\alpha,\alpha\rangle}\int_0^{\langle \xi,\alpha\rangle}v_\alpha(\text{x})\text{d} \text{x},
\end{equation}
where $\rho^\vee =\frac{1}{2}\sum_{\alpha\in R_0^+} \alpha^\vee$ (cf. Remark \ref{dual-cox:rem}) and
\begin{align}\label{mf2}
v_\alpha(\text{x})&:=(1-\tau_{\alpha^\vee}^4)\int_0^\text{x} \frac{\text{d}\text{y}}{1-2\tau_{\alpha^\vee}^2 \cos(\text{y})+\tau_{\alpha^\vee}^4} \\
& =2\arctan\Bigl(\frac{1+\tau_{\alpha^\vee}^2}{1-\tau_{\alpha^\vee}^2}\tan \bigl(\frac{\text{x}}{2}\bigr)\Bigr)=
i\log\Bigl(\frac{1-\tau_{\alpha^\vee}^2e^{i\text{x}}}{e^{i\text{x}}-\tau_{\alpha^\vee}^2}\Bigr) . \nonumber
\end{align}
\end{subequations}
Here the branches of $\arctan(\cdot )$ and $\log (\cdot)$ are assumed to be chosen such that
$v_\alpha (\text{x} )$ varies from $-\pi$ to $\pi$ as $\text{x}$ varies from $-\pi$ to $\pi$ (which corresponds to the principal
branch) and $v_\alpha$ is quasi-periodic: $v_\alpha (\text{x} +2\pi)=v_\alpha (\text{x})+2\pi$.  Our parameter restriction \eqref{tau-rep} moreover guarantees that the odd function
 $v_\alpha$ is smooth and strictly monotonously increasing on $\mathbb{R}$.
 Let us denote the directional derivative of $\mathcal{V}_\mu(\xi)$  in the direction $\eta\in V$ by $(\partial_\eta V_\mu)(\xi):=\langle(\nabla V_\mu ) (\xi), \eta\rangle$ (where $\nabla$ refers to the gradient) and write
$\mathcal{H}(\xi)\in \text{End}(V)$ for the ($\mu$-independent) Hessian of $\mathcal{V}_\mu$ at the point $\xi\in V$.
An explicit computation reveals that the associated quadratic form
$\mathcal{H}_{\eta ,\zeta}(\xi ):=\langle \mathcal{H}(\xi) \eta ,\zeta\rangle =(\partial_\eta\partial_{\zeta} \mathcal{V}_\mu)(\xi)$ characterizing this Hessian is given by:
\begin{equation}\label{hessian}
\mathcal{H}_{\eta ,\zeta}(\xi )=
c\langle \eta,\zeta\rangle +
\sum_{\alpha\in R_0} \frac{ (1-\tau_{\alpha^\vee}^4)\langle \eta,\alpha\rangle  \langle \zeta,\alpha\rangle }{\langle\alpha ,\alpha\rangle(1-2\tau_{\alpha^\vee}^2\cos(\langle \xi,\alpha\rangle)+ \tau_{\alpha^\vee}^4 )} \quad   (\xi,\eta,\zeta\in V).
\end{equation}
It is manifest from this explicit expression that $\mathcal{H}_{\eta ,\eta}(\xi )>0$ for $\eta \neq 0$, whence $\mathcal{H}(\xi)$ is positive definite and $V_\mu :V\to\mathbb{R}$ is a strictly convex Morse function. For any $\mu\in P^\vee$, we now define $\xi_\mu$ as the unique global minimum of $\mathcal{V}_\mu$ \eqref{mf1}, \eqref{mf2}.  The existence of such a minimum is guaranteed since
 $\mathcal{V}_\mu (\xi )\to +\infty$ when $\xi \to \infty$ (as $\int_0^{\text{x}} v_\alpha (\text{y})\text{d}\text{y}\geq 0$ for $\text{x}\in \mathbb{R}$).
Moreover, because $\xi_\mu$ is a critical point of $\mathcal{V}_\mu$, it provides the unique solution of the critical equation
$\nabla \mathcal{V}_\mu =0$, or more explicitly:
\begin{equation}\label{critical:eq}
c \xi_\mu+\rho^\vee_v(\xi_\mu)
 =2\pi( \rho^\vee+\mu)
 \qquad \text{where}\ \rho^\vee_v(\xi):= \sum_{\alpha\in R_0^+}v_\alpha(\langle  \xi , \alpha\rangle)\alpha^\vee .
\end{equation}

\begin{proposition}\label{spectrum:prp}
\begin{subequations}
For a coweight $\mu$ belonging to the fundamental region
 \begin{equation}\label{fr}
 P^{\vee,+}_c:=\{ \mu\in P^\vee\mid 0\leq \langle \mu ,\alpha\rangle \leq c,\, \forall \alpha\in R_0^+\} ,
 \end{equation}
the unique global minimum $\xi_\mu$ of $\mathcal{V}_\mu$ \eqref{mf1}, \eqref{mf2} enjoys the
following properties:
\begin{itemize}
\item[$i)$.] The point $\xi_\mu$ provides a solution to the Bethe type equations \eqref{BAE}.
\item[$ii)$.] The parametrization $\mu\mapsto\xi_\mu$, $\mu\in P^{\vee,+}_c$ is injective.
\item[$iii)$.] The minimum $\xi_\mu$ belongs the intersection of the open alcove $2\pi A^\vee$ \eqref{Avee} with the compact convex polytope in $V$ cut out by the following (moment gap) inequalities:
\begin{equation}\label{momentgaps}
\frac{2\pi \langle \rho^\vee+\mu ,\alpha\rangle }{c+\kappa_-(\tau^2)}
\le \langle \xi_\mu,\alpha\rangle\le
\frac{2\pi \langle \rho^\vee+\mu,\alpha\rangle }{c+\kappa_+(\tau^2)}
\qquad \forall\alpha\in R^+_0 ,
  \end{equation}
where
\begin{equation}\label{kappa}
  \kappa_{\pm }(\tau^2):= \frac{2}{n}\sum_{\alpha\in R_0^+}
   \frac{1-\tau_{\alpha^\vee}^4}{(1\pm |\tau_{\alpha^\vee}^2|)^2}
  .
 \end{equation}
 \item[$iv)$.] The position of $\xi_\mu$ depends analytically on the parameter(s) $\tau_{\alpha^\vee}^2\in (-1,1)$, and one has that
 \begin{equation}\label{tau=01}
\lim_{\tau\to 0}\xi_\mu  =\frac{2\pi }{c+h(R_0)} (\rho^\vee+\mu) \quad\text{and}\quad
\lim_{\tau\uparrow 1}\xi_\mu  =2\pi c^{-1} \mu ,
\end{equation}
where---recall---$h(R_0)$ denotes the Coxeter number of $R_0$ (cf. Eq. \eqref{coxeter}).
\end{itemize}
\end{subequations}
\end{proposition}

\begin{proof}

$i)$.
Upon pairing Eq. \eqref{critical:eq} with an arbitrary positive root:
\begin{equation}\label{critical-beta:eq}
c \langle \xi_\mu,\beta\rangle+\langle \rho^\vee_v(\xi_\mu),\beta\rangle
 =2\pi\langle  \rho^\vee+\mu ,\beta\rangle \qquad (\beta\in R_0^+),
\end{equation}
multiplication by $i$ and exponentiation of both sides (while using that
$e^{iv_\alpha (\text{x})}=(e^{i\text{x}}-\tau_{\alpha^\vee}^2)/(1-\tau_{\alpha^\vee}^2e^{i\text{x}})$)
reproduces the Bethe type equations in Eq. \eqref{BAE} evaluated at $\xi=\xi_\mu$ (first for $\nu=\beta$
and then for arbitrary $\nu\in Q$ as the root lattice is generated over $\mathbb{Z}$ by $R_0^+$).

$ii)$. The injectivity of the parametrization $\mu\mapsto\xi_\mu$ is
immediate from the observation that Eq. \eqref{critical:eq} permits recovering $\mu$ from $\xi_\mu$.

$iii)$. Using that $v_\alpha (\text{x})$ is an odd function, one may rewrite $\langle \rho^\vee_v(\xi_\mu),\beta\rangle$ in Eq. \eqref{critical-beta:eq} as
\begin{align}
\langle \rho^\vee_v(\xi_\mu),\beta\rangle & = \frac{1}{2}\sum_{\alpha\in R_0}
v_\alpha(\langle  \xi_\mu, \alpha\rangle) \langle\alpha^\vee,\beta \rangle  \nonumber\\
& =
\frac{1}{2}\sum_{\substack{\alpha\in R_0 \\ \langle \alpha^\vee,\beta \rangle >0}}
\bigl( v_\alpha(\langle  \xi_\mu, \alpha\rangle)-v_\alpha (\langle  s_\beta \xi_\mu, \alpha\rangle) \bigr) \langle\alpha^\vee,\beta \rangle . \label{rho-vb}
\end{align}
Because $v_\alpha (\text{x})$ is strictly monotonously increasing and
\begin{equation}\label{pos}
\langle \xi_\mu,\alpha\rangle-\langle s_\beta \xi_\mu,\alpha\rangle
=
\langle\xi_\mu,\beta\rangle\langle\alpha,\beta^\vee\rangle  ,
\end{equation}
one reads-off from Eq. \eqref{critical-beta:eq} with $\langle \rho^\vee_v(\xi_\mu),\beta\rangle$
in the form given by Eq. \eqref{rho-vb}
that the sign of $\langle \xi_\mu,\beta\rangle$ must be the same as that of
$\langle \rho^\vee+\mu ,\beta\rangle$, i.e. $\langle \xi_\mu,\beta\rangle >0$ for $\mu\in P^{\vee,+}_c$.
Since $\beta\in R_0^+$ was arbitrary, to confirm that $\xi_\mu$ lies in the fundamental alcove $2\pi A^\vee$ \eqref{Avee} it only remains to verify that
$\langle \xi_\mu ,\varphi\rangle < 2\pi$ (where---recall---$\varphi$ denotes the highest root of $R_0$, cf. Remark \ref{dual-cox:rem}).
To this end we will infer that the inequality $\langle \xi_\mu ,\varphi\rangle \geq 2\pi$ would imply that
$\langle \mu,\varphi \rangle > c$, which contradicts our assumption that $\mu\in P^{\vee,+}_c$ \eqref{fr}.
Indeed, from Eqs. \eqref{rho-vb}, \eqref{pos} with $\beta=\varphi$ and the quasi-periodicity of the strictly increasing function $v_\alpha (\text{x})$ one deduces that for $\langle \xi_\mu ,\varphi\rangle \geq 2\pi$:
\begin{equation}\label{lhs-bound}
c\langle \xi_\mu,\varphi\rangle + \langle \rho_v^\vee (\xi_\mu),\varphi\rangle\geq
2\pi c +\pi
\sum_{\alpha\in R_0^+} \langle \alpha,\varphi^\vee\rangle \langle \varphi,\alpha^\vee\rangle
 \stackrel{\text{Rem}.~\ref{schur:rem}}{=} 2\pi (c+h(R_0) )
\end{equation}
(where we used in passing that $\varphi\in P^+$ and that $h(R_0)=|R_0|/n$). By combining the inequality in Eq. \eqref{lhs-bound}  with Eq. \eqref{critical-beta:eq} for $\beta=\varphi$, we conclude that in this situation
$$c+h(R_0)\leq \langle \rho^\vee +\mu ,\varphi\rangle
 =\langle \mu ,\varphi\rangle + h(R_0)-1,$$
 i.e. $\langle \mu ,\varphi\rangle>c$ as announced.

It remains to verify the (moment gap) bounds in Eqs. \eqref{momentgaps}, \eqref{kappa}. Upon writing
\begin{equation*}
v_{\alpha}(\langle \xi,\alpha \rangle) - v_{\alpha}(\langle s_\beta\xi, \alpha\rangle)=
\int_{\langle \xi, \alpha \rangle}^{\langle s_\beta\xi,\alpha \rangle} \frac{(1-\tau_{\alpha^\vee}^4)\, \text{d}\text{y}}{1-2\tau_{\alpha^\vee}^2 \cos(\text{y})+\tau_{\alpha^\vee}^4}
  \end{equation*}
and observing that the integrand stays bounded between
$(1-\tau_{\alpha^\vee}^4)/(1+ |\tau_{\alpha^\vee}^2|)^2 $ and $(1-\tau_{\alpha^\vee}^4)/(1- |\tau_{\alpha^\vee}^2|)^2$, one readily
infers from Eqs. \eqref{rho-vb}, \eqref{pos} that
\begin{equation}\label{rho-v-bound}
 \kappa_+(\tau^2) \langle \xi_\mu ,\beta\rangle\leq  \langle \rho_v^\vee (\xi_\mu ),\beta\rangle\leq \kappa_- (\tau^2)\langle \xi_\mu ,\beta\rangle \quad (\mu\in P^{\vee ,+},\, \beta\in R_0^+),
\end{equation}
where it was used that
  \begin{eqnarray*}
\lefteqn{\frac{1}{2} \sum_{\substack{\alpha\in R_0\\ \langle\beta, \alpha^\vee\rangle >0 }} \frac{1-\tau_{\alpha^\vee}^4}{(1\pm |\tau_{\alpha^\vee}^2|)^2}
\langle \beta,\alpha^\vee\rangle \langle \alpha,\beta^\vee\rangle = }&& \\
 &&\frac{1}{2}
\sum_{\alpha\in R_0^+ } \frac{1-\tau_{\alpha^\vee}^4}{(1\pm |\tau_{\alpha^\vee}^2|)^2}
\langle \beta,\alpha^\vee\rangle \langle \alpha,\beta^\vee\rangle
  \stackrel{\text{Rem}.~\ref{schur:rem}}{=}\frac{1}{n}
  \sum_{\alpha\in R_0 } \frac{1-\tau_{\alpha^\vee}^4}{(1\pm |\tau_{\alpha^\vee}^2|)^2}  =\kappa_{\pm}(\tau^2) .
  \end{eqnarray*}
Combination of Eqs. \eqref{critical-beta:eq} and \eqref{rho-v-bound} now entails the desired (moment gap) inequalities in Eqs. \eqref{momentgaps}, \eqref{kappa}.

$iv)$. The integrand of $v_\alpha (\text{x})$ \eqref{mf2} is analytic in $\tau_{\alpha^\vee}^2\in (-1,1)$, and thus so is the function $\mathcal{V}_\mu (\xi)$ \eqref{mf1} and its critical equation \eqref{critical:eq} determining
the global minimum $\xi_\mu$. Since the Jacobian of the critical equation \eqref{critical:eq} with respect to $\xi$ amounts to the positive definite Hessian $\mathcal{H} (\xi)$ \eqref{hessian} of $\mathcal{V}_\mu (\xi)$, its determinant  is nonzero (viz. positive). Straightforward application of the
implicit function theorem therefore yields that the solution $\xi_\mu$ of the critical equation must depend analytically on
$\tau_{\alpha^\vee}^2\in (-1,1)$. The first limit in Eq. \eqref{tau=01} is now immediate from the (moment gap) inequalities \eqref{momentgaps}, \eqref{kappa} and the observation that $\lim_{\tau\to 0} \kappa_\pm (\tau^2 )=|R_0|/n =h(R_0)$. For $\tau\uparrow 1$, the function $v_\alpha (\text{x})$ \eqref{mf2} tends pointwise to the following staircase function
$$
v (\text{x}):=
\begin{cases} 2\pi\, \text{sign}(x)\left( \left\lceil \frac{|\text{x}|}{2\pi}\right\rceil -\frac{1}{2}\right)&\text{if}\ \text{x}\in\mathbb{R}\setminus 2\pi\mathbb{Z} ,\\
\text{x}&\text{if}\ \text{x}\in 2\pi\mathbb{Z},
\end{cases}
$$
where $\lceil \cdot \rceil$ refers to the ceiling function rounding up to the next higher integer. Since $v (\text{x})=\pi$ for $0<\text{x}<2\pi$,
it follows that on the compact closure
$2\pi \overline{A^\vee}=\{ \xi\in V\mid 0\leq \langle \xi ,\alpha\rangle \leq 2\pi,\ \forall\alpha\in R^+_0\}$  the function $\mathcal{V}_\mu (\xi )$ \eqref{mf1}, \eqref{mf2} converges in this limit uniformly to the strictly convex Morse function
$
\frac{c}{2}\langle \xi,\xi \rangle-2\pi \langle \mu,\xi\rangle ,
$
which has a unique global minimum given by $2\pi c^{-1}\mu$ (belonging to $2\pi \overline{A^\vee}$ if $\mu\in P^{\vee ,+}_c$ \eqref{fr}).
The upshot is that $\mathcal{V}_\mu$ extends to a continuous function
of $(\xi ,\tau)\in 2\pi \overline{A^\vee}\times \mathcal{T}$ with $\mathcal{T}=\{ -1<\tau ^2< 1\}\cup \{ \tau=1\}$.
Invoking of \cite[Lem. 4]{ems:completeness} (with $X=2\pi \overline{A^\vee}$, $Y=\mathcal{T}$, and $f=\mathcal{V}_\mu$)
now ensures that for $\mu\in P^{\vee,+}_c$ the global minimum $\xi_\mu\in 2\pi A^\vee$ of $\mathcal{V}_\mu $ with $-1<\tau^2<1$ converges for $\tau\uparrow 1$ to the global minimum $2\pi c^{-1}\mu\in 2\pi \overline{A^\vee}$ of $\lim_{\tau\uparrow 1}\mathcal{V}_\mu$.
\end{proof}

\begin{remark}\label{schur:rem} In the proof of part $iii)$ of Proposition \ref{spectrum:prp} it was used that for {\em any} root multiplicative function $\tau$ and root $\beta\in R_0$:
\begin{equation}\label{schur:eq}
\sum_{\alpha\in R_0^+} \tau_{\alpha^\vee} \langle \beta,\alpha^\vee  \rangle \langle\alpha ,\beta^\vee\rangle =\frac{2}{n}\sum_{\alpha\in R_0} \tau_{\alpha^\vee}
\end{equation}
(cf. the proof of \cite[\text{Lem}.~10.1]{ems-opd-sto:periodic}). Indeed, the $W_0$-invariant linear map $A_\tau:V\to V$ defined
by $A_\tau x:=\sum_{\alpha\in R_0^+}\tau_{\alpha^\vee} \langle x, \alpha^\vee \rangle\alpha$ ($x\in V$) is a constant multiple of the identity
by Schur's lemma and the irreducibility of the representation of $W_0$ on $V$. A computation of the trace of $A_\tau$ reveals
that the proportionality constant at issue is equal to
\begin{align*}
c_\tau &=\frac{1}{n}\sum_{j=1}^n \langle A_\tau e_j, e_j \rangle  =\frac{1}{n} \sum_{\alpha\in R_0^+} \tau_{\alpha^\vee} \sum_{j=1}^n \langle e_j,\alpha^\vee\rangle \langle e_j,\alpha\rangle \\
&=\frac{1}{n}\sum_{\alpha\in R_0^+} \tau_{\alpha^\vee} \langle \alpha^\vee,\alpha\rangle
=\frac{1}{n}\sum_{\alpha\in R_0} \tau_{\alpha^\vee} ,
\end{align*}
whence $\langle A_\tau \beta ,\beta^\vee\rangle = c_\tau \langle \beta,\beta^\vee\rangle=2c_\tau$ (which amounts to Eq. \eqref{schur:eq}).
\end{remark}

\subsection{Diagonalization}
For
$p=\sum_{\lambda\in P} p_\lambda X^\lambda\in\mathbb{C}[X]$ (so only a finite number of the complex coefficients $p_\lambda$, $\lambda\in P$ are nonzero), we define
\begin{equation}
p(e^{i\xi}):=\sum_{\lambda\in P} p_\lambda e^{i\xi}(\lambda) = \sum_{\lambda\in P} p_\lambda e^{i\langle \lambda ,\xi \rangle} \in \mathcal{P}(V/2\pi Q^\vee ),
\end{equation}
where $\mathcal{P}(V/2\pi Q^\vee )$ denotes the algebra of trigonometric polynomials on $V$ with period lattice $2 \pi Q^\vee$. Clearly the assignment $p\mapsto p(e^{i\xi})$ defines an algebra isomorphism between $\mathbb{C}[X]$ and
$\mathcal{P}(V/2\pi Q^\vee )$.

The main theorem of this paper affirms that
the periodic Macdonald spherical functions constitute a complete basis of eigenfunctions for the Laplace operator
$L_p$ with eigenvalues given by the evaluations of $p(e^{-i\xi})$ at the minima $\xi_\mu$ corresponding to the coweights in the fundamental region $P^{\vee,+}_c$ \eqref{fr}.

\begin{theorem}\label{diagonal:thm}
The periodic Macdonald spherical functions $\Phi_{\xi_\mu}$, $\mu\in P^{\vee,+}_c$ form a basis of $\mathcal{C}(P_c^+)$
consisting of joint eigenfunctions for the commuting Laplacians $L_p$ \eqref{Lp}:
\begin{equation}
L_p \Phi_{\xi_\mu} = p(e^{-i\xi_\mu }) \Phi_{\xi_\mu} \qquad (p\in \mathbb{C}[X]^{W_0},\, \mu\in P^{\vee,+}_c).
\end{equation}
\end{theorem}
\begin{proof}
For any $\xi\in V_{\text{reg}}$ and $\omega\in P^+$, the affine Macdonald spherical function $\Phi_\xi$ \eqref{msf}, \eqref{idem}
satisfies the eigenvalue equation for $L_\omega=L_{m_\omega}$:
\begin{eqnarray*}
\lefteqn{L_\omega \Phi_\xi = m_\omega (\hat{X})\mathcal{J} \phi_\xi = \hat{T}(m_\omega (X))\mathcal{J} \phi_\xi  =} && \\ && \mathcal{J} I(m_\omega (X))\phi_\xi
=\mathcal{J} m_\omega (t)\phi_\xi =  m_\omega (e^{-i\xi}) \Phi_\xi ,
\end{eqnarray*}
where the last step hinges on the plane wave expansion in Eqs. \eqref{cf-decomposition}, \eqref{cfun} and the explicit action of the free Laplacian $m_\omega (t)$ \eqref{free-laplacian} on plane waves:
\begin{equation*}
m_\omega (t) e^{iv\xi}= \sum_{\nu\in W_0\omega}t_\nu e^{iv\xi} =
m_\omega (e^{-i\xi}) e^{iv\xi}\qquad (v\in W_0) .
\end{equation*}
The explicit formula of the periodic
Macdonald spherical function in Corollary \ref{msf-decomposition:cor}
confirms---upon evaluation at $\lambda=0$---that for
$\xi$ satisfying the Bethe type equations in Eq. \eqref{BAE}
we indeed arrive this way at a {\em nonvanishing} eigenfunction in
$\mathcal{C}(P)^{W_R}\cong \mathcal{C}(P^+_c)$:
\begin{align*}
\Phi_{\xi} (0)& = \sum_{v\in W_0} C(v\xi )=\sum_{v\in W_0} \prod_{\alpha\in R^+_0} \frac{1-\tau_{\alpha^\vee}^2 e^{-i\langle v\xi ,\alpha\rangle }}{1-e^{-i\langle v\xi ,\alpha\rangle}} \\
&=
\sum_{v\in W_{0}}\tau_v^2 = \prod_{\alpha\in R_0^+}
\frac{1-\tau^2_{\alpha^\vee} e_\tau (\alpha^\vee ) }{1-e_\tau (\alpha^\vee )}> 0 ,
\end{align*}
where the equalities on the second line rely on Macdonald's well-known identity from Ref.
\cite[\text{Thm.} (2.8)]{mac:poincare} and Macdonald's product formula \eqref{poincare-stab}, \eqref{htau} for $\lambda=0$.
Since the symmetric monomials $m_\omega (X)$, $\omega\in P^+$ form a basis of $\mathbb{C}[X]^{W_0}$, and the algebra of $W_0$-invariant trigonometric polynomials $p(e^{-i\xi})$ on the torus $\mathcal{P}(V/2\pi Q^\vee)$ separates the points of the fundamental alcove $2\pi A^\vee$ \eqref{Avee}, it follows that distinct solutions of the Bethe type equations
belonging to $2\pi A^\vee$ give rise to linearly independent eigenfunctions.
As a consequence, the eigenfunctions $\Phi_{\xi_\mu}$, $\mu\in P^{\vee,+}_c$ associated with the Bethe solutions in Proposition \ref{spectrum:prp} form a complete basis of $\mathcal{C}(P_c^+)$ because
$\dim \mathcal{C}(P_c^+)=|P^{+}_c|=|P^{\vee ,+}_c|$ (cf. Remark \ref{dimensions:rem} below).
\end{proof}

In particular, for the simplest Laplace operator $L_\omega$ with $\omega$ (quasi-)minuscule given by Theorems \ref{action:thm} and \ref{action-symmetric:thm} the diagonalization in Theorem \ref{diagonal:thm} boils down to the following spectral decomposition.

\begin{corollary}\label{diagonal:cor}
For $\omega\in P^+$ (quasi-)minuscule, the eigenvalues and eigenfunctions of the integrable Laplacian $L_\omega $ \eqref{Lomega-sym}--\eqref{Ulambdaomega} in the space $\mathcal{C}(P_c^+)$ are given by $m_\omega (e^{-i\xi_\mu })$ and $\Phi_{\xi_\mu}$, $\mu\in P^{\vee,+}_c$.
\end{corollary}

\begin{remark}\label{dimensions:rem}
Since the fundamental region $P_c^+$  \eqref{Pc} consists of all nonnegative integral combinations of the fundamental weights  $k_1\omega_1+\cdots +k_n\omega_n$ such that
$k_1m_1+\cdots +k_nm_n\leq c$ (where the integers
$m_1,\dots ,m_n$ refer to the coefficients of the highest coroot
$\vartheta^\vee$ in the simple basis $\alpha_1^\vee,\ldots,\alpha_n^\vee$ of $R_0^\vee$, cf.
Eq. \eqref{highest-coroot}), it is elementary that the dimensions
$\dim \mathcal{C}(P_c^+) =|P_c^+|$, $c\in\mathbb{N}$ can be computed from the generating function
\begin{equation}
1+\sum_{c> 0} \dim \mathcal{C}(P_c^+)\, q^c =(1-q)^{-1}\prod_{j=1}^n (1-q^{m_j})^{-1} \qquad (|q|<1).
\end{equation}
One reads-off from this generating function that the dimensions in question are invariant when replacing $R_0$ by its dual root system $R_0^\vee$, i.e. $|P_c^+|=|P^{\vee ,+}_c|$. For root systems other than of type $B$ or $C$ this is obvious because then
$R_0$ and $R_0^\vee$ are isomorphic. The root systems of type $B$ and $C$ on the other hand are dual to each other
and share the property that $1\leq m_j\leq 2$ for $j=1,\ldots ,n$ with the lower bound being reached only {\em once}, so their generating functions coincide.
\end{remark}

\begin{remark}\label{WR-symmetry}
It is immediate from the proof of Proposition \ref{spectrum:prp} that in fact for {\em any} $\mu\in P^\vee$ the global minimum $\xi_\mu$ of $\mathcal{V}_\mu$ \eqref{mf1}, \eqref{mf2}
provides a solution to the Bethe type equations \eqref{BAE} in $V$ and that the assignment $\mu\to \xi_\mu$, $\mu\in P^\vee$ is still injective.
This injection turns out to be equivariant with respect to two actions of the affine Weyl group $W_{\hat{R}}$
 generated by the orthogonal reflections across the walls of $2\pi A^\vee$:
\begin{equation} \label{equivariance}
 w\xi_{\mu}=\xi_{ w\cdot\mu} \quad  (w\in {W}_{\hat{R}},\, \mu\in P^\vee),
\end{equation}
where ${w}$ acts on the LHS via the standard action of $W_{\hat{R}}$ on $V$ and on the RHS via a `dot action'
$$
v\cdot x:=v(x+\rho^\vee)-\rho^\vee,  \quad t_{2\pi \nu}\cdot x:=x+(c+h(R_0))\nu
$$
($x\in V$, $v\in W_0$, $\nu\in Q^\vee$) for which $P^\vee$ is manifestly stable. Indeed, the equivariance in Eq. \eqref{equivariance} is a consequence of the following $W_{\hat{R}}$-equivariance of the gradient of the
Morse function $\mathcal{V}_\mu$:
$$ \nabla \mathcal V_{ w\cdot\mu}( w \xi)=  w'\nabla \mathcal V_\mu(\xi),
\quad (\xi\in V,\,  w\in W_{\hat{R}},\, \mu\in P^\vee ),
$$
which is readily inferred from Eq. \eqref{critical:eq} upon using that
$\rho^\vee_v({w}\xi)={w}\rho^\vee_v(\xi)$ for ${w}\in W_0$ (since $v_\alpha$ is an odd function) and
$\rho^\vee_v(\xi+2\pi \nu)=\rho^\vee_v(\xi) + 2\pi h(R_0) \nu$ for $\nu\in Q^\vee$
(by the quasi-periodicity of $v_\alpha$, Remark \ref{schur:rem}, and the fact that $h(R_0)=|R_0|/n$).
It follows from Eq. \eqref{equivariance} and the injectivity that
$\xi_\mu\in V_{\text{reg}}$ iff
$\langle \rho^\vee+\mu,\alpha \rangle\not\in (c+h(R_0))\mathbb Z$ for all $\alpha\in R_0$ (i.e. iff $\mu\in P^\vee$ is regular with respect to the `dot action' of $W_{\hat{R}}$).
The parametrization $\mu\to\xi_\mu$, $\mu\in P^{\vee ,+}_c$ in Proposition \ref{spectrum:prp} arises by restricting the injection on $P^\vee$
to the regular elements of fundamental domains for the corresponding $W_{\hat{R}}$-actions.
\end{remark}

\begin{remark}
It follows from the proof of Theorem \ref{diagonal:thm} that for {\em any} spectral parameter value $\xi\in V_{\text{reg}}$
the affine Macdonald spherical function $\Phi_\xi$ \eqref{msf}, \eqref{idem} is a joint eigenfunction of the Laplacians $L_p$ \eqref{Lp} in $\mathcal{C}(P)^{W_0}$,  but---in view of Proposition \ref{BAE:prp}---only for $\xi$ satisfying the Bethe type equations \eqref{BAE} it restricts to an eigenfunction in $\mathcal{C}(P_c^+)\cong\mathcal{C}(P)^{W_R}$. The completeness of the eigenbasis in Theorem \ref{diagonal:thm} implies that apart from the solutions detailed in Proposition \ref{spectrum:prp} there are no other solutions of the Bethe type equations inside the fundamental alcove $2\pi A^\vee$.
\end{remark}

\begin{remark}
For $p=m_\omega$ with $\omega\in P^+$ (quasi-)minuscule, the eigenvalue equation in Theorem \ref{diagonal:thm} amounts
to that of a free discrete Laplacian:
\begin{subequations}
\begin{equation}\label{fl}
\sum_{\nu\in W_0\omega} \psi_\xi (\lambda-\nu) = m_\omega (e^{-i\xi}) \psi_\xi (\lambda)\qquad (\lambda\in P_c^+),
\end{equation}
subject to the boundary conditions
\begin{equation}\label{bc}
\psi_\xi (\lambda-\nu)=
\begin{cases}
\tau_j^2 \psi_\xi (s_j(\lambda-\nu)) &\text{if}\ a_j(\lambda-\nu)=-1 \\
\tau_j^2 \psi_\xi (s_j(\lambda-\nu))+(\tau_j^2-1)\psi (\lambda)  &\text{if}\ a_j(\lambda-\nu)=-2
\end{cases},
\end{equation}
\end{subequations}
$0\leq j\leq n$.  Indeed, it was shown in \cite[\text{Sec}.~4]{die:plancherel} that the free Laplacian on $\mathcal{C}(P^+)$ with boundary conditions of the form \eqref{bc} for $j=1,\ldots ,n$ is diagonalized by a Bethe Ansatz wave function of the form in Eqs. \eqref{cf-decomposition}, \eqref{cfun}. For $\lambda\in P^+_c$, this Bethe Ansatz wave function satisfies moreover the boundary condition \eqref{bc} for $j=0$ provided the spectral parameter $\xi$ solves the Bethe Ansatz equations in Proposition \ref{BAE:prp} (cf. also \cite[\text{Prp}.~3.2]{die:diagonalization}).
\end{remark}

\begin{remark} It is immediate from the proof of Proposition \ref{WR-reduction:prp} that its analog for the extended affine Weyl group is also valid:
the Laplacians $L_p$, $p\in \mathbb{C}[X]^{W_0}$ map the $W$-invariant subspace
\begin{align*}
\mathcal{C}(P)^W &=\{ f\in\mathcal{C}(P)\mid wf=f,\, w\in W\}\\
                 & =\{ f\in \mathcal{C}(P)\mid \hat{T}_wf=\tau_wf,\, w\in W\}
\end{align*}
into itself. The corresponding Bethe-Ansatz equations guaranteeing the $W$-invariance of the Macdonald spherical function $\Phi_\xi$ \eqref{msf}, \eqref{idem} are of the form in Eq. \eqref{BAE} with $\nu\in P$. By adapting the argument in the proof of part $i)$ of Proposition \ref{spectrum:prp} it is seen that
this $\Omega$-invariant subspace of  $\mathcal{C}(P)^{W_R}$ is spanned by the eigenfunctions
$\Phi_{\xi_\mu}$, $\mu\in P^{\vee ,+}_c\cap Q^\vee$.
\end{remark}

\section{Hilbert Space Structure}\label{sec8}
In this section we endow the function space over our weight lattice $P$ with an appropriate Hilbert space structure and study the unitarity of the difference-reflection representation $\hat{T}(H)$, the (self-)adjointness of the Laplacian $L_\omega$, and the orthogonality of the periodic Macdonald spherical functions $\Phi_{\xi_\mu}$ in this Hilbert space.
For this purpose we will further restrict to the positive parameter regime
\begin{equation}
0<\tau^2<1
\end{equation}
from now on (but see Remark \ref{parameter-domain:rem} below).

\subsection{Hilbert space}
Let us define $l^2(P,\delta )$ to be the Hilbert space of functions
 $\{ f\in C(P)\mid \langle f,f\rangle_\delta <\infty \}$ associated with the inner product
\begin{subequations}
\begin{equation}\label{ip-delta}
\langle f ,g\rangle_\delta :=\sum_{\lambda\in P}  f(\lambda) \overline{g(\lambda)} \delta_\lambda \qquad
(f,g\in l^2(P,\delta )) ,
\end{equation}
where
\begin{equation}\label{delta}
\delta_\lambda := W_R^{-1}(\tau^2)\, \tau^2_{w_\lambda}  \quad \text{and} \ W_R (\tau^2):=\sum_{w\in W_R}\tau_w^2.
\end{equation}
\end{subequations}
Here the normalization of the orthogonality measure $\delta:P\to (0,1)$ is chosen such that it restricts to a probability measure on the $W_R$-orbits of the regular weights:  $\sum_{\mu\in W_R\lambda}\delta_\mu \leq 1$ for $\lambda\in P$ with equality holding when $|W_{R,\lambda}|=1$ (cf. Eq. \eqref{Delta} below).
It follows from \cite[\text{Thm.} (3.3)]{mac:poincare} that the relevant normalization factor given by the generalized Poincar\'e series $W_R(\tau^2)$ may be alteratively rewritten in product form as (cf. Eqs. \eqref{poincare-stab}, \eqref{htau})
\begin{align*}
W_R(\tau^2) &= (1-h_\tau)^{-1}\prod_{\substack{a\in R^+\\ \text{ht}(a)<h(R_0)}}\frac{1-\tau_a^2 \tau_s^{2\text{ht}_s(a)} \tau_l^{2\text{ht}_l(a)} }{1-\tau_s^{2\text{ht}_s(a)} \tau_l^{2\text{ht}_l(a)}} \\
\intertext{where $\text{ht}(a):=\text{ht}_s(a)+\text{ht}_l(a)$, or equivalently}
W_R(\tau^2)&=(1-h_\tau)^{-1} \prod_{\alpha\in R_0^+}
\Bigl(\frac{1-\tau^2_{\alpha^\vee} e_\tau (\alpha^\vee ) }{1-e_\tau (\alpha^\vee )}\Bigr)
\Bigl(\frac{1-\tau^2_{\alpha^\vee} h_\tau e_\tau (-\alpha^\vee )}{1-h_\tau e_\tau (-\alpha^\vee )}\Bigr)
\end{align*}
(though we will not actually use these product formulas here).
Since the total mass of $\delta$ is finite, all bounded functions in $\mathcal{C}(P)$ belong to $l^2(P,\delta )$. As such it is clear that our Hilbert space contains
the $W_R$-invariant space $\mathcal{C}(P)^{W_R}$ as a finite-dimensional linear subspace.
By partitioning the sum in Eq. \eqref{ip-delta} into orbit sums
\begin{align*}
\sum_{\lambda\in P}  f(\lambda) \overline{g(\lambda)} \delta_\lambda &=
\sum_{\lambda\in P^+_c} \sum_{\mu\in W_R\lambda} f(\mu) \overline{g(\mu)} \delta_\mu\qquad (\text{for}\ f,g\in l^2(P,\delta )) \\
&=
\sum_{\lambda\in P^+_c} f(\lambda) \overline{g(\lambda)} \sum_{\mu\in W_R\lambda}  \delta_\mu\qquad (\text{for}\ f,g\in \mathcal{C}(P)^{W_R}\subset l^2(P,\delta )),
\end{align*}
it is seen that---upon identifying $\mathcal{C}(P^+_c)$ with $\mathcal{C}(P)^{W_R}$ via the $W_R$-invariant embedding of $\mathcal{C}(P^+_c)$ into $\mathcal{C}(P)$---the inner product $\langle \cdot ,\cdot \rangle_\delta$ pulls back to an inner product on $\mathcal{C}(P^+_c)$ of the form
\begin{subequations}
\begin{equation}\label{ip-Delta}
\langle f,g\rangle_{ \Delta}:=\sum_{\lambda\in P^+_c} f(\lambda)\overline{g(\lambda)}\Delta_\lambda\qquad (f,g\in \mathcal{C}(P^+_c)),
\end{equation}
with
\begin{equation}\label{Delta}
\Delta_\lambda := \sum_{\mu\in W_R\lambda} \delta_\mu = \frac{1}{W_{R}(\tau^2)}\sum_{\mu\in W_R\lambda} \tau_{w_\mu}^2
=
\frac{1}{W_{R,\lambda}(\tau^2)} .
\end{equation}
\end{subequations}
In other words, the inner product $\langle \cdot ,\cdot\rangle_\Delta$ associated with the orthogonality measure $\Delta:P^+_c\to (0,1)$ turns $\mathcal{C}(P^+_c)$ into the finite-dimensional Hilbert space $l^2(P^+_c,\Delta)$ in such a way that its $W_R$-invariant embedding into $l^2(P,\delta )$ becomes an isometry.

\subsection{Unitarity and (self-)adjointness}
The extended affine Hecke algebra $H$ carries a natural antilinear involution $*:H\to H$ determined by
\begin{equation}
T_w^*:=T_{w^{-1}}\qquad (w\in W),
\end{equation}
which turns it into a $*$-algebra. The difference-reflection representation $\hat{T}(H)$ restricts to a unitary representation on $\ell^2 (P,\delta)$ with respect to this $*$-structure.

\begin{proposition}\label{unitary-H:prp}
The difference-reflection representation $h\to\hat{T}(h)$ ($h\in H$) on $C(P)$ restricts to a unitary representation of the affine Hecke algebra
into the space of bounded operators on $l^2(P,\delta )$, i.e.
\begin{equation}
\langle \hat{T}(h)f,g\rangle_\delta =
\langle f,\hat{T}(h^*)g\rangle_\delta \qquad (h\in H,\ f,g\in l^2(P,\delta )).
\end{equation}
\end{proposition}
\begin{proof}
Formally the proof is the same as that of \cite[Thm. 6.1]{die-ems:unitary}, which corresponds to the situation that $c=1$. Since strictly speaking the statement for $c\in \mathbb{N}_{>1}$ only follows from that for $c=1$ when $f,g\in l^2(P,\delta)$ are supported on the sublattice $cP$, the argument is repeated here (for $c\in \mathbb{N}_{>1}$) so as to keep our presentation complete.
Let $f,g\in l^2(P,\delta )$. It suffices to show that the actions of $\hat{T}_j$ ($0\leq j\leq n$) and  $u$ ($u\in\Omega$) determine bounded operators on $l^2(P,\delta )$ satisfying
$\text{(i)}$ $\langle \hat{T}_j f,g\rangle_\delta=\langle f, \hat{T}_j g\rangle_\delta$ and  $\text{(ii)}$ $\langle u f,g\rangle_\delta=\langle f, u^{-1} g\rangle_\delta$.
Property $\text{(ii)}$  follows by performing the change of coordinates $\lambda\to u\lambda$ to the (discrete) integral $\langle u f,g\rangle_\delta$. Indeed, invoking of the symmetry $\delta_{u\lambda}=\delta_\lambda$ (as
$w_{u\lambda}=uw_\lambda u^{-1}$ and therefore $\tau_{w_{u\lambda}}=\tau_{w_\lambda}$) then produces the integral $\langle  f,u^{-1}g\rangle_\delta$.
Property $\text{(i)}$ follows in turn by performing the change of coordinates $\lambda\to s_j\lambda$ to the integral $\langle \chi_{a_j} s_j f,g\rangle_\delta$, which entails the integral  $\langle f,  \chi_{a_j} s_j g\rangle_\delta$.
Here one uses the symmetries
$s_j\chi_{a_j}=\chi_{a_j}^{-1}s_j$ and
$\delta_{s_j\lambda}=\chi_{a_j}^2(\lambda )\delta_\lambda$ (as $w_{s_j\lambda}=w_\lambda s_j$ and $\ell(w_{s_j\lambda})=\ell(w_\lambda)+\text{sign}(a_j(\lambda))$ for $a_j(\lambda)\neq 0$). The computations in question also reveal that the actions of
$u$  and $s_j$ (and thus that of $\hat{T}_j$) are indeed bounded in  $l^2(P,\delta )$
(as $\langle uf ,uf\rangle_\delta =\langle f ,f\rangle_\delta$
and $\langle s_jf ,s_j f\rangle_\delta =\langle\chi_{a_j}s_j f ,\chi_{a_j}^{-1}s_jf\rangle_\delta=\langle f ,\chi_{a_j} s_j\chi_{a_j}^{-1}s_jf\rangle_\delta=\langle f ,\chi_{a_j}^{2}f\rangle_\delta$, and
$\chi_{a_j}$ is a bounded function on $P$).
\end{proof}

Let us denote the longest element of $W_0$ by $v_0$. The $*$-structure on $H$ will now be extended to an antilinear anti-involution of the subalgebra $ \mathbb{C}[X]^{W_0}\otimes H \cong   \mathbb{C}[X]^{W_0}H\subset \mathbb{H}$ with basis $m_\lambda (X)T_w$ ($\lambda\in P^+$, $w\in W$)
as follows:
\begin{equation}
m_\lambda (X)^* := m_{\lambda^*}(X)\qquad (\lambda\in P^+),
\end{equation}
where $\lambda^*:=-v_0\lambda$ ($\in P^+$).  We expect that the unitarity of Proposition \ref{unitary-H:prp} carries
over to $ \mathbb{C}[X]^{W_0} H$:
\begin{equation}\label{unitary-Z}
\langle L_p f,g,\rangle_\delta \stackrel{?}{=}\langle f,L_{p^{*}}g\rangle_\delta \qquad (\forall p\in\mathbb{C}[X]^{W_0}\ \text{and}\ f,g\in l^2(P,\delta)) .
\end{equation}
For $p=m_\omega$
with $\omega\in P^+$ (quasi-)minuscule, the explicit formula for $L_p$ in Theorem \ref{action:thm} allows us to confirm
that  Eq. \eqref{unitary-Z} indeed holds in this special case. Notice in this connection that if $\omega$ is minuscule then so is $\omega^*$, and that if $\omega$ is quasi-minuscule then $\omega^*=\omega$.

\begin{theorem}\label{unitary-Z:thm}
For $\omega\in P^+$ (quasi-)minuscule, one has that
\begin{equation}\label{adjointness-delta:rel}
\langle L_\omega f,g,\rangle_\delta =\langle f,L_{\omega^{*}}g\rangle_\delta \qquad (\forall f,g\in l^2(P,\delta)).
\end{equation}
\end{theorem}
\begin{proof}
It follows from the explicit formula in Theorem \ref{action:thm} that $L_\omega$ is bounded in $l^2 (P,\delta)$ for $\omega$ (quasi-)minuscule. Indeed, for any $\nu\in W_0\omega$ the coefficients $a_{\lambda ,\nu}$ and $b_{\lambda ,\nu}$ \eqref{a-coef}--\eqref{hvee} remain bounded as functions of $\lambda\in P$.
By comparing $\langle L_\omega f, g \rangle_\delta$ with $\langle  f, L_{\omega^*} g \rangle_\delta$ we see that
\begin{eqnarray*}
\lefteqn{\langle L_\omega f, g \rangle_\delta - \langle  f, L_{\omega^*} g \rangle_\delta } && \\
&&= \sum_{\substack{\lambda\in P\\ \nu\in W_0\omega}}  \left( a_{\lambda,\nu} f(\lambda-\nu)\overline{g(\lambda)}
-a_{\lambda,-\nu} f(\lambda) \overline{g(\lambda+\nu)}\right)
\delta_\lambda \\
&&=
\sum_{\substack{\lambda\in P\\ \nu\in W_0\omega}}  \bigl( a_{\lambda,\nu} \delta_\lambda
-a_{\lambda-\nu ,-\nu} \delta_{\lambda-\nu} \bigr) f(\lambda-\nu)\overline{g(\lambda)} ,
\end{eqnarray*}
whence it suffices to show that $ a_{\lambda,\nu} \delta_\lambda
=a_{\lambda-\nu ,-\nu} \delta_{\lambda-\nu}$ for all $\lambda\in P$ and $\nu\in W_0\omega$, or more explicitly:
\begin{equation*}
\tau_{w_{w_\lambda(\lambda-\nu)}w_\lambda}\tau_{w_{w_\lambda(\lambda-\nu)}}\tau_{w_\lambda}
=\tau_{w_{w_{\lambda-\nu}(\lambda)}w_{\lambda-\nu}}\tau_{w_{w_{\lambda-\nu}(\lambda)}}\tau_{w_{\lambda-\nu}} .
\end{equation*}
The proof of this relation for the length multiplicative function $\tau$ is relegated to Appendix \ref{appB} below.
\end{proof}

The following (self-)adjointness relation is immediate from
Theorem \ref{unitary-Z:thm} upon the restriction of $L_\omega$ to the $W_R$-invariant subspace $l^2(P,\delta)^{W_R}\cong l^2 (P^+_c,\Delta )$.
\begin{corollary}\label{unitary-Z:cor}
For $\omega\in P^+$ (quasi-)minuscule, one has that
\begin{equation}\label{adjointness-Delta:rel}
\langle L_\omega f,g,\rangle_\Delta =\langle f,L_{\omega^{*}}g\rangle_\Delta \qquad (\forall f,g\in l^2(P^+_c,\Delta)).
\end{equation}
\end{corollary}

\begin{remark}\label{adjointness:rem}
For the classical root systems of type $A_n$, $B_2$ and $D_4$ the unitarity in Eq. \eqref{unitary-Z} and the (self-)adjointness relation in Eq. \eqref{adjointness-Delta:rel}  with $\omega\in P^+$ arbitrary are
a direct consequence of Theorem \ref{unitary-Z:thm}, because in these special cases the monomials $m_\omega (X)$ with $\omega$ (quasi-)minuscule already generate the complete algebra $\mathbb{C}[X]^{W_0}$.
\end{remark}

\subsection{Orthogonality}
The diagonalization in Corollary \ref{diagonal:cor} and the expected unitarity in Eq. \eqref{unitary-Z} suggest that the periodic Macdonald spherical functions form an {\em orthogonal} basis of $l^2(P^+_c,\Delta)$. However, since Corollary \ref{unitary-Z:cor} only establishes the (self-)adjointness relations for $\omega$ (quasi-)minuscule, the orthogonality in question is not immediate at this point (because of possible degeneracies in the spectrum of the relevant discrete Laplacians $L_\omega$) and a more sophisticated analysis is required.

By a standard continuity argument, it does follow from the diagonalization in Corollary \ref{diagonal:cor} and the (self-)adjointness relations in Corollary \ref{unitary-Z:cor} that the periodic Macdonald spherical functions $\Phi_{\xi_{\mu}}$ and $\Phi_{\xi_{\tilde{\mu}}}$ ($\mu,\tilde{\mu}\in P^{\vee ,+}_c$) are orthogonal in $l^2(P^+_c,\Delta)$ for all $0<\tau^2 < 1$ if---for some (quasi-)minuscule weight $\omega$---the corresponding eigenvalues $m_{\omega} (e^{i\xi_\mu})$ and $m_{\omega} (e^{i\xi_{\tilde{\mu}}})$ (of $L_{\omega^*}$) are distinct as analytic functions in the parameter(s) $\tau_{\alpha^\vee}$, $\alpha\in R_0^+$. In view of the injectivity in part $ii)$ of Proposition \ref{spectrum:prp},
 this assures the orthogonality for $\mu\neq\tilde{\mu}$ in the case of
 the classical root systems $A_n$, $B_2$ and $D_4$, as for these types the relevant symmetric monomials $m_\omega (e^{i\xi})$ with $\omega\in P^+$ (quasi-)minuscule separate the points of $2\pi A^\vee$ \eqref{Avee} (cf. Remark \ref{adjointness:rem}). More generally, it suffices to verify the inequality of $m_\omega (e^{i\xi_\mu})$ and $m_\omega (e^{i\xi_{\tilde{\mu}}})$ (or any of their derivatives with respect to the parameter(s))
at any fixed value for $\tau $ in the analyticity domain $-1<\tau^2 < 1$ to conclude their inequality as analytic functions.
By determining the limiting behavior of $m_\omega (e^{i\xi_\mu})$ for $\tau\uparrow 1$, and
computing
$m_\omega (e^{i\xi_\mu})$ and $\frac{\partial}{\partial \tau_{\beta^\vee}^2} m_\omega (e^{i\xi_\mu})$ ($\beta\in R_0^+$)  at the special limiting value $\tau\to 0$ by means
of part $iv)$ of Proposition \ref{spectrum:prp} and the implicit function theorem,  we arrive at the following explicit numerical criterion guaranteeing the orthogonality of the periodic Macdonald spherical functions
for general $R_0$.

\begin{proposition}\label{orthogonality:prp}
For any $\mu,\tilde{\mu}\in P^{\vee,+}_c$ the corresponding periodic Macdonald spherical functions are orthogonal:
\begin{subequations}
\begin{equation}\label{ort-rel}
\langle \Phi_{\xi_\mu},\Phi_{\xi_{\tilde{\mu}}}\rangle_\Delta
=\sum_{\lambda\in P_c^+}\Phi_{\xi_\mu}(\lambda) \Phi_{\xi_{\tilde{\mu}}}(\lambda) \Delta_\lambda=0 ,
\end{equation}
if for {\em some} $\omega\in P^+$ (quasi-)minuscule and {\em some} $\epsilon \in \{ 0,1\}$, $\beta\in R_0^+$:
\begin{equation}\label{ort-crit1}
m_\omega ( e^{i\xi_{\mu }^{\epsilon  }}) \neq m_\omega ( e^{i\xi_{\tilde{\mu}}^\epsilon} )
\end{equation}
{\em or}
\begin{eqnarray}\label{ort-crit2}
\lefteqn{\sum_{\nu\in W_0\omega}
\sum_{\substack{\alpha\in R_0^+\\ \|\alpha \| =\|\beta\| }}
e^{i\langle \nu ,\xi_\mu^0\rangle}\sin (\langle \xi_\mu^0 ,\alpha\rangle )\langle \nu , \alpha^\vee \rangle\neq }&& \\
&& \sum_{\nu\in W_0\omega}
\sum_{\substack{\alpha\in R_0^+\\ \|\alpha \|=\|\beta\| }}
e^{i\langle \nu ,\xi_{\tilde{\mu}}^0\rangle}\sin (\langle \xi_{\tilde{\mu}}^0 ,\alpha\rangle )\langle \nu , \alpha^\vee \rangle , \nonumber
\end{eqnarray}
where $\|\alpha\| :=\langle \alpha ,\alpha\rangle^{1/2}$ and
\begin{equation}\label{spectrum-tau-1-0}
\xi_\mu^\epsilon :=\lim_{\tau\to \epsilon} \xi_\mu =\begin{cases}
\frac{2\pi i}{c+ h(R_0)}( \rho^\vee+\mu) &\text{if}\ \epsilon =0, \\
2\pi i c^{-1}\mu &\text{if}\ \epsilon =1 .
\end{cases}
\end{equation}
\end{subequations}
\end{proposition}

\begin{proof}
It is clear from part $iv)$ of Proposition \ref{spectrum:prp} that for $0<\tau^2 < 1$
\begin{equation*}
\lim_{\tau\to \epsilon} m_\omega (e^{i\xi_\mu }) = m_\omega (e^{i\xi_\mu^\epsilon}) .
\end{equation*}
Hence, if Eq. \eqref{ort-crit1} is satisfied (either for $\epsilon =0$ or for $\epsilon =1$) then the eigenvalues $m_{\omega} (e^{i\xi_\mu})$ and $m_{\omega} (e^{i\xi_{\tilde{\mu}}})$ cannot be identical as analytic functions of $\tau_{\alpha^\vee}$, $\alpha\in R_0^+$,
whence $\langle \Phi_{\xi_\mu},\Phi_{\xi_{\tilde{\mu}}}\rangle_\Delta=0$.
In the same way it is seen from the limit
\begin{eqnarray}\label{der-eig-lim}
\lefteqn{\lim_{\tau\to 0}\frac{\partial}{\partial \tau_{\beta^\vee}^2} m_\omega (e^{i\xi_\mu})=} && \\
&& \frac{2}{i(c+h(R_0))}
\sum_{\nu\in W_0\omega}
\sum_{\substack{\alpha\in R_0^+\\ \|\alpha \| =\|\beta\| }}
e^{i\langle \nu ,\xi_\mu^0\rangle}\sin (\langle \xi_\mu^0 ,\alpha\rangle )\langle \nu , \alpha^\vee \rangle  \nonumber
\end{eqnarray}
that $\langle \Phi_{\xi_\mu},\Phi_{\xi_{\tilde{\mu}}}\rangle_\Delta=0$ if the inequality in Eq. \eqref{ort-crit2} is satisfied.
To verify the limit in Eq. \eqref{der-eig-lim}
we first compute $\frac{\partial}{\partial \tau_{\beta^\vee}^2} m_\omega (e^{i\xi_\mu})$ for general $0<\tau^2 < 1$:
\begin{subequations}
\begin{equation}\label{der-eig1}
\frac{\partial}{\partial \tau_{\beta^\vee}^2} m_\omega (e^{i\xi_\mu})=\frac{\partial}{\partial \tau_{\beta^\vee}^2}
\sum_{\nu\in W_0\omega}e^{i\langle \nu ,\xi_\mu\rangle}
=i \sum_{\nu\in W_0\omega}e^{i\langle \nu ,\xi_\mu\rangle}\langle \nu ,\frac{\partial \xi_\mu}{\partial \tau_{\beta^\vee}^2}\rangle
\end{equation}
with (upon employing the implicit function theorem to Eq. \eqref{critical:eq})
\begin{equation}
\frac{\partial \xi_\mu}{\partial \tau_{\beta^\vee}^2}=
-\mathcal{H}^{-1}(\xi_\mu ) \frac{\partial  \rho^\vee_v}{\partial \tau_{\beta^\vee}^2}(\xi_\mu) .
\end{equation}
Here $\mathcal{H}(\xi)$ refers to the Hessian with components $\mathcal{H}_{\eta ,\zeta}(\xi )$ given by Eq. \eqref{hessian}
and
\begin{equation}\label{der-eig3}
\frac{\partial  \rho^\vee_v}{\partial \tau_{\beta^\vee}^2}(\xi)=
\sum_{\substack{\alpha\in R_0^+\\ \|\alpha \| =\|\beta\|}}
\frac{2\sin (\langle \xi ,\alpha\rangle )}{1-2\tau_{\alpha^\vee}^2\cos(\langle \xi,\alpha\rangle)+ \tau_{\alpha^\vee}^4} \alpha^\vee .
\end{equation}
\end{subequations}
For $\tau\to 0$ the expression in Eqs. \eqref{der-eig1}--\eqref{der-eig3} simplifies to the RHS of Eq. \eqref{der-eig-lim}
since $\xi_\mu \stackrel{\tau\to 0}{=} \xi_\mu^0$ and
$
\mathcal{H}_{\eta ,\zeta}(\xi )\stackrel{\tau\to 0}{=}
(c+h(R_0))\langle \eta,\zeta\rangle
$
(in view of Remark \ref{schur:rem} and recalling also that $|R_0|=nh(R_0)$).
\end{proof}

For a given concrete root system $R_0$ and a fixed value of $c\in\mathbb{N}_{>1}$ of `reasonable size', a direct verification of the numerical criterion in Proposition \ref{orthogonality:prp} readily entails the orthogonality of $\Phi_{\xi_\mu}$ and $\Phi_{\xi_{\tilde{\mu}}}$ in $l^2(P^+_c,\Delta)$ for most (and possibly all) coweights $\mu\neq\tilde{\mu}$ in $P^{\vee ,+}_c$. In fact,
we have not spotted any counterexamples where our criterion fails to separate the eigenvalues if $\mu\neq \tilde{\mu}$, even though we are not in the position to offer an a priori argument ruling out the existence of such degeneracies altogether
(apart from the above separation argument in the already mentioned cases when $R_0$ is of type
$A_n$, $B_2$ or $D_4$).

\subsection{Normalization}
To determine the spectral measure of the eigenfunction transform of our Laplacians $L_p$ \eqref{Lp} it remains to compute the quadratic norms of the periodic Macdonald spherical functions in $l^2(P^+_c,\Delta)$.
After the seminal contributions of Gaudin and Korepin \cite{gau:fonction,kor:calculations}, it is nowadays a paradigm of the Bethe Ansatz method that the quadratic norm of a Bethe eigenfunction corresponding to a solution of the Bethe equations
determined by the global minimum of a strictly convex Morse function should (essentially) be given by the determinant of its Hessian \cite{kor-bog-ize:quantum,sla:algebraic}.
Translated to our present setting, this heuristics---which is confirmed for many Bethe Ansatz models \cite{kor-bog-ize:quantum}---gives rise to the following conjectural Gaudin-type determinantal formula for the quadratic norms in question.

\begin{conjecture}\label{gaudin:con} For any $\mu \in P^{\vee ,+}_c$, the quadratic norm of the periodic Macdonald spherical function $\Phi_{\xi_\mu}$ in the Hilbert space $l^2(P^+_c,\Delta)$ is given by
\begin{equation}\label{gaudin-formula}
\langle \Phi_{\xi_\mu},\Phi_{\xi_\mu}\rangle_\Delta
=\sum_{\lambda\in P^+_c}|\Phi_{\xi_\mu}(\lambda)|^2 \Delta_\lambda= \text{Ind}(R_0)C(\xi_\mu)C(-\xi_\mu)\det\mathcal{H}(\xi_\mu) ,
\end{equation}
where $\text{Ind}(R_0):=|\Omega |$, $C(\cdot )$ is taken from Eq. \eqref{cfun}, and $\det\mathcal{H}(\cdot)$ refers to the determinant of the Hessian given by Eq. \eqref{hessian}.
\end{conjecture}

Conjecture \ref{gaudin:con} generalizes analogous conjectural normalization formulas from \cite[\text{Eq.}~(3.5)]{die:finite-dimensional}
(corresponding to the special situation of a root system of type $A$ \cite{die:diagonalization})  and \cite[\text{Eq.}~(11)]{ems:completeness} (corresponding to the trigonometric degenerate double affine Hecke algebra at critical level encoding
Gaudin's Weyl-group invariant delta-potential models with periodic boundary conditions \cite{ems-opd-sto:periodic}).
Following a brute-force approach detailed in Refs. \cite{die:finite-dimensional} and \cite{bus-die-maz:norm} for these two previous conjectures, it is possible to confirm that Conjecture \ref{gaudin:con} holds true for small root systems.
In a nutshell, the idea is that the plane wave decomposition in Eq. \eqref{msf-decomposition} permits expressing the sum $\sum_{\lambda\in P^+_c}|\Phi_{\xi}(\lambda)|^2 \Delta_\lambda$ in terms of exponential sums of the type
$\sum_{\lambda\in P^+_c} e^{i\langle \tilde{\xi},\lambda\rangle} \Delta_\lambda$ (with $\tilde{\xi}\in W_0(\xi -v\xi)$, $v\in W_0$).  Since $\Delta_\lambda$ \eqref{Delta}
is determined completely by the stabilizer $W_{R,\lambda}$, partitioning of the exponential sum into partial sums over weights of $P^+_c$
having the same stabilizer subgroup inside $W_R$  (i.e., weights belonging to the same facet of the Coxeter complex of $W_R$) produces an exact evaluation of the exponential sum in terms of terminating {\em geometric} sums. Upon
elimination of any exponentials of the form $e^{ic\langle \xi,\beta\rangle}$, $\beta\in R_0^+$ with the aid of the Bethe type equations \eqref{BAE}, one ends up with a tedious algebraic expression for $\sum_{\lambda\in P^+_c}|\Phi_{\xi}(\lambda)|^2 \Delta_\lambda$ that is to be compared with the conjectural formula $\text{Ind}(R_0)C(\xi)C(-\xi)\det\mathcal{H}(\xi)$ on the RHS. In all cases that we have checked (using computer algebra), both expressions turn out to agree as algebraic {\em functions} of the spectral variable $\xi$ and the parameter $\tau$. Specifically,
equality was checked
for all classical root systems of rank $\leq 2$ symbolically and for all classical root systems of rank $\leq 4$ upon evaluating both expressions at a large number of random values for $\tau_s$, $\tau_l$ and $\xi$. We also verified the case of the exceptional root system $G_2$ symbolically for a significant number of random values for $\tau_s$ and $\tau_l$.

For $\tau\uparrow 1$  (cf. Remark \ref{free-laplacians:rem}), the (conjectural) orthogonality of the periodic Macdonald spherical functions reduces to the following orthogonality relations
\begin{subequations}
\begin{equation}\label{tau=1:ort}
\sum_{\lambda\in P^+_c} M_\lambda ( e^{i\xi_{\mu}^1})
M_\lambda ( e^{-i\xi_{\tilde{\mu}}^1}) |W_{R,\lambda} |^{-1} =
\begin{cases}
c^n\text{Ind} (R_0)|W_{\hat{R}_c,\mu} | &\text{if}\ \mu =\tilde{\mu}\\
0&\text{if}\ \mu \neq \tilde{\mu}
\end{cases}
\end{equation}
for the periodic symmetric monomials $M_\lambda (  e^{i\xi_{\mu}^1})$ $(\lambda\in P_c^+,\ \mu\in P^{\vee ,+}_c)$,
where
\begin{equation}
M_\lambda (  e^{i\xi}):= |W_{R,\lambda}\cap W_0| m_\lambda ( e^{i\xi})=\sum_{v\in W_0}e^{i\langle \lambda,v \xi \rangle} ,
\end{equation}
\end{subequations}
$\xi_\mu^1=2\pi c^{-1}\mu$ (cf. Eq. \eqref{spectrum-tau-1-0}), and $\hat{R}_c=R_0+c\mathbb{Z}$ refers to the affine root system with $R_0$ replaced by $R_0^\vee$.
For $\tau\to 0$ on the other hand, one arrives at the following orthogonality relations (cf. \cite[\text{Thm.}~6.2]{kir:inner}, \cite[\text{Prp.}~5.4]{hri-pat:discretization} and---for $R_0$ of type $A$---\cite[\text{Eq.}~(4.8)]{die:finite-dimensional} and
\cite[\text{Thm.}~6.7]{kor-stro:slnk})
\begin{subequations}
\begin{equation}\label{tau=0:ort}
\sum_{\lambda\in P^+_c} \chi_\lambda ( e^{i\xi_{\mu}^0})
\chi_\lambda ( e^{-i\xi_{\tilde{\mu}}^0})=
\begin{cases}
(c+h(R_0))^n\text{Ind} (R_0) &\text{if}\ \mu =\tilde{\mu}\\
0&\text{if}\ \mu \neq \tilde{\mu}
\end{cases}
\end{equation}
for the periodic anti-symmetric monomials $ \chi_\lambda ( e^{i\xi_{\mu}^0})$ ($\lambda\in P_c^+$, $\mu\in P^{\vee ,+}_c$), where
\begin{equation}
\chi_\lambda (e^{i\xi} ):= \sum_{v\in W_0} (-1)^{\ell(v)} e^{i\langle \rho+\lambda ,v\xi\rangle}
\end{equation}
\end{subequations}
and $\xi_\mu^0=\frac{2\pi}{c+h(R_0)}(\rho^\vee+\mu)$ (cf. Eq. \eqref{spectrum-tau-1-0}).

Both degenerations of the orthogonality relations under consideration
can be verified directly with the aid of the orthogonality of the characters $e^{2\pi i c^{-1}\mu}$, $\mu\in P^\vee/c Q^\vee$ of the finite abelian group $P/cQ$:
\begin{equation}\label{character:ort}
\sum_{\lambda\in P/cQ} e^{2\pi i c^{-1}\mu} (\lambda)=
\begin{cases} 0 &\text{if}\ \mu\in P^\vee\setminus c Q^\vee \\
|P/cQ |&\text{if}\ \mu\in c Q^\vee
\end{cases},
\end{equation}
upon  using that $|P/cQ|=c^n\text{Ind}(R_0)$ and
\begin{equation}\label{symmetrization}
\frac{1}{|W_0|} \sum_{\lambda\in P/cQ} f(\lambda)=\sum_{\lambda\in P_c^{+}}\frac{1}{|W_{R,\lambda}|}  f(\lambda) \quad\text{for}\quad f\in\mathcal{C}(P)^{W_R}.
\end{equation}
Indeed, the LHS of Eq. \eqref{tau=1:ort} is readily rewritten as
\begin{eqnarray*}
\lefteqn{\sum_{\tilde{v}\in W_0} \sum_{\lambda\in P_c^+}  \frac{1}{ |W_{R,\lambda} |}\sum_{v\in W_0}  e^{2\pi ic^{-1} v (\mu-\tilde{v}\tilde{\mu})} (\lambda)\stackrel{\text{Eq.}~\eqref{symmetrization}}{=}} && \\
&& \frac{1}{|W_0|} \sum_{v,\tilde{v}\in W_0}  \sum_{\lambda\in P/cQ}   e^{2\pi ic^{-1} v (\mu-\tilde{v}\tilde{\mu})} (\lambda ) \stackrel{\text{Eq.}~\eqref{character:ort}}{=}
\begin{cases}
c^n\text{Ind} (R_0)|W_{\hat{R}_c,\mu} | &\text{if}\ \mu =\tilde{\mu}\\
0&\text{if}\ \mu \neq \tilde{\mu}
\end{cases}
\end{eqnarray*}
(where it was exploited in the last step that
both coweights $\mu$ and $\tilde{\mu}$ belong to the fundamental domain $P_c^{\vee , +}$ \eqref{fr}).
Moreover, the LHS of Eq. \eqref{tau=0:ort} is rewritten along similar lines as
\begin{eqnarray*}
\lefteqn{\sum_{\tilde{v}\in W_0} (-1)^{\ell (\tilde{v})} \sum_{\lambda\in P_c^+} \sum_{v\in W_0}  e^{\frac{2\pi i}{c+h} v (\rho^\vee+\mu-\tilde{v}(\rho^\vee+\tilde{\mu}))} (\rho+\lambda)} && \\
&&=
\sum_{\tilde{v}\in W_0} (-1)^{\ell (\tilde{v})} \sum_{\rho+\lambda\in P_{c+h}^+}
 \frac{1}{ |W_{R_{c+h},\lambda} |}
 \sum_{v\in W_0}  e^{\frac{2\pi i}{c+h} v (\rho^\vee+\mu-\tilde{v}(\rho^\vee+\tilde{\mu}))} (\rho+\lambda)
 \\
&& \stackrel{\text{Eq.}~\eqref{symmetrization}}{=}  \frac{1}{|W_0|} \sum_{v,\tilde{v}\in W_0}  (-1)^{\ell (\tilde{v})} \sum_{\rho+\lambda\in P/(c+h)Q}   e^{\frac{2\pi i}{c+h} v (\rho^\vee+\mu-\tilde{v}(\rho^\vee+\tilde{\mu}))} (\rho+\lambda)\\
&&\stackrel{\text{Eq.}~\eqref{character:ort}}{=}  \begin{cases}
 (c+h)^n\text{Ind} (R_0) &\text{if}\ \mu =\tilde{\mu}\\
0&\text{if}\ \mu \neq \tilde{\mu}
\end{cases} ,
\end{eqnarray*}
where we have employed the shorthands $h$ for $h(R_0)$ and $R_{c+h}$ for the affine root system with $c$ replaced by $c+h$. The first equality relies on the fact that $\lambda\in P_c^+$ iff $\rho+\lambda$ is a regular point of $P_{c+h}^+$ (with respect to the action of $W_{R_{c+h}}$), while we also used that in the expression on the second line the terms with $ |W_{R_{c+h},\lambda} |>1$ cancel out; in the last equality it was exploited that for $\mu,\tilde{\mu}\in P^{\vee ,+}_c$ both coweights $\rho^\vee+\mu$ and $\rho^\vee+\tilde{\mu}$ are regular points of the fundamental domain $P_{c+h (R_0^\vee)}^{\vee , +}$  (with respect to the action of $W_{\hat{R}_{c+h (R_0^\vee)}}$).

It is not very difficult to infer that for $\tau\uparrow 1$ the inner product
$
\langle \Phi_{\xi_\mu},\Phi_{\xi_{\tilde{\mu}}}\rangle_\Delta $ indeed tends
to the LHS of Eq. \eqref{tau=1:ort}. Here one uses that for $\mu\in P_c^{\vee ,+}$: $\lim_{\tau\uparrow 1}\xi_\mu=\xi_\mu^1$ and $\lim_{\tau\uparrow 1} C(\xi_\mu)=1$. (When $|W_{\hat{R}_c,\mu}|>1$ the second limit can be deduced from the first
with the aid of the critical equation \eqref{critical:eq}, cf. also the proof of Prop. 3 in Ref. \cite{ems:completeness}.)
This confirms the orthogonality of the periodic Macdonald
spherical functions in the limit $\tau\uparrow 1$. Moreover, the conjectured value of the quadratic norm of the periodic Macdonald spherical function on the RHS of Eq. \eqref{gaudin-formula} converges for $\tau\uparrow 1$ to the RHS of Eq. \eqref{tau=1:ort} provided
\begin{equation}\label{hessian-limit-tau=1}
\lim_{\tau\uparrow 1} \text{det} \mathcal{H}(\xi_\mu)= c^n |W_{\hat{R}_c,\mu }| \qquad (\mu\in P^{\vee ,+}_c).
\end{equation}
Whereas Eq. \eqref{hessian-limit-tau=1} is readily seen to hold when $\mu$ is $W_{\hat{R}_c}$-regular (i.e. $|W_{\hat{R}_c,\mu }|=1$), the limit in question is far from obvious when $|W_{\hat{R}_c,\mu }|>1$
(cf. also \cite[\text{Sec.}~2]{ems:completeness}).
In other words, this only confirms our norm formula for the periodic Macdonald
spherical functions in the limit $\tau\uparrow 1$
when $\mu\in P_c^{\vee ,+}$ is $W_{\hat{R}_c}$-regular, whereas for $|W_{\hat{R}_c,\mu }|>1$ the limit in Eq. \eqref{hessian-limit-tau=1} would follow rather as a {\em consequence} of the norm formula in Conjecture \ref{gaudin:con}.
The limiting behavior for $\tau\to 0$ is straightforward: upon multiplying out the (normalizing) Weyl denominators
$\chi_\rho (e^{i\xi_\mu^0})$ and $\chi_\rho (e^{-i\xi_{\tilde{\mu}}^0})$  the inner product
$
\langle \Phi_{\xi_\mu},\Phi_{\xi_{\tilde{\mu}}}\rangle_\Delta $ tends for $\tau\to 0$
to the LHS of Eq. \eqref{tau=0:ort} and the  RHS of Eq. \eqref{gaudin-formula} converges  to the RHS of Eq. \eqref{tau=0:ort},
which confirms our orthogonality relations for the periodic Macdonald
spherical functions in the limit $\tau\to 0$.

\begin{remark}
In the special situation that we are dealing with a root system $R_0$ of type $A$ and $\omega$ is chosen to be minuscule, the diagonalization, adjointness relations, and the integrability of the symmetrized operator $L_\omega$ \eqref{Lminuscule-red} in Corollaries \ref{diagonal:cor},
\ref{unitary-Z:cor} and Eq. \eqref{integrability}, respectively, as well as the orthogonality of the basis of periodic Macdonald spherical functions implied by Proposition \ref{orthogonality:prp} (and Remark \ref{adjointness:rem}), reproduce the principal results of Ref. \cite{die:diagonalization} in Prop. 2.3 and  Thms. 5.1--5.3. While double affine Hecke algebras did not manifest themselves at all in Ref. \cite{die:diagonalization}, their role in the present generalization of these previous results to the case of {\em arbitrary} Weyl groups seems to be fundamental (reminding of a similar state of affairs in the theory of Macdonald's polynomials \cite{mac:symmetric,mac:affine,che:double}).
\end{remark}

\begin{remark}\label{parameter-domain:rem}
 For the parameter regime $-1<\tau^2 <0$, the coefficients $a_{\lambda ,\nu}$ and $b_{\lambda ,\nu}$ \eqref{a-coef}--\eqref{hvee} are still real-valued for any $\lambda\in P$ and $\nu\in P_\vartheta^\ast$ because of their quadratic dependence on $\tau$. For the coefficient $b_{\lambda,\nu}$ this is immediate form the definition whereas for $a_{\lambda ,\nu}$ this follows from the observation that the coefficient in question
is always of the form $\tau^2_{w_{w_\lambda(\lambda-\nu)}}\tau_0^{-2\epsilon}$ with $\epsilon\in \{0 ,1\}$ (cf. Appendix \ref{appB}). Macdonald's product formula \eqref{poincare-stab}, \eqref{htau} reveals moreover that $\Delta_\lambda>0$ in this parameter domain even though the positivity of the nondegenerate scalar product $\langle \cdot ,\cdot\rangle_\delta$
\eqref{ip-delta}, \eqref{delta} is now lost. The upshot is that
Theorem \ref{unitary-Z:thm}, Corollary \ref{unitary-Z:cor} and Proposition \ref{orthogonality:prp} remain valid (with the given proofs applying verbatim) for $-1<\tau^2 <0$. It is also expected that the prediction of the quadratic norms in Conjecture \ref{gaudin:con} still holds for the parameter regime at issue.
\end{remark}

\vspace{3ex}
\section*{Acknowledgments.}
The results of this paper were presented at the Complex Analysis and Integrable Systems program of the Mittag-Leffler Institute, Djursholm, Sweden (Fall, 2011) and at the Analysis Colloquium of the Delft University of Technology, Delft, The Netherlands (Spring, 2012). We would like to thank Alexander Vasil'ev (JFvD) and Wolter Groenevelt (EE) for these opportunities to report our work.
Our brute-force verification for small root systems  of the Gaudin-type normalization formula in Conjecture \ref{gaudin:con} depended heavily on
Stembridge's Maple packages {\tt COXETER} and
{\tt WEYL}. We are grateful to Manuel O'Ryan of the Universidad de Talca for the necessary computer facilities permitting to carry out these computations. Finally,
the referees'
constructive comments helping us to straighten the presentation are very much appreciated.

\appendix

\section{Intertwining Properties}\label{appA}
In this appendix we prove that for any $f\in\mathcal{C}(P)$,
$\lambda\in P$ and $\nu\in P_\vartheta^\star$:
\begin{equation}\label{intertwining-property}
\tau^{-1}_{w_\lambda}  (I_{w_\lambda}f)(w_\lambda(\lambda-\nu)) =
a_{\lambda,\nu}(\mathcal{J}f)(\lambda-\nu)+b_{\lambda,\nu} (1-\tau_0^{-2}) (\mathcal{J}f)(\lambda),
\end{equation}
with $a_{\lambda,\nu}$ and $b_{\lambda,\nu}$ given by Eqs. \eqref{a-coef}--\eqref{hvee}.
This relation lies at the basis of the affine intertwining relation in Eq. \eqref{itc} and the explicit expression for $L_\omega$ in Theorem \ref{action:thm}.
It is a double affine analog of a similar relation in \cite[\text{Eq}.~(7.7)]{die-ems:unitary}. The proof below runs along the same lines as the corresponding proof in \cite[\text{Sec.}~7.1]{die-ems:unitary}, but we feel compelled to provide details in this appendix as the transition from `affine' to `double affine' is quite subtle at key points.

Our verification of Eq. \eqref{intertwining-property} hinges on two technical lemmas concerning the properties of
$\theta (\lambda-\nu)$
\eqref{e:theta} and $(\mathcal{J}f)(\lambda -\nu)$ \eqref{int-op} for $\lambda\in P_c^+$ (and $\nu\in P_\vartheta^\star$).

\begin{lemma}\label{theta:lem}
 For $\lambda\in P_c^+$ and $\nu\in P_\vartheta^\star$, we are in either one of the following two situations:
$i)$ if $(\lambda-\nu)_+\neq \lambda$, then $w_{\lambda-\nu}\in W_{R,\lambda}$ and
 \begin{equation*}
\theta(\lambda-\nu)=
\begin{cases}
1 & \text{for $\nu^\vee-\langle \lambda,\nu^\vee\rangle \in R(w_{\lambda-\nu})$}\\
0 & \text{for $\nu^\vee-\langle \lambda,\nu^\vee\rangle \not\in R(w_{\lambda-\nu})$}\\
\end{cases} ,
\end{equation*}
or $ii)$ if $(\lambda-\nu)_+=\lambda$, then $w_{\lambda-\nu}'\nu=-\alpha_j$ for some $j\in\{0,\dots, n\}$ with $\tau_j=\tau_0$, moreover,
$s_j w_{\lambda-\nu}\in W_{R,\lambda}$, $\theta(\lambda-\nu)=0$,
$$R(w_{\lambda-\nu})\setminus R(s_j w_{\lambda-\nu})
= \{ (s_j w_{\lambda-\nu})^{-1}a_j\}  , $$
and  $a_j(\lambda)=1$, i.e.  $ (s_j w_{\lambda-\nu})^{-1}a_j =\nu^\vee+(1-\langle \lambda,\nu^\vee\rangle) .$
\end{lemma}

\begin{lemma}\label{Iqm-action:lem} For any $f\in\mathcal{C}(P)$, $\lambda\in P_c^+$ and
$\nu\in P_\vartheta^\star$:
 \begin{equation*}
\tau_{w_{\lambda-\nu}} (I_{w_{\lambda-\nu}} f)((\lambda-\nu )_+)=
f(\lambda-\nu) - c_{\lambda ,\nu}(1-\tau^{-2}_0)  f(\lambda) ,
\end{equation*}
with
$c_{\lambda ,\nu}=\theta (\lambda-\nu) e^\vee_\tau (\nu) (h_\tau^\vee)^{-\text{sign} (\langle \lambda ,\nu^\vee\rangle)}$ (cf. Eqs. \eqref{c-coef}--\eqref{hvee}).
\end{lemma}

It is readily seen from Lemma \ref{theta:lem}
that for $\lambda\in P^+_c$ and $\nu\in P_\vartheta$:
\begin{equation}\label{theta:range}
\theta (\lambda -\nu)\in \{ 0,1\}\quad \text{and}\quad
\langle \lambda ,\nu^\vee\rangle\in \{-c,0,c\} \  \text{if}\  \theta(\lambda -\nu)>0
\end{equation}
(which means, in particular, that in $\hat{X}^\nu$ \eqref{Xnu}--\eqref{hvee} the factor
$\theta (\lambda_+-\eta)$ of $c_{\lambda,\eta}$ \eqref{c-coef}
only assumes the values $0$ or $1$). To infer the second statement one notices that $\nu^\vee-\langle \lambda,\nu^\vee\rangle \in R^+$
if $\theta(\lambda -\nu)>0$, i.e.
$\nu\in R_0\cap P_\vartheta = W_0\vartheta$ and $\langle \lambda,\nu^\vee\rangle \in c\mathbb{Z}$. The statement now follows
because $|\langle \lambda,\nu^\vee\rangle |\leq c$ when $\lambda\in P_c^+$.

The proof of Lemma \ref{theta:lem} uses an elementary recurrence relation for $\theta$ \eqref{e:theta}.
Let $\mu\in P\setminus P^+_c$ and $j\in \{ 0,\ldots ,n\}$ such that $a_j\in R(w_{\mu})$ (cf. Eq. \eqref{Rwx}). Then
$w_\mu =w_{s_j\mu}s_j$ with $\ell( w_\mu )= \ell (w_{s_j\mu})+1$, and thus
$R(w_\mu)=s_jR(w_{s_j\mu})\cup \{a_j\}$ (cf. \cite[(2.2.4)]{mac:affine}).
From the definition of $\theta$ \eqref{e:theta} it is clear that in this situation:
\begin{equation}  \label{theta-rec}
  \theta(\mu)=
  \begin{cases}
  \theta(s_j\mu)+1& \text{if $a_j(\mu)=-2$},\\
  \theta(s_j\mu) & \text{if $a_j(\mu)\neq-2$}.
  \end{cases}
\end{equation}
Armed with this recurrence for $\theta$ it is seen that for $\lambda$ and $\nu$ as in Lemma \ref{theta:lem} with $\lambda-\nu\not\in P^+_c$, and
$\tilde{\nu}:=\lambda -s_j(\lambda-\nu)$ (so $s_j (\lambda -\nu)=\lambda-\tilde{\nu}$) with $j\in \{ 0,\ldots ,n\}$ such that $a_j\in R(w_{\lambda-\nu})$, we are necessarily in one of the following three cases:
\begin{itemize}
\item[$(A)$]$a_j(\lambda)=0$ and $ \langle \nu, \alpha^\vee_j\rangle =1$ (so $a_j(\lambda-\nu)=-1$). Then
$s_j\in W_{R,\lambda}$, so
 $\tilde{\nu}=s'_j\nu$ and  $\theta(\lambda-\nu)\stackrel{\text{Eq.}~\eqref{theta-rec}}{=}\theta(\lambda-s'_j\nu)$.
\item[$(B)$] $a_j(\lambda)=0$ and $\langle \nu,\alpha^\vee_j\rangle =2$ (so $a_j(\lambda-\nu)=-2$). Then
$s_j\in W_{R,\lambda}$ and $\nu=\alpha_j$,
so $\tilde{\nu}=s'_j\nu=-\alpha_j$ and $\theta(\lambda-\nu)\stackrel{\text{Eq.}~\eqref{theta-rec}}{=}\theta(\lambda-s'_j\nu)+1$.
\item[$(C)$] $a_j(\lambda)=1$ and $\langle\nu,\alpha_j^\vee\rangle=2$  (so $a_j(\lambda-\nu)=-1$).  Then
$\nu=\alpha_j$ and $\tilde{\nu}=0$, so
$w_{\lambda-\nu}=s_j$ and $\theta(\lambda-\nu)\stackrel{\text{Eq.}~\eqref{theta-rec}}{=}\theta (\lambda)=0$.
\end{itemize}
It is moreover manifest that the cases $(B)$ and $(C)$ only occur when $\nu\in R_0\cap P_\vartheta=W_0\vartheta$ and
(thus) $\tau_j=\tau_0$.

\subsection{Proof of Lemma~\ref{theta:lem}}
It is sufficient to restrict attention to the case that
$\lambda -\nu\not\in P^+_c$ (as for $\lambda-\nu\in P^+_c$ the lemma is trivial). For a reduced
decomposition
$w_{\lambda-\nu}=s_{j_\ell}\cdots s_{j_1}$ with $\ell =\ell (w_{\lambda-\nu})\geq 1$, we write
$\nu_k:= s'_{j_{k}}\cdots s'_{j_1} \nu$ for $k=0,\ldots ,\ell$ and
$b_k=\beta_k^\vee+ r_k c:= s_{j_1}\cdots s_{j_{k}}a_{j_{k+1}}$ for $k=0,\ldots ,\ell-1$ (with the conventions that
$\nu_0:=\nu$ and $b_0:=a_{j_1}$).
This means that $R(w_{\lambda-\nu})=\{ b_0,\ldots ,b_{\ell-1}\}$ (cf. \cite[(2.2.9)]{mac:affine}).
It is immediate from the observations $(A)$--$(C)$ above that the minimal sequence of weights taking $\lambda-\nu$ to $(\lambda-\nu)_+$ by successive application of the simple reflections
in our reduced decomposition of $w_{\lambda-\nu}$ is either of the form
(situation $i)$):
\begin{equation*}\label{chain1}
\lambda-\nu=\lambda-\nu_0\stackrel{s_{j_1}}{\longrightarrow}\lambda-\nu_1\stackrel{s_{j_2}}{\longrightarrow}\cdots  \stackrel{s_{j_{\ell-1}}}{\longrightarrow}
\lambda-\nu_{\ell-1} \stackrel{s_{j_\ell}}{\longrightarrow} \lambda-\nu_\ell =(\lambda-\nu)_+ ,
\end{equation*}
or of the form (situation $ii)$):
\begin{equation*}\label{chain2}
\lambda-\nu=\lambda-\nu_0\stackrel{s_{j_1}}{\longrightarrow}\lambda-\nu_1\stackrel{s_{j_2}}{\longrightarrow}\cdots  \stackrel{s_{j_{\ell-1}}}{\longrightarrow}
\lambda-\nu_{\ell-1} \stackrel{s_{j_\ell}}{\longrightarrow} \lambda =(\lambda-\nu)_+ ,
\end{equation*}
because case $(C)$ can at most occur
at the last step:  $\lambda-\nu_{\ell-1} \stackrel{s_{j_\ell}}{\longrightarrow} \lambda= (\lambda-\nu)_+$ (as
this case takes us back to $P^+_c$).
In situation $i)$ (i.e. case $(C)$ does not occur at the last step) we have that
$w_{\lambda-\nu}\in W_{R,\lambda}$ and $(\lambda-\nu)_+\neq \lambda $, whereas in situation $ii) $ (i.e. case $(C)$ does occur at the last step) we have that
 $s_{j_\ell}w_{\lambda-\nu}=s_{j_{\ell-1}}\cdots s_{j_1}\in W_{R,\lambda}$,
 $(\lambda-\nu)_+=\lambda$, and $\tau_{j_\ell}=\tau_0$.
Moreover, in the latter situation
$s^\prime_{j_{\ell}}w'_{\lambda-\nu}\nu=\nu_{\ell -1}=\alpha_{j_\ell}$, i.e.
 $\nu =(s'_{j_{\ell}}w'_{\lambda-\nu})^{-1}\alpha_{j_\ell}=(s'_{j_1}\cdots s'_{j_{\ell-1}}\alpha_{j_\ell})=\beta_{\ell-1}$, which implies that
$$\langle\lambda,\nu^\vee\rangle+r_{\ell -1}c =b_{\ell-1}(\lambda)=((s_{j_{\ell}}w_{\lambda-\nu})^{-1}a_{j_\ell})(\lambda)
=a_{j_\ell}(s_{j_{\ell-1}}\cdots s_{j_1}\lambda)=a_{j_{\ell}}(\lambda)=1,$$ i.e.
\begin{equation}\label{nuvee-1}
\nu^\vee+(1-\langle \lambda,\nu^\vee\rangle)=b_{\ell -1}\in R(w_{\lambda-\nu})  \setminus R(s_{j_\ell} w_{\lambda-\nu}).
\end{equation}

It remains to compute $\theta (\lambda-\nu )$. Since $\theta ((\lambda-\nu)_+)=0$, it is clear from the observations $(A)$--$(C)$ that $\theta (\lambda-\nu )$ is equal to the number of
 times case $(B)$ occurs in the above sequences, i.e. the number of times that
 $ \langle \nu_k,\alpha^\vee_{j_{k+1}} \rangle =2$ for $k=0,\ldots ,\ell^\prime -1$, where
 $\ell^\prime =\ell$ in situation $i)$ and $\ell^\prime=\ell -1$ in situation $ii)$.
Since for $k=0,\ldots ,\ell^\prime -1$: $$\langle \nu_k,\alpha^\vee_{j_{k+1}} \rangle =2\Leftrightarrow\langle \nu ,\beta^\vee_k \rangle =2\Leftrightarrow\nu=\beta_k,$$ and
$$ \langle\lambda,\beta_{k}^\vee\rangle+r_{k} c=b_{k}(\lambda)=(s_{j_1}\cdots s_{j_{k}}a_{j_{k+1}})(\lambda)= a_{j_{k+1}}(s_{j_{k}}\cdots s_{j_1}\lambda)=a_{j_{k+1}}(\lambda)=0,$$ i.e.
$$
\beta_{k}^\vee-\langle\lambda,\beta_{k}^\vee\rangle =b_{k}\in R(w_{\lambda-\nu}) ,
$$
it follows that $\theta (\lambda-\nu)$ is equal to $1$ or $0$ depending whether $\nu^\vee-\langle\lambda,\nu^\vee\rangle\in R(w_{\lambda-\nu})$ or $\nu^\vee-\langle\lambda,\nu^\vee\rangle\not\in R(w_{\lambda-\nu})$, respectively. In particular, in situation $ii)$ we have that $\theta(\lambda-\nu)=0$,
because by Eq. \eqref{nuvee-1} (and thus $\langle \lambda,\nu^\vee\rangle-1\in c\mathbb{Z}$) the fact that $\nu^\vee-\langle\lambda,\nu^\vee\rangle$ belongs to $ R(w_{\lambda-\nu})$ (and thus
$\langle \lambda,\nu^\vee\rangle\in c\mathbb{Z}$) would imply that $c=1$, which
contradicts our assumption that $c>1$.

\subsection{Proof of Lemma~\ref{Iqm-action:lem}}

The explicit action of $I_j$ implied by Eqs. \eqref{Ia:action}--\eqref{Ja:action} simplifies close to (the positive side of) the hyperplane $V_j$
to
\begin{equation}\label{Ijact}
(I_jf)(\mu )
 =
 \begin{cases}\tau_j f(\mu )&\text{if}\ a_j(\mu) =0 \\
 \tau_j^{-1}f(s_j\mu ) =\tau_j^{-1} f(\mu-\alpha_j) &\text{if}\ a_j(\mu) =1 \\
\tau_j^{-1}f(\mu-2\alpha_j) -(\tau_j-\tau_j^{-1})f(\mu-\alpha_j) &\text{if}\ a_j(\mu) = 2
\end{cases}
\end{equation}
(for $f\in C(P)$, $\mu\in P$, $j=0,\ldots, n$).
With the aid of these formulas
the proof  of the Lemma follows by induction on $\ell (w_{\lambda-\nu})$ starting from the trivial base
$\lambda-\nu\in P^+_c$.

Specifically, let $\ell (w_{\lambda-\nu})>1$ and $s_j$ ($0\leq j\leq n$) such that
$\ell (w_{\lambda-\nu}s_j)=\ell (w_{\lambda-\nu})-1$ (i.e. $a_j\in R(w_{\lambda-\nu})$).
From the observations before the proof of Lemma \ref{theta:lem} it is clear that
$w_{\lambda-\nu}s_j=w_{s_j(\lambda-\nu)}$  with either
$s_j(\lambda -\nu)=\lambda-s'_j\nu$ (cases $(A)$ and $(B)$) or
$s_j(\lambda -\nu)=\lambda (\in P^+_c)$  (case $(C)$). In the latter situation $w_{\lambda-\nu}=s_j$ and the statement
of the lemma reduces to the second line of  Eq. \eqref{Ijact} (with $\mu=\lambda$).
Moreover, in the cases $(A)$ and $(B)$  invoking of the induction hypothesis yields
\begin{align*}
\tau_{w_{\lambda-\nu}}(I_{w_{\lambda-\nu}} f)((\lambda-\nu)_+)&=
\tau_j \tau_{w_{\lambda-s'_j\nu}} (I_{w_{\lambda-s'_j\nu}} I_jf)((\lambda-s'_j\nu)_+)  \\  &=
\tau_j(I_j f) (\lambda-s'_j\nu) - \tau_j c_{\lambda,s'_j\nu} (1-\tau_0^{-2})(I_jf)(\lambda)
\end{align*}
(where we have used that $(\lambda-s'_j\nu)_+=(\lambda-\nu)_+$).

In case $(A)$, one has that $\tau_j(I_j f) (\lambda-s'_j\nu) =f(\lambda -\nu)$ (by the second line of Eq. \eqref{Ijact}
with $\mu=\lambda-s'_j\nu$) and $(I_jf)(\lambda )=\tau_j f(\lambda )$ (by the first line of Eq. \eqref{Ijact}
with $\mu=\lambda$), which completes the induction step for this situation as now
$c_{\lambda ,s_j^\prime\nu}=c_{\lambda ,\nu}\tau_j^{-2} $.
Indeed, clearly
$\theta (\lambda-s'_j\nu)=\theta (\lambda-\nu)$ with for $j>0$: $e^\vee_\tau(s_j\nu)=e^\vee_\tau(\nu)\tau_j^{-2}$
and $\langle \lambda,s_j \nu^\vee \rangle=
\langle s_j\lambda,\nu^\vee \rangle=\langle \lambda,\nu^\vee \rangle$, whereas for $j=0$:
$e^\vee_\tau(s'_0\nu)=e^\vee_\tau(\nu-\alpha_0)=e^\vee_\tau(\nu+\vartheta)=e^\vee_\tau(\nu)h^\vee_\tau \tau_0^{-2}$
and---assuming $\theta (\lambda -\nu)>0$---$\langle \lambda,s'_0 \nu^\vee \rangle=
\langle s'_0\lambda,\nu^\vee \rangle=\langle \lambda+c\alpha_0,\nu^\vee \rangle =
\langle \lambda,\nu^\vee \rangle+c\langle \alpha_0,\nu^\vee \rangle=\langle \lambda,\nu^\vee \rangle+c$ (upon recalling that
$\nu\in W_0\vartheta$
when $\theta (\lambda -\nu)>0$ and thus $\langle \alpha_0,\nu^\vee \rangle=\langle \alpha_0^\vee,\nu \rangle=1$), i.e.
$\text{sign} (\langle \lambda,s_0^\prime \nu\rangle)=\text{sign} (\langle \lambda, \nu\rangle)+1$ (cf. Eq. \eqref{theta:range}).

In case $(B)$, one has that $\tau_j(I_j f) (\lambda-s'_j\nu) =f(\lambda-\nu)-\tau_0^2(1-\tau_0^{-2})f(\lambda )$ (by the third line of Eq. \eqref{Ijact} with $\mu=\lambda-s'_j\nu$ and the fact that
$\tau_j=\tau_0$) and
$c_{\lambda,s'_j\nu}=0$ (since $0\leq\theta (\lambda-s'_j\nu )<\theta (\lambda-\nu)\leq 1$), which completes the induction step for this situation as now
$c_{\lambda ,\nu}=\tau_0^2$.
Indeed, clearly $\theta (\lambda-\nu)=1$ with for $j>0$:
$e_\tau^\vee(\nu)=e_\tau^\vee(\alpha_j)=\tau_j^2=\tau_0^2$ (as $e_\tau^\vee (\alpha_j)=  e_\tau^\vee(s_j\alpha_j)\tau_j^{2\langle \alpha_j,\alpha_j^\vee\rangle}
= e_\tau^\vee (-\alpha_j)\tau_j^4=\tau_j^4/e_\tau^\vee (\alpha_j)$) and $\langle \lambda ,\nu^\vee\rangle=\langle \lambda,\alpha^\vee_j\rangle=0$, whereas for $j=0$: $e_\tau^\vee (\nu)=e_\tau^\vee (\alpha_0)=e_\tau^\vee (-\vartheta)=\tau_0^2/h^\vee_\tau$ and $\langle \lambda,\nu^\vee\rangle=\langle \lambda,\alpha_0^\vee\rangle= a_0(\lambda)-c=-c<0$.

\subsection{Proof of Eq.~\eqref{intertwining-property}}

We are now in a position to verify Eq. \eqref{intertwining-property} by making the action of the intertwining operator  on $f\in\mathcal{C}(P)$ explicit:
\begin{equation}\label{aJact}
a_{\lambda,\nu}(\mathcal{J}f) (\lambda-\nu)
\stackrel{\text{Eq.}~\eqref{stable}}{=} a_{\lambda,\nu}
\tau_{w_{w_\lambda (\lambda-\nu)}w_\lambda}^{-1}(I_{w_{w_\lambda (\lambda-\nu)}w_\lambda}f)((\lambda-\nu)_+)  .
\end{equation}
Throughout it will be used that $w_\lambda (\lambda-\nu)=\lambda_+-w_\lambda^\prime\nu$.

For $(\lambda-\nu)_+\neq\lambda_+$,
Lemma \ref{theta:lem} (with $\lambda$ and $\nu$ replaced by
$\lambda_+$ and $w'_\lambda\nu$) ensures that $w_{w_\lambda (\lambda-\nu)}\in W_{R,\lambda_+}$, whence
$\ell (w_{w_\lambda (\lambda-\nu)}w_\lambda)=\ell (w_{w_\lambda (\lambda-\nu)})+\ell (w_\lambda)$
and we may rewrite the expression in question as:
\begin{align*}
\tau_{w_\lambda} ^{-1} &\tau_{w_{w_\lambda (\lambda-\nu)}} (I_{w_{w_\lambda (\lambda-\nu)}} I_{w_\lambda}f)((\lambda-\nu)_+)
  \\
   \stackrel{\text{Lem.}~\ref{Iqm-action:lem}}{=} &\tau_{w_\lambda} ^{-1}\left(
(I_{w_\lambda}f)(w_\lambda (\lambda -\nu) )
 - c_{\lambda_+,w'_\lambda \nu} (1-\tau_0^{-2})
(I_{w_\lambda} f)(\lambda_+)
\right)  \\
 =\ \ \ \, &\tau_{w_\lambda} ^{-1}
(I_{w_\lambda}f)(w_\lambda (\lambda -\nu) )-b_{\lambda,\nu}(1-\tau_0^{-2})(\mathcal{J}f)(\lambda) ,
\end{align*}
which proves Eq. \eqref{intertwining-property} when $(\lambda-\nu)_+\neq\lambda_+$.

Similarly, for $(\lambda-\nu)_+= \lambda_+$ we rewrite the RHS of Eq. \eqref{aJact} as:
\begin{align*}
& a_{\lambda ,\nu} \tau_{w_{w_\lambda (\lambda-\nu)}w_\lambda}^{-1}(I_{w_{w_\lambda (\lambda-\nu)}w_\lambda}f)(\lambda_+)
   \\
& \stackrel{\text{(i)}}{=}     \tau_{w_\lambda}^{-1}\tau_{w_{w_\lambda (\lambda-\nu)}}\left(
(I_jI_{s_jw_{w_\lambda (\lambda-\nu)} w_\lambda}f)(\lambda_+)  \right. \\
&\qquad \qquad \ \ \ \
  \left.  -\chi((s_jw_{w_\lambda (\lambda-\nu)} w_\lambda)^{-1}a_j)  (\tau_j-\tau_j^{-1})
(I_{s_jw_{w_\lambda (\lambda-\nu)}w_\lambda} f)(\lambda_+)
\right)   \\
&\stackrel{\text{(ii)}}{=}  \tau_{w_\lambda} ^{-1}\Bigl(
(I_{w_\lambda}f)(w_\lambda (\lambda -\nu) )  \\
&\qquad \qquad \ \ \ \
  \left. -\tau_{w_{w_\lambda (\lambda-\nu)}}^2 \chi(\nu^\vee+(1-\langle \lambda,\nu^\vee\rangle))(1-\tau_0^{-2})
(I_{w_\lambda} f)(\lambda_+)
\right)  \\
&= \tau_{w_\lambda} ^{-1}
(I_{w_\lambda}f)(w_\lambda (\lambda -\nu) )-b_{\lambda,\nu}(1-\tau_0^{-2})(\mathcal{J}f)(\lambda) ,
\end{align*}
which proves Eq. \eqref{intertwining-property} when $(\lambda-\nu)_+=\lambda_+$.
Here the equality
$\text{(i)}$ hinges on Eq.~\eqref{TjTw} (for any $j\in \{0,\ldots ,n\}$), whereas for inferring
equality $\text{(ii)}$ we
pick $j$ as in part $ii)$ of Lemma \ref{theta:lem} (with $\lambda$ and $\nu$ replaced by
$\lambda_+$ and $w'_\lambda\nu$). Then
$s_jw_{w_\lambda (\lambda-\nu)}\in W_{R,\lambda_+}$ (so $\ell (s_jw_{w_\lambda (\lambda-\nu)}w_\lambda)=\ell (s_jw_{w_\lambda (\lambda-\nu)})+\ell (w_\lambda )$) and
$(s_jw_{w_\lambda (\lambda-\nu)})^{-1}a_j\in R (w_{w_\lambda (\lambda-\nu)})$
(so $\ell (w_{w_\lambda (\lambda-\nu)})=\ell (s_jw_{w_\lambda (\lambda-\nu)})+1$). Hence, in this situation
the first term on the RHS of equality $\text{(i)}$ may be rewritten as:
\begin{subequations}
\begin{eqnarray}
\tau_{w_{w_\lambda (\lambda-\nu)} }(I_jI_{s_jw_{w_\lambda (\lambda-\nu)} w_\lambda}f)(\lambda_+)&=& \tau_{w_{w_\lambda (\lambda-\nu)} }(I_{w_{w_\lambda (\lambda-\nu)} }I_{ w_\lambda}f)((\lambda-\nu)_+) \nonumber\\
& \stackrel{\text{Lem.}~\ref{Iqm-action:lem}}{=}& (I_{w_\lambda}f)(w_\lambda (\lambda -\nu) ) , \label{rw1}
\end{eqnarray}
where it was used that $c_{\lambda_+,w^\prime_\lambda\nu}=0$ as
$\theta (\lambda_+-w^\prime_\lambda\nu)=0$ (by Lemma \ref{theta:lem}). For the second term one deduces in a similar way that:
\begin{eqnarray}
\lefteqn{(I_{s_jw_{w_\lambda (\lambda-\nu)}w_\lambda} f)(\lambda_+) = (I_{s_jw_{w_\lambda (\lambda-\nu)}}I_{w_\lambda} f)(\lambda_+)
\stackrel{\text{Eq.}~\eqref{stable}}{=}} && \nonumber \\
&&\tau_{s_jw_{w_\lambda (\lambda-\nu)}} (I_{w_\lambda} f)(\lambda_+)=\tau_{w_{w_\lambda (\lambda-\nu)}}\tau_j^{-1} (I_{w_\lambda} f)(\lambda_+) . \label{rw2}
\end{eqnarray}
\end{subequations}
Performing the substitutions \eqref{rw1} and \eqref{rw2} on the RHS of equality $\text{(i)}$ gives rise to equality $\text{(ii)}$.
Here one uses in addition that---for this particular choice of $j$---Lemma \ref{theta:lem} guarantees that
$\tau_j=\tau_0$ and
\begin{equation}\label{wd}
(s_jw_{w_\lambda (\lambda-\nu)} w_\lambda)^{-1}a_j=
w_\lambda^{-1}\bigl( w_\lambda^\prime\nu^\vee+(1-\langle \lambda_+,w_\lambda^\prime\nu^\vee\rangle)\bigr)
=\nu^\vee+(1-\langle \lambda,\nu^\vee\rangle) .
\end{equation}

\vspace{2ex}
\begin{remark}\label{root:rem}
It is read-off from Eq. \eqref{wd} that for any $\lambda\in P$ and $\nu\in P_\vartheta^\star$
such that $(\lambda-\nu)_+=\lambda_+$, one has that
$\nu^\vee+(1-\langle \lambda,\nu^\vee\rangle)\in R$ (and thus $\nu\in R_0\cap P_\vartheta = W_0\vartheta$ and $\langle \lambda,\nu^\vee\rangle -1\in  c\mathbb{Z}$).
\end{remark}

\section{The adjoint of $L_\omega$}\label{appB}

In this appendix it is shown that
for all $\lambda\in P$ and $\nu\in P_\vartheta^\star$:
\begin{equation}\label{sym-rel}
\tau_{w_{w_\lambda(\lambda-\nu)}w_\lambda}\tau_{w_{w_\lambda(\lambda-\nu)}}\tau_{w_\lambda}
=\tau_{w_{w_{\lambda-\nu}(\lambda)}w_{\lambda-\nu}}\tau_{w_{w_{\lambda-\nu}(\lambda)}}\tau_{w_{\lambda-\nu}} .
\end{equation}
This relation for the length multiplicative function $\tau$ lies at the basis of the unitarity in Theorem \ref{unitary-Z:thm} and the consequent (self-)adjointness relation in Corollary \ref{unitary-Z:cor}.
The proof of Eq. \eqref{sym-rel} hinges on two lemmas elucidating the geometric interpretation of the group element $w_{w_\lambda(\lambda-\nu)}w_\lambda$
for $(\lambda-\nu)_+\neq \lambda_+$ (i.e. $\lambda$ and $\lambda-\nu$ lie on the same closed alcove $wA_c$, $w\in W_R$) and for $(\lambda-\nu)_+= \lambda_+$ (i.e. $\lambda$ and $\lambda-\nu$ are separated by a (unique) wall $V_a$, $a\in R^+$), respectively.

\begin{lemma}\label{same-alcove:lem} Let $\lambda\in P$ and $\nu\in P_\vartheta^\star$ such that $(\lambda-\nu)_+\neq \lambda_+$. Then $w_{w_\lambda(\lambda-\nu)}w_\lambda$  is the {\em unique} element in $W_R$ of minimal length mapping both $\lambda-\nu$ and $\lambda$ to $P^+_c$, i.e. $w_{w_\lambda(\lambda-\nu)}\in W_{R;\lambda_+}$ and
\begin{equation}\label{translation-inv}
w_{w_{\lambda-\nu}(\lambda)}w_{\lambda-\nu}=w_{w_\lambda(\lambda-\nu)}w_\lambda .
\end{equation}
\end{lemma}

\begin{lemma}\label{neighbor-alcove:lem} Let $\lambda\in P$ and $\nu\in P_\vartheta^\star$ such that $(\lambda-\nu)_+= \lambda_+$. Then
\begin{subequations}
\begin{equation}\label{simple-wall}
-( w_{w_\lambda(\lambda-\nu)}w_\lambda)^\prime \nu=\alpha_j=(w_{w_{\lambda-\nu}(\lambda)}w_{\lambda-\nu})^\prime\nu
\end{equation}
for some $j\in\{0,\dots, n\}$ with $\tau_j=\tau_0$, and
 $w_{w_\lambda(\lambda-\nu)}w_\lambda$ is the {\em unique} element in $W_R$ of minimal length mapping $\lambda-\nu$ to $\lambda_+$ and $\lambda$ to $s_j\lambda_+$, i.e.
 $s_j w_{w_\lambda (\lambda-\nu)}\in W_{R,\lambda_+}$  and
 \begin{equation}\label{translation-cov}
 w_{w_{\lambda-\nu}(\lambda)}w_{\lambda-\nu}=s_jw_{w_\lambda(\lambda-\nu)}w_\lambda
 \end{equation}
 \end{subequations}
with
$\ell(w_{w_\lambda(\lambda-\nu)})=\ell(s_jw_{w_\lambda(\lambda-\nu)})+1$ and
$\ell(w_{w_{\lambda-\nu}(\lambda)})=\ell(s_jw_{w_{\lambda-\nu}(\lambda)})+1$.
\end{lemma}

For $(\lambda-\nu)_+\neq \lambda_+$ the relation in Eq. \eqref{sym-rel} is manifest from Lemma \ref{same-alcove:lem} and the symmetry in $\lambda$ and $\lambda-\nu$, which entail that in this situation both sides simplify to $\tau_{w_{w_\lambda(\lambda-\nu)}w_\lambda}^2=\tau_{w_{w_{\lambda-\nu}(\lambda)}w_{\lambda-\nu}}^2$.
For $(\lambda-\nu)_+=\lambda_+$ the relation follows in turn from Lemma \ref{neighbor-alcove:lem}.
Indeed, in the latter situation it is readily seen that
$$
\tau_{w_{w_\lambda(\lambda-\nu)}}\tau_{w_\lambda}
=\tau_0 \tau_{s_jw_{w_\lambda(\lambda-\nu)}}\tau_{w_\lambda}
=\tau_0\tau_{s_jw_{w_\lambda(\lambda-\nu)}w_\lambda}
=\tau_0 \tau_{w_{w_{\lambda-\nu}(\lambda)}w_{\lambda-\nu}}
$$
and (upon interchanging the role of $\lambda$ and $\lambda-\nu$)
$$
\tau_{w_{w_{\lambda-\nu}(\lambda)}}\tau_{w_{\lambda-\nu}}
=\tau_0\tau_{s_jw_{w_{\lambda-\nu}(\lambda)}}\tau_{w_{\lambda-\nu}}
=\tau_0\tau_{s_jw_{w_{\lambda-\nu}(\lambda)}w_{\lambda-\nu}}
=\tau_0\tau_{w_{w_\lambda(\lambda-\nu)}w_\lambda} ,
$$
whence both sides of Eq. \eqref{sym-rel} now simplify to $ \tau_0 \tau_{w_{w_\lambda(\lambda-\nu)}w_\lambda} \tau_{w_{w_{\lambda-\nu}(\lambda)}w_{\lambda-\nu}}$.

\subsection{Proof of Lemma~\ref{same-alcove:lem}}
Clearly $w_{w_\lambda(\lambda-\nu)}w_\lambda (\lambda-\nu)=(\lambda-\nu)_+$ and---by
Lemma \ref{theta:lem}  (with $\lambda$ and $\nu$ replaced by $\lambda_+$ and $w'_\lambda\nu$)---it furthermore follows
that $w_{w_\lambda(\lambda-\nu)}\in W_{R,\lambda_+}$, i.e.
$w_{w_\lambda(\lambda-\nu)}w_\lambda(\lambda)=\lambda_+$ and
$\ell(w_{w_\lambda(\lambda-\nu)}w_\lambda)=\ell(w_{w_\lambda(\lambda-\nu)})+\ell(w_\lambda)$.
Let $w$ now denote {\em any} element in $W_R$ of {\em minimal length} sending $\lambda$ and $\lambda-\nu$ to $P^+_c$.
Then $ww_\lambda^{-1}\in W_{R,\lambda_+}$ (so $\ell (w)=\ell (ww_\lambda^{-1})+\ell (w_\lambda)$) and
$ww_\lambda^{-1} (w_\lambda (\lambda-\nu))=(\lambda-\nu)_+$ (so
$\ell (ww_\lambda^{-1})\geq \ell (w_{w_\lambda (\lambda-\nu)})$), i.e.
$$\ell (w)=\ell (ww_\lambda^{-1})+\ell (w_\lambda )\geq \ell (w_{w_\lambda (\lambda-\nu)})+\ell (w_\lambda)=\ell (w_{w_\lambda (\lambda-\nu)} w_\lambda).$$ It thus follows that $\ell (w)=\ell (w_{w_\lambda (\lambda-\nu)} w_\lambda)$
(and therefore $\ell (ww_\lambda^{-1})=\ell (w_{w_\lambda (\lambda-\nu)})$) by the minimality of $\ell (w)$, and consequently $ww_\lambda^{-1}=w_{w_\lambda (\lambda-\nu)}$
(by the uniqueness of $w_{w_\lambda (\lambda-\nu)}$), i.e. $w=w_{w_\lambda(\lambda-\nu)}w_\lambda$.
The equality in Eq. \eqref{translation-inv} now follows
upon interchanging the roles of $\lambda$ and $\lambda-\nu$.

\subsection{Proof of Lemma~\ref{neighbor-alcove:lem}}
For $j\in \{ 0,\ldots ,n\}$ as in Lemma \ref{theta:lem} (with $\lambda$ and $\nu$ replaced by
$\lambda_+$ and $w'_\lambda\nu$) we have that $s_jw_{w_\lambda (\lambda-\nu)}\in W_{R,\lambda_+}$, whence
$w_{w_\lambda(\lambda-\nu)}w_\lambda$ maps $\lambda-\nu$ to $\lambda_+$ and $\lambda$ to $s_j\lambda_+$. In addition
$( w_{w_\lambda(\lambda-\nu)}w_\lambda)^\prime \nu=w_{\lambda-w_\lambda^\prime\nu}^\prime w_\lambda^\prime\nu=-\alpha_j$ and
$\ell(w_{w_\lambda(\lambda-\nu)})=\ell(s_jw_{w_\lambda(\lambda-\nu)})+1$ (as
$R(w_{w_\lambda(\lambda-\nu)})$ is the disjoint union of $R(s_jw_{w_\lambda(\lambda-\nu)})$
and $\{ (s_jw_{w_\lambda(\lambda-\nu)})^{-1} a_j  \}$).

Let $w$ now be {\em any} element in $W_R$ of {\em minimal length}
satisfying that $w(\lambda-\nu)=\lambda_+$ and $w\lambda=s_j\lambda_+$.
Then $s_jww_\lambda^{-1}\in W_{R,\lambda_+}$ (so $\ell (s_jw)=\ell (s_jww_\lambda^{-1})+\ell (w_\lambda)$) and
$ww_\lambda^{-1} (w_\lambda (\lambda-\nu))=\lambda_+$ (so
$\ell (ww_\lambda^{-1})\geq \ell (w_{w_\lambda (\lambda-\nu)})$), i.e.
\begin{align}
&\ell (w)\stackrel{\text{(i)}}{=}\ell (s_jw)-\text{sign}(w^{-1}a_j)=
\ell (s_jww_\lambda^{-1})+\ell (w_\lambda)-\text{sign}(w^{-1}a_j)
\nonumber \\ & \geq
\ell (ww_\lambda^{-1})+\ell (w_\lambda)-\text{sign}(w^{-1}a_j)-1
\geq
\ell (w_{w_\lambda (\lambda-\nu)})+\ell (w_\lambda)-\text{sign}(w^{-1}a_j)-1 \nonumber \\
&=
\ell (s_jw_{w_\lambda (\lambda-\nu)})+\ell (w_\lambda)-\text{sign}(w^{-1}a_j)
=
\ell (s_jw_{w_\lambda (\lambda-\nu)}w_\lambda)-\text{sign}(w^{-1}a_j) \nonumber\\
&\stackrel{\text{(ii)}}{=}
\ell (s_jw_{w_\lambda (\lambda-\nu)}w_\lambda)-\text{sign}((w_{w_\lambda (\lambda-\nu)}w_\lambda)^{-1}a_j)
\stackrel{\text{(iii)}}{=}\ell (w_{w_\lambda (\lambda-\nu)} w_\lambda). \label{lw-est}
\end{align}
In steps $\text{(i)}$ and $\text{(iii)}$ we employed Eq.~\eqref{ls1}, and in step $\text{(ii)}$
it was used that
$w^{-1}a_j=(w_{w_\lambda (\lambda-\nu)}w_\lambda)^{-1}a_j$. Indeed, one has that
$$w^{-1}V_j=(w_{w_\lambda (\lambda-\nu)}w_\lambda)^{-1}V_j=V_a,$$
where $V_a$, $a\in R^+$ stands for the root hyperplane consisting of all points equidistant to $\lambda$ and $\lambda-\nu$ (which is in fact the unique root hyperplane separating $\lambda$ and $\lambda-\nu$),
and furthermore that $(w^\prime)^{-1}\alpha_j= (w^\prime)^{-1}(\lambda_+ -s_j\lambda_+)=w^{-1}\lambda_+-w^{-1}s_j\lambda_+=(\lambda-\nu)-\lambda=-\nu=
(w_{w_\lambda (\lambda-\nu)}^\prime w_\lambda^\prime)^{-1}\alpha_j.$
By the minimality of $\ell (w)$, one concludes that all inequalities in Eq. \eqref{lw-est} must in fact be equalities. In particular, one has that $\ell (w)=\ell (w_{w_\lambda (\lambda-\nu)} w_\lambda)$
and $\ell (ww_\lambda^{-1})=\ell (w_{w_\lambda (\lambda-\nu)})$. But then  $ww_\lambda^{-1}=w_{w_\lambda (\lambda-\nu)}$
(by the uniqueness of $w_{w_\lambda (\lambda-\nu)}$), i.e. $w=w_{w_\lambda(\lambda-\nu)}w_\lambda$.

The second equality in Eq. \eqref{simple-wall} and the equality in Eq. \eqref{translation-cov} now follow
upon interchanging the roles of $\lambda$ and $\lambda-\nu$ (which amounts to replacing $\lambda$ by $\tilde{\lambda}:=\lambda-\nu$ and $\nu$ by $\tilde{\nu}:=-\nu$).
More specifically, by proceeding as before, application of Lemma \ref{theta:lem} (with $\lambda$ and $\nu$ replaced by
$\tilde{\lambda}_+=\lambda_+$ and $w'_{\tilde{\lambda}}\tilde{\nu}$) entails that
$w^\prime_{w_{\lambda-\nu}(\lambda)}w^\prime_{\lambda-\nu}\nu=
-w^\prime_{w_{\tilde{\lambda}}(\tilde{\lambda}-\tilde{\nu})}w^\prime_{\tilde{\lambda}}\tilde{\nu}=\alpha_k$ for certain $k\in\{0,\ldots ,n\}$, and furthermore
$w_{w_{\lambda-\nu}(\lambda)}w_{\lambda-\nu}=w_{w_{\tilde{\lambda}}(\tilde{\lambda}-\tilde{\nu})}w_{\tilde{\lambda}}$ maps $\tilde{\lambda}-\tilde{\nu}=\lambda$ to $\lambda_+$ and $\tilde{\lambda}=\lambda-\nu$ to $s_k\lambda_+$.
Since both the line segment connecting $\lambda_+$ and $s_j\lambda_+$ and the line segment connecting $\lambda_+$ and $s_k\lambda_+$ belong to the $W_R$-orbit of the line segment
connecting $\lambda$ and $\lambda-\nu$, we have that $k=j$ and the second equality in Eq. \eqref{simple-wall} follows. Here we have used the fact that if $\text{Conv}\{ \lambda_+ ,s_k\lambda_+\}=w\text{Conv}\{ \lambda_+,s_j\lambda_+\}$ for some $w\in W_R$ and some $j,k\in \{0,\ldots ,n\}$ such that $s_j\lambda_+\neq\lambda_+$ and $s_k\lambda_+\neq\lambda_+$---where we assume without loss of generality that $w\in W_{R,\lambda_+}$ (since otherwise we replace $w$ by $ws_j$)---then $w\in W_{R,s_j\lambda_+}$ and thus $k=j$. Indeed $\text{Conv}\{ \lambda_+,(\lambda_+ +s_j\lambda_+)/2\}\subset A_c$ and $w\text{Conv}\{ \lambda_+,(\lambda_+ +s_j\lambda_+)/2\}=\text{Conv}\{ \lambda_+,(\lambda_+ +s_k\lambda_+)/2\}\subset A_c$, so $w\text{Conv}\{ \lambda_+,(\lambda_+ +s_j\lambda_+)/2\}=\text{Conv}\{ \lambda_+,(\lambda_+ +s_j\lambda_+)/2\}$, i.e. $ws_j\lambda_+=s_j\lambda_+$.

The upshot is that both sides of Eq. \eqref{translation-cov} map
$\lambda $ to $\lambda_+$ and $\lambda-\nu$ to $s_j\lambda_+$.
To see that the group elements at issue are in effect the same it therefore only remains to verify that
their lengths are equal (by the uniqueness of the minimal element $w_{w_{\lambda-\nu}(\lambda)}w_{\lambda-\nu}$).
We exploit Eq. \eqref{ls1} and the minimality of $w_{w_\lambda(\lambda-\nu)}w_\lambda$ and $w_{w_{\lambda-\nu}(\lambda)}w_{\lambda-\nu}$ to arrive at the following estimates:
\begin{align}
&\ell (w_{w_{\lambda-\nu}(\lambda)}w_{\lambda-\nu})\leq \ell (s_jw_{w_\lambda(\lambda-\nu)}w_\lambda)
=
\ell(w_{w_\lambda(\lambda-\nu)}w_\lambda)+\text{sign}((w_{w_\lambda(\lambda-\nu)}w_\lambda)^{-1}a_j) \nonumber\\
& \leq
\ell(s_jw_{w_{\lambda-\nu}(\lambda)}w_{\lambda-\nu})+\text{sign}((w_{w_\lambda(\lambda-\nu)}w_\lambda)^{-1}a_j)
\nonumber \\
&=
\ell(w_{w_{\lambda-\nu}(\lambda)}w_{\lambda-\nu})+
\text{sign}((w_{w_{\lambda-\nu}(\lambda)}w_{\lambda-\nu})^{-1}a_j)+
\text{sign}((w_{w_\lambda(\lambda-\nu)}w_\lambda)^{-1}a_j) \nonumber \\
&=\ell (w_{w_{\lambda-\nu}(\lambda)}w_{\lambda-\nu}), \label{lest}
\end{align}
where we used in the last equality that (cf. step $\text{(ii)}$ above):
$(w_{w_{\lambda-\nu}(\lambda)}w_{\lambda-\nu})^{-1}a_j=(s_jw_{w_\lambda(\lambda-\nu)}w_\lambda )^{-1}a_j=-(w_{w_\lambda(\lambda-\nu)}w_\lambda )^{-1}a_j.$ Hence, all inequalities in Eq. \eqref{lest} are again equalities, and in particular $\ell (w_{w_{\lambda-\nu}(\lambda)}w_{\lambda-\nu})=\ell (s_jw_{w_\lambda(\lambda-\nu)}w_\lambda)$, i.e. $w_{w_{\lambda-\nu}(\lambda)}w_{\lambda-\nu}= s_jw_{w_\lambda(\lambda-\nu)}w_\lambda$.

\bibliographystyle{amsplain}

\end{document}